\documentclass[10pt]{amsart}

\usepackage[dvipsnames]{xcolor}
\usepackage{changepage}
\usepackage{graphicx}
\usepackage{hyperref}
\usepackage{dsfont}
\usepackage{xfrac,faktor}
\usepackage{amsmath}
\usepackage{yfonts}
\usepackage{color}

\usepackage[margin=.5cm]{caption}
\usepackage{subcaption}
\usepackage[T1]{fontenc}
\usepackage{mathtools}
\usepackage{amsfonts,amssymb,amscd,amsthm}
\usepackage{tikz}
\usepackage{comment}
\usepackage{float,url}
\usepackage{tikz-cd}
\usepackage[all]{xy}

\newlength{\mywidth}

\calclayout

\newtheorem{thm}{Theorem}[section]
\newtheorem{prop}[thm]{Proposition}
\newtheorem{lem}[thm]{Lemma}
\newtheorem{cor}[thm]{Corollary}

\newtheorem{principle}{Principle}
\newtheorem*{cor*}{Corollary}

\makeatletter
\DeclareFontFamily{U}{tipa}{}
\DeclareFontShape{U}{tipa}{m}{n}{<->tipa10}{}
\newcommand{\arc@char}{{\usefont{U}{tipa}{m}{n}\symbol{62}}}

\newcommand{\arc}[1]{\mathpalette\arc@arc{#1}}

\newcommand{\arc@arc}[2]{
  \sbox0{$\m@th#1#2$}
  \vbox{
    \hbox{\resizebox{\wd0}{\height}{\arc@char}}
    \nointerlineskip
    \box0
  }
}
\makeatother

\theoremstyle{definition}
\newtheorem{definition}[thm]{Definition}
\newtheorem{example}[thm]{Example}

\newtheorem{thmx}{Theorem}

\makeatletter
\renewcommand*\env@matrix[1][\arraystretch]{
  \edef\arraystretch{#1}
  \hskip -\arraycolsep
  \let\@ifnextchar\new@ifnextchar
  \array{*\c@MaxMatrixCols c}}
\makeatother

\theoremstyle{remark}
\newtheorem{rem}{Remark}

\theoremstyle{question}

\newcommand{\R}{\mathbb{R}}  
\newcommand{\D}{\mathbb{D}}
\newcommand{\Q}{\mathbb{Q}}
\newcommand{\C}{\mathbb{C}} 
\newcommand{\Z}{\mathbb{Z}}  
\newcommand{\N}{\mathbb{N}}

\DeclareMathOperator{\interior}{int}

\DeclareMathOperator{\re}{Re}
\DeclareMathOperator{\im}{Im}

\tikzset{
  mynode/.style={fill,circle,inner sep=1pt,outer sep=0pt}
}

\begin{document}

\title[Univalent Polynomials and Hubbard Trees]{Univalent Polynomials and Hubbard Trees}

\author[K. Lazebnik]{Kirill Lazebnik}
\address{Department of Mathematics, University of Toronto, Toronto, Canada}
\email{kylazebnik@gmail.com}

\author[N. G. Makarov]{Nikolai G. Makarov}
\address{Department of Mathematics, California Institute of Technology, Pasadena, California 91125, USA}
\email{makarov@caltech.edu}

\author[S. Mukherjee]{Sabyasachi Mukherjee}
\address{School of Mathematics, Tata Institute of Fundamental Research, 1 Homi Bhabha Road, Mumbai 400005, India}
\email{sabya@math.tifr.res.in}

\date{\today}

\maketitle

\begin{abstract} 

We study rational functions $f$ of degree $d+1$ such that $f$ is univalent in the exterior unit disc, and the image of the unit circle under $f$ has the maximal number of cusps ($d+1$) and double points $(d-2)$. We introduce a bi-angled tree associated to any such $f$. It is proven that any bi-angled tree is realizable by such an $f$, and moreover, $f$ is essentially uniquely determined by its associated bi-angled tree. This combinatorial classification is used to show that such $f$ are in natural 1:1 correspondence with anti-holomorphic polynomials of degree $d$ with $d-1$ distinct, fixed critical points (classified by their Hubbard trees). 

\end{abstract}

\setcounter{tocdepth}{1}

\tableofcontents

\section{Introduction}
\label{introduction}

\subsection{Background}

Let $p(z)$, $q(z)$ be complex analytic polynomials of degrees $n$, $m$, respectively, where $m< n$. In \cite{MR1443416} it was proven that there are at most $n^2$ zeros of the harmonic polynomial $p(z)-\overline{q(z)}$. While this upper bound is sharp when $m=n-1$, it was conjectured in \cite{MR1443416} that when $m=1$ (i.e., $q(z)=z$), an upper bound is given by $3n-2$. We refer to this conjecture as \emph{Wilmshurst's conjecture}.

Wilmshurst's conjecture was proven for $n=3$ in joint work of Crofoot and Sarason described in the letter \cite{Sarason_written_comm}. Therein, it was observed that if one could prove that $p(z)$ has at most $n-1$ non-repelling fixed points, then the $3n-2$ bound would follow by an argument principle for harmonic functions. The $n-1$ bound, in turn, was shown to be equivalent to a norm inequality involving a certain linear operator acting on $\mathbb{C}^n$, which is proven in \cite{Sarason_written_comm} for $n=3$, but not $n>3$.

A closely related conjecture, which we will refer to as \emph{Crofoot--Sarason's conjecture}, was given in \cite{Sarason_written_comm}: \emph{for each $n>1$, there exists a polynomial $p(z)$ of degree $n$ with $n-1$ distinct critical points, each of which is fixed by $z\mapsto\overline{p(z)}$.} Such polynomials are termed \emph{Crofoot--Sarason polynomials} in the present work. Their existence implies sharpness of the upper bound in Wilmshurst's conjecture (see Remark \ref{lefschetz_count_rem}). In degree $n=3$, an example is given in \cite{Sarason_written_comm} by the polynomial \[ p(z)=\frac{3z-z^3}{2}\textrm{ with } \textrm{critical points } \pm1 \textrm{, both fixed by } z\mapsto\overline{p(z)}. \]  The points $0\textrm{, } \frac{1}{2}(\pm\sqrt{7}\pm i)$ are also fixed by $\overline{p}$, yielding a total of $7=3\cdot3-2$ fixed points of $\overline{p}$ and thus verifying sharpness of Wilmshurst's conjecture for $n=3$.  Some further explicit examples of Crofoot--Sarason polynomials in degrees $4, 5, 6, 8$ were given in \cite{MR2081663}.

Wilmshurst's conjecture was proven for all $n>3$  in \cite{KhSw} using complex-dynamical methods (see also \cite{10.2307/26315487} for an overview of more recent work in this area). Crofoot--Sarason's conjecture was later proven in \cite{Ge} via Thurston's topological characterization of rational maps: a topological polynomial is defined which is then shown to have no Thurston obstructions by a criterion of \cite{MR2634168}. We now briefly describe a different approach to Crofoot--Sarason's conjecture later given in \cite{2014arXiv1411.3415L}.

Therein, the space of (external) polynomials \[  \Sigma_d^* := \left\{ f(z)= z+\frac{a_1}{z} + \cdots +\frac{a_d}{z^d} : a_d=-\frac{1}{d}\textrm{ and } f|_{\widehat{\mathbb{C}}\setminus\overline{\mathbb{D}}} \textrm{ is conformal.}\right\}  \] is considered in relation to the problem of the topological connectivity of unbounded quadrature domains. We remark that the space $\Sigma_d^*$ is closely related to a space of polynomials studied by Suffridge in  \cite{MR0235107}, \cite{MR294609} in connection with coefficient-body problems for univalent functions (see also Remark \ref{notation_remark}). Let $\mathbb{T}$ denote the unit circle in $\mathbb{C}$. It is proven in  \cite{2014arXiv1411.3415L} that for $p\in\Sigma_d^*$, the curve $p(\mathbb{T})$ can have at most $d-2$ double points (see Table \ref{table_1}). If $p\in \Sigma_d^*$ is such that $p(\mathbb{T})$ has $d-2$ double points, we refer to $p$ as a \emph{Suffridge polynomial}. The existence of Suffridge polynomials in all degrees can be deduced from the Krein--Milman theorem. Moreover, it is proven that a perturbation of the Schwarz reflection map associated with a Suffridge polynomial can be extended to a map of the plane which is quasiconformally equivalent to a Crofoot--Sarason polynomial. This provides an alternative proof of Crofoot--Sarason's conjecture.

The Krein--Milman approach to existence of Suffridge polynomials yields only a single Suffridge polynomial in each degree. In the present work we construct, by a different method, \emph{all} Suffridge polynomials. Furthermore, we show that they admit a simple combinatorial classification by \emph{bi-angled trees}. Remarkably, the same bi-angled trees also classify Crofoot--Sarason polynomials. We use this connection to show that Suffridge polynomials are in 1:1 correspondence with Crofoot--Sarason polynomials, and conjecture as to a deeper connection between the two classes. We emphasize that while the original proof of the existence of Crofoot--Sarason polynomials required an appeal to Thurston's topological characterization of rational maps, the present paper proves the existence of Suffridge polynomials by developing a more elementary technique of \emph{pinching} in the class $\Sigma_d^*$: relying mainly on quasiconformal deformation methods and the compactness of $\Sigma_d^*$.

\subsection{Overview and Statement of Main Results}

In this paper, our main purpose is to establish a canonical correspondence between the following three classes of objects:

\begin{enumerate}

\item\label{Sigma_d^*} Extremal functions in the class $\Sigma_d^*$.

\item\label{bi-angled_trees} Bi-angled trees embedded in $\mathbb{C}$ with $d-1$ vertices (edges are straight line segments, meeting at angles $2\pi/3$ or $4\pi/3$).  

\item\label{crofoot-sarason} Anti-holomorphic polynomials of degree $d$ with $d-1$ distinct, fixed critical points. 
\end{enumerate}

\noindent This correspondence is in fact one-to-one given appropriate equivalence relations on (\ref{Sigma_d^*}), (\ref{bi-angled_trees}), (\ref{crofoot-sarason}). We will discuss the main objects of study (\ref{Sigma_d^*}) - (\ref{crofoot-sarason}).

Given an (external) polynomial $f\in\Sigma_d^*$, the non-zero critical points of $f$ are all distinct and lie on $\mathbb{T}$, so that $f(\mathbb{T})$ is a curve with $d+1$ cusps (see Proposition~\ref{crit_points_on_circle}). We call $f$ a \emph{Suffridge polynomial}, or refer to $f$ as \emph{extremal} if, furthermore, the curve $f(\mathbb{T})$ has the maximal number $d-2$ of self-intersections (see Table \ref{table_1} and Section~\ref{prelim_3}). Given such an extremal $f$, the interior of the curve $f(\mathbb{T})$ consists of $d-1$ topological triangles, and hence there is a natural bi-angled tree structure $\mathcal{T}(f)$ associated to $f$ (see  Table \ref{table_1} or Figure \ref{fig:tree}) given by assigning a vertex to each topological triangle and connecting two vertices if the associated triangles share a common boundary point (see Definition \ref{associated_tree}). Thus there is a natural map (\ref{Sigma_d^*}) $\mapsto$ (\ref{bi-angled_trees}). In fact, we prove that any bi-angled tree (up to isomorphism) is realized by some extremal polynomial, and, moreover, any extremal polynomial is essentially uniquely determined by its associated bi-angled tree  (see Theorem \ref{theorem_A} below).

We term the class of objects (\ref{crofoot-sarason}) as \emph{Crofoot--Sarason anti-polynomials} (CS anti-polynomials in short) of degree $d$: they are also in fact encoded by the combinatorial information given in (\ref{bi-angled_trees}). Let $p$ be a CS anti-polynomial. We prove that the \emph{angled Hubbard tree} $\mathcal{T}(p)$ of $p$ has a natural bi-angled tree structure (see Section~\ref{crofoot_sarason_sec}), giving a map (\ref{crofoot-sarason}) $\mapsto$ (\ref{bi-angled_trees}). The bijective correspondence (\ref{bi-angled_trees}) $\leftrightarrow$  (\ref{crofoot-sarason}) is then given by invoking a classification of Hubbard trees of anti-holomorphic polynomials \cite{Poi3}. 

We will now state our main result after introducing the appropriate equivalence relations on (\ref{Sigma_d^*}) - (\ref{crofoot-sarason}). There is a natural $\mathbb{Z}_{d+1}$-action on $\Sigma_d^*$ given by conjugating any $f\in\Sigma_d^*$ by multiplication by a $(d+1)^{\textrm{st}}$ root of unity. The appropriate equivalence relation $\sim$ on (\ref{bi-angled_trees}) is given by the notion of bi-angled tree isomorphism as in Definition \ref{isom_def}, and for (\ref{crofoot-sarason}) we consider two anti-polynomials to be equivalent if they are conjugate by an affine map $z\mapsto az+b$.

\begin{thmx}\label{theorem_A} Let $d\geq2$. There is a canonical bijection between:

\begin{align} \left\{ f \in \Sigma_d^* : f(\mathbb{T}) \textrm{ has } d-2 \textrm{ double points} \right\} \big/ \hspace{1mm} \mathbb{Z}_{d+1} \phantom{} \label{Suffridge_part} \\ \left\{ \textrm{Bi-angled trees with } d-1 \textrm{ vertices}\right\}  \big/ \sim \label{tree_part} \phantom{} \\ \left\{ \textrm{CS anti-polynomials of degree } d\right\} \big/ \textrm{Aut}(\mathbb{C}) \label{CS_polynomial_part}. \end{align}

\end{thmx}

For each $f\in\Sigma_d^*$, the unbounded, simply connected domain $f(\widehat{\C}\setminus\overline{\D})$ admits an anti-meromorphic map $\sigma$ that continuously extends to the identity map on the boundary of the domain. Such a map is called a \emph{Schwarz reflection map}, and a domain that admits a Schwarz reflection map is called a \emph{quadrature domain} (see Subsection~\ref{prelim_1} for precise definitions). With this terminology, the notion of a Suffridge polynomial is essentially equivalent to that of an unbounded, \emph{extremal quadrature domain} (see Proposition~\ref{suffridge_extremal_qd_equiv_prop}). The bijective correspondence (\ref{Sigma_d^*}) $\leftrightarrow$ (\ref{bi-angled_trees}) stated in Theorem~\ref{theorem_A} can thus be rephrased as a classification theorem for unbounded, extremal quadrature domains (see Theorem~\ref{bijection_thm}), and it is this perspective that we will maintain through most of this work.

The established combinatorial link between (\ref{Sigma_d^*}), (\ref{crofoot-sarason}) is remarkable, and its existence is only the tip of the iceberg. In fact, in a sequel to this work, we will explain a deeper connection between (\ref{Sigma_d^*}), (\ref{crofoot-sarason}): for $f\in\Sigma_d^*$ with $d-2$ double points on $f(\mathbb{T})$, the dynamics of the Schwarz reflection map (associated to the unbounded extremal quadrature domain $f(\widehat{\C}\setminus\overline{\D})$) on its limit set is topologically conjugate to the dynamics of the corresponding CS anti-polynomial on its Julia set.

\begin{table}
\begin{adjustwidth}{-.65in}{-.65in} 
\begin{center}
\begin{tabular}{ |c|c|c|c|c| } 
\hline
$d$ & Bi-angled tree & Droplet & Suffridge Polynomial & Generators of Lamination \\
\hline

    2 & \parbox[c]{4em}{\scalebox{.11}{
\begingroup%
  \makeatletter%
  \providecommand\color[2][]{%
    \errmessage{(Inkscape) Color is used for the text in Inkscape, but the package 'color.sty' is not loaded}%
    \renewcommand\color[2][]{}%
  }%
  \providecommand\transparent[1]{%
    \errmessage{(Inkscape) Transparency is used (non-zero) for the text in Inkscape, but the package 'transparent.sty' is not loaded}%
    \renewcommand\transparent[1]{}%
  }%
  \providecommand\rotatebox[2]{#2}%
  \ifx\svgwidth\undefined%
    \setlength{\unitlength}{286.88332625bp}%
    \ifx\svgscale\undefined%
      \relax%
    \else%
      \setlength{\unitlength}{\unitlength * \real{\svgscale}}%
    \fi%
  \else%
    \setlength{\unitlength}{\svgwidth}%
  \fi%
  \global\let\svgwidth\undefined%
  \global\let\svgscale\undefined%
  \makeatother%
  \begin{picture}(1,1.02253522)%
    \put(0,0){\includegraphics[width=\unitlength,page=1]{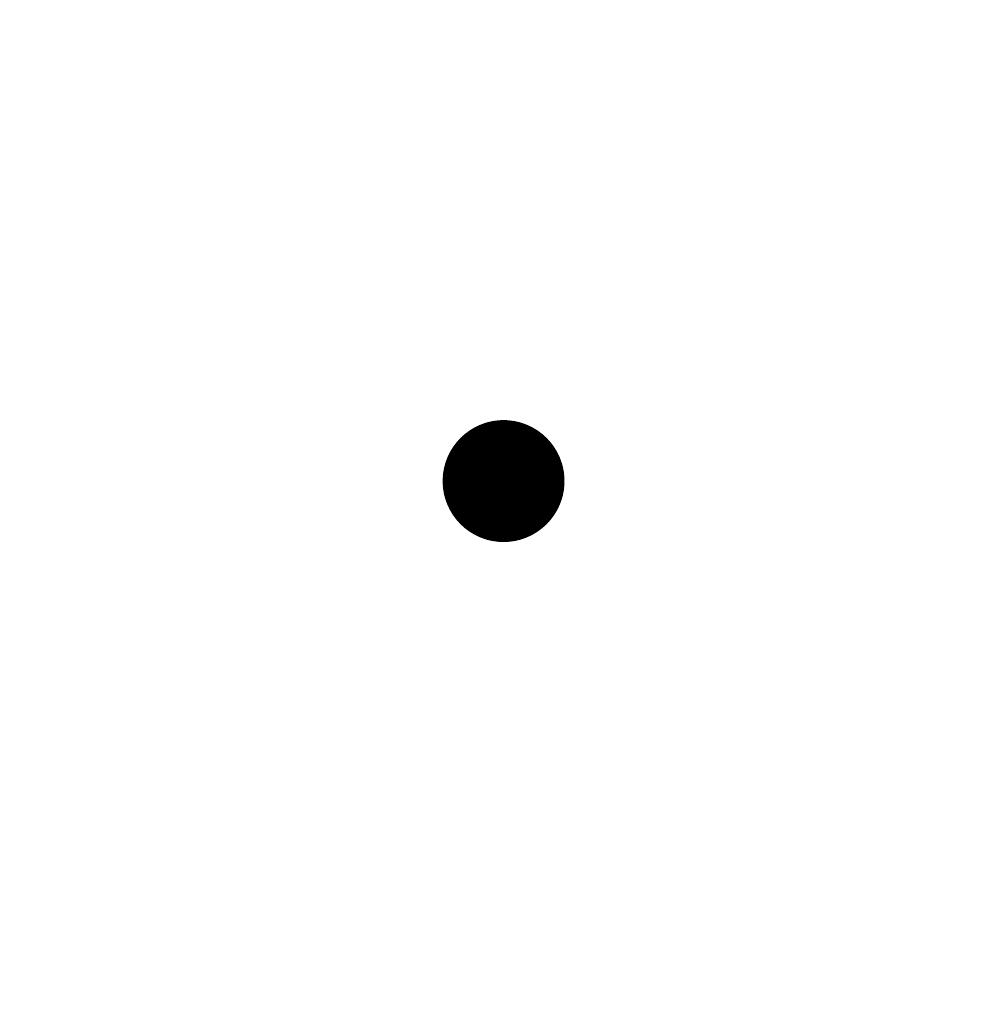}}%
  \end{picture}%
\endgroup%
}} & \parbox[c]{1em}{\includegraphics[width=0.04\textwidth]{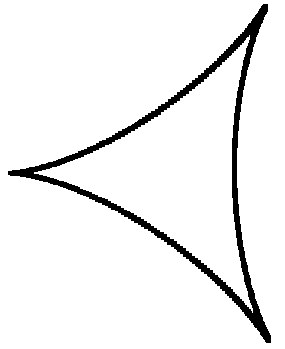}} & $z-\frac{1}{2z^2}$ & $\emptyset$ \\

    3 & \parbox[c]{4em}{\scalebox{.11}{
\begingroup%
  \makeatletter%
  \providecommand\color[2][]{%
    \errmessage{(Inkscape) Color is used for the text in Inkscape, but the package 'color.sty' is not loaded}%
    \renewcommand\color[2][]{}%
  }%
  \providecommand\transparent[1]{%
    \errmessage{(Inkscape) Transparency is used (non-zero) for the text in Inkscape, but the package 'transparent.sty' is not loaded}%
    \renewcommand\transparent[1]{}%
  }%
  \providecommand\rotatebox[2]{#2}%
  \ifx\svgwidth\undefined%
    \setlength{\unitlength}{286.88332625bp}%
    \ifx\svgscale\undefined%
      \relax%
    \else%
      \setlength{\unitlength}{\unitlength * \real{\svgscale}}%
    \fi%
  \else%
    \setlength{\unitlength}{\svgwidth}%
  \fi%
  \global\let\svgwidth\undefined%
  \global\let\svgscale\undefined%
  \makeatother%
  \begin{picture}(1,1.02253522)%
    \put(0,0){\includegraphics[width=\unitlength,page=1]{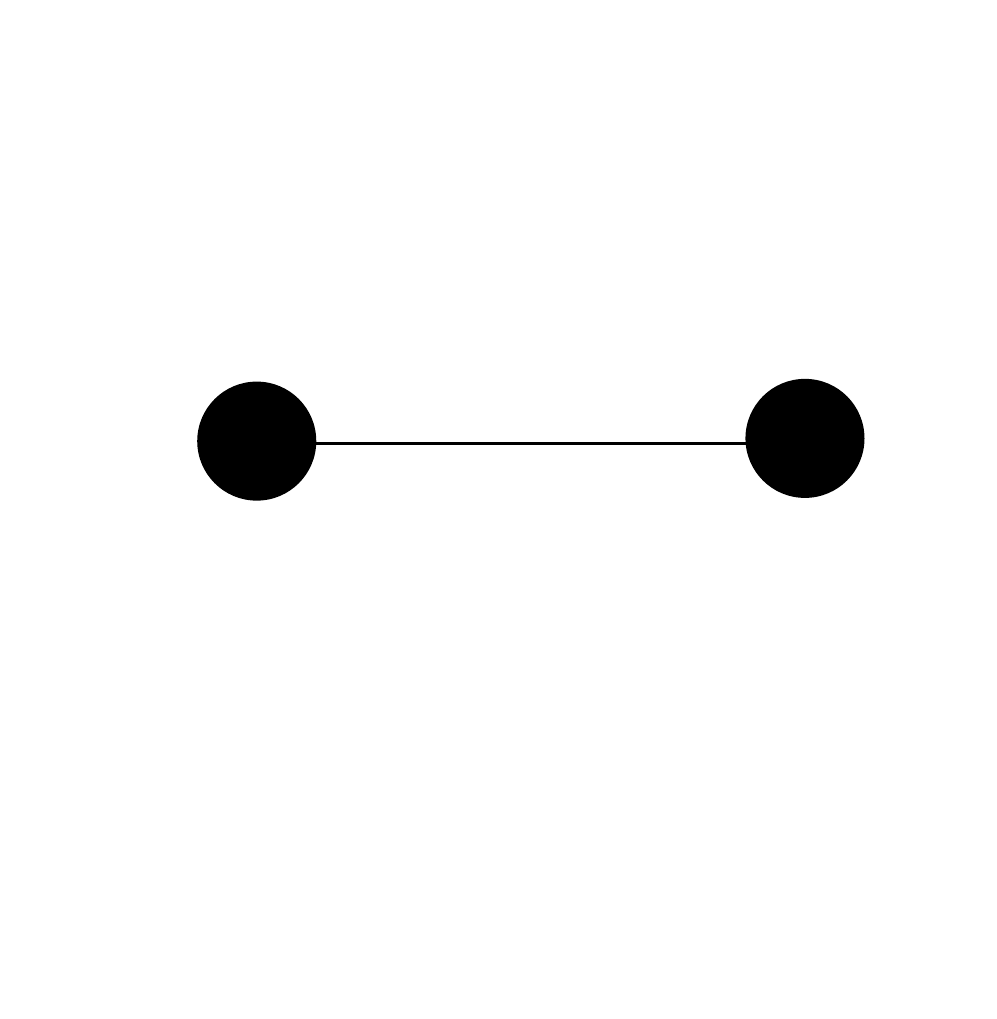}}%
  \end{picture}%
\endgroup%
}} & \parbox[c]{3.75em}{\includegraphics[scale=0.06]{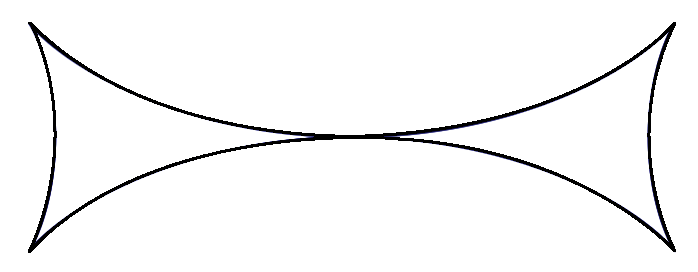}} & $z+\frac{2}{3z}-\frac{1}{3z^3}$ & $\{\frac18, \frac58\}$ \\

    4 & \parbox[c]{2.5em}{\scalebox{.11}{
\begingroup%
  \makeatletter%
  \providecommand\color[2][]{%
    \errmessage{(Inkscape) Color is used for the text in Inkscape, but the package 'color.sty' is not loaded}%
    \renewcommand\color[2][]{}%
  }%
  \providecommand\transparent[1]{%
    \errmessage{(Inkscape) Transparency is used (non-zero) for the text in Inkscape, but the package 'transparent.sty' is not loaded}%
    \renewcommand\transparent[1]{}%
  }%
  \providecommand\rotatebox[2]{#2}%
  \ifx\svgwidth\undefined%
    \setlength{\unitlength}{252.9421901bp}%
    \ifx\svgscale\undefined%
      \relax%
    \else%
      \setlength{\unitlength}{\unitlength * \real{\svgscale}}%
    \fi%
  \else%
    \setlength{\unitlength}{\svgwidth}%
  \fi%
  \global\let\svgwidth\undefined%
  \global\let\svgscale\undefined%
  \makeatother%
  \begin{picture}(1,1.38338663)%
    \put(0,0){\includegraphics[width=\unitlength,page=1]{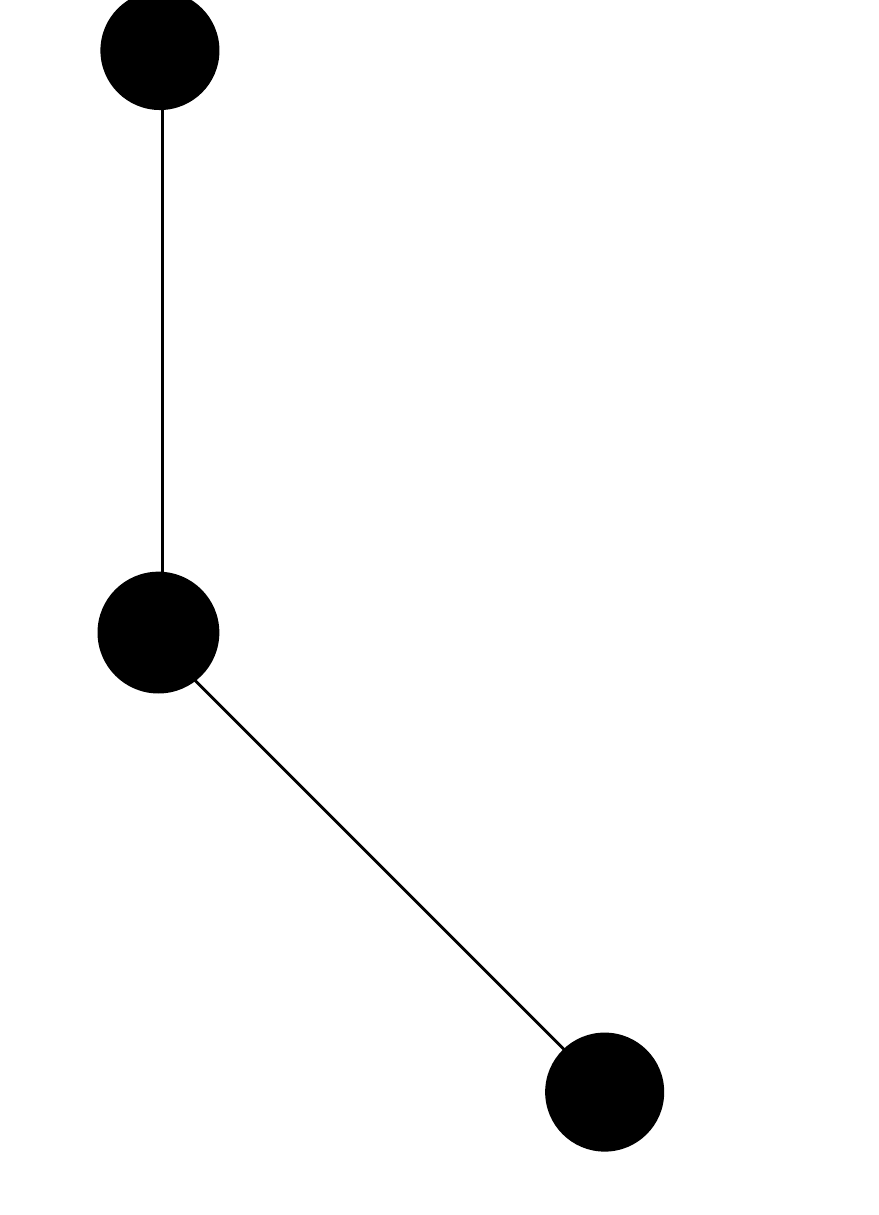}}%
  \end{picture}%
\endgroup%
}} & \parbox[c]{3.em}{\includegraphics[scale=0.032]{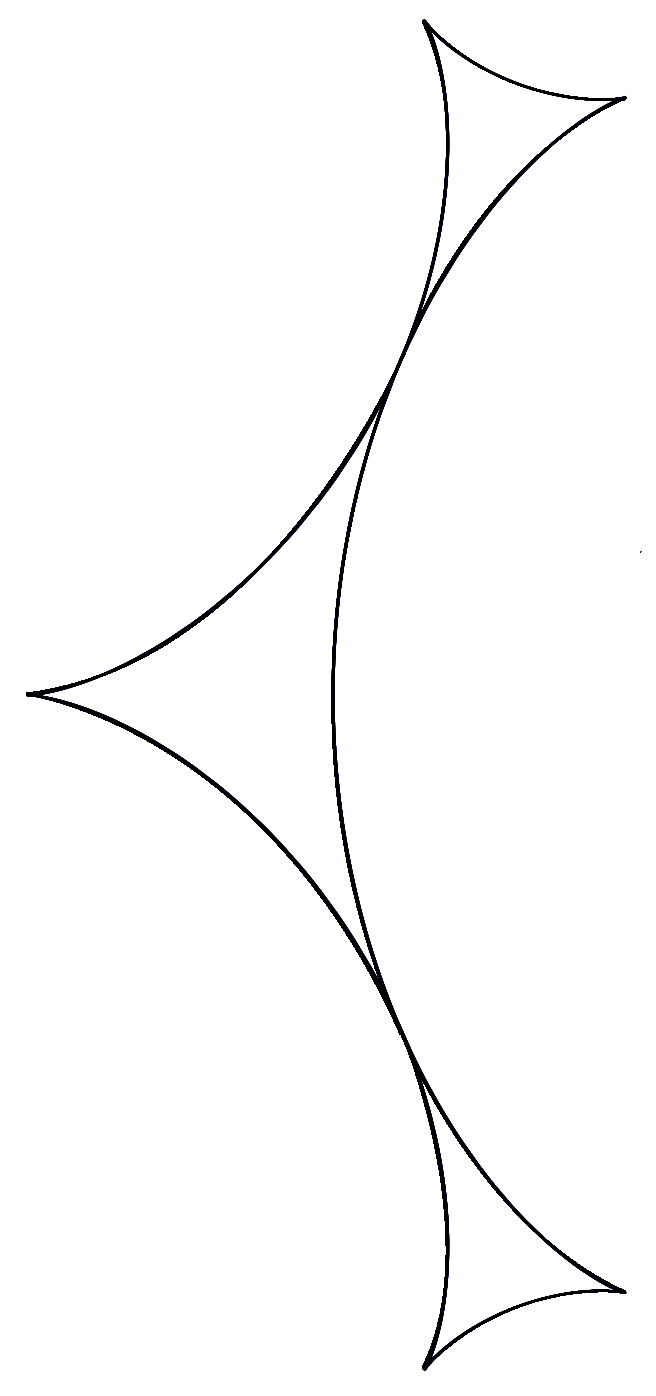}} & $z-\frac{5}{8z} -\frac{5}{16z^2}-\frac{1}{4z^4}$ & $\{\{\frac{1}{15},\frac{11}{15}\}, \{\frac{2}{15},\frac{7}{15}\}\}$ \\

    5 & \parbox[c]{4em}{\scalebox{.11}{
\begingroup%
  \makeatletter%
  \providecommand\color[2][]{%
    \errmessage{(Inkscape) Color is used for the text in Inkscape, but the package 'color.sty' is not loaded}%
    \renewcommand\color[2][]{}%
  }%
  \providecommand\transparent[1]{%
    \errmessage{(Inkscape) Transparency is used (non-zero) for the text in Inkscape, but the package 'transparent.sty' is not loaded}%
    \renewcommand\transparent[1]{}%
  }%
  \providecommand\rotatebox[2]{#2}%
  \ifx\svgwidth\undefined%
    \setlength{\unitlength}{336.97875496bp}%
    \ifx\svgscale\undefined%
      \relax%
    \else%
      \setlength{\unitlength}{\unitlength * \real{\svgscale}}%
    \fi%
  \else%
    \setlength{\unitlength}{\svgwidth}%
  \fi%
  \global\let\svgwidth\undefined%
  \global\let\svgscale\undefined%
  \makeatother%
  \begin{picture}(1,1.05036101)%
    \put(0,0){\includegraphics[width=\unitlength,page=1]{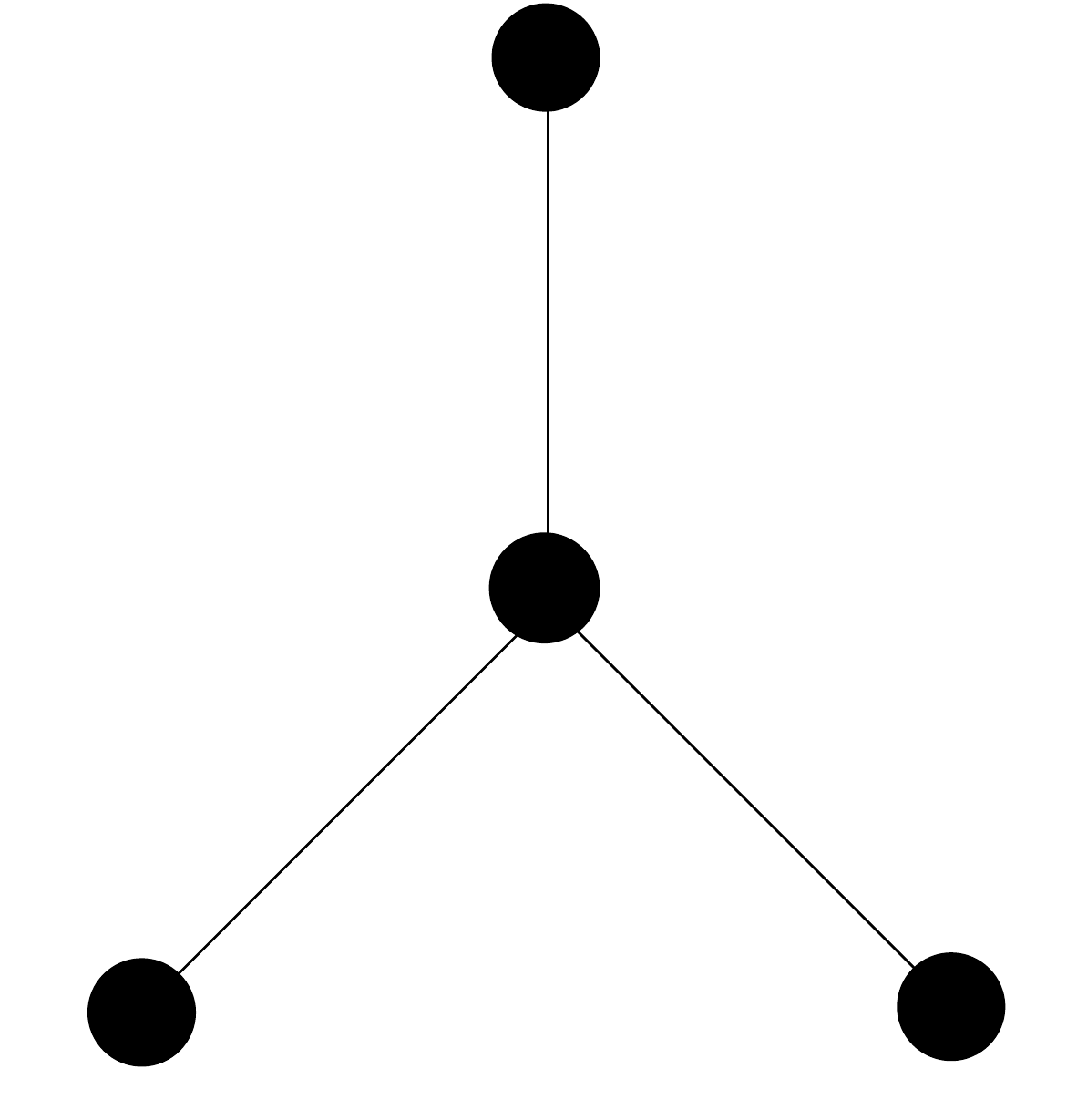}}%
  \end{picture}%
\endgroup%
}} & \parbox[c]{3.em}{\includegraphics[scale=0.1]{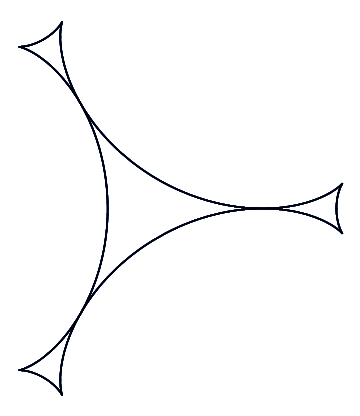}} & $z+\frac{2\sqrt{2}}{5z^2}-\frac{1}{5z^5}$ & $\{\{\frac{5}{24},\frac{23}{24}\}, \{\frac{7}{24},\frac{13}{24}\},\{\frac{15}{24},\frac{21}{24}\}\}$ \\

    5 & \parbox[c]{3em}{\scalebox{.11}{
\begingroup%
  \makeatletter%
  \providecommand\color[2][]{%
    \errmessage{(Inkscape) Color is used for the text in Inkscape, but the package 'color.sty' is not loaded}%
    \renewcommand\color[2][]{}%
  }%
  \providecommand\transparent[1]{%
    \errmessage{(Inkscape) Transparency is used (non-zero) for the text in Inkscape, but the package 'transparent.sty' is not loaded}%
    \renewcommand\transparent[1]{}%
  }%
  \providecommand\rotatebox[2]{#2}%
  \ifx\svgwidth\undefined%
    \setlength{\unitlength}{243.23661264bp}%
    \ifx\svgscale\undefined%
      \relax%
    \else%
      \setlength{\unitlength}{\unitlength * \real{\svgscale}}%
    \fi%
  \else%
    \setlength{\unitlength}{\svgwidth}%
  \fi%
  \global\let\svgwidth\undefined%
  \global\let\svgscale\undefined%
  \makeatother%
  \begin{picture}(1,2.05651365)%
    \put(0,0){\includegraphics[width=\unitlength,page=1]{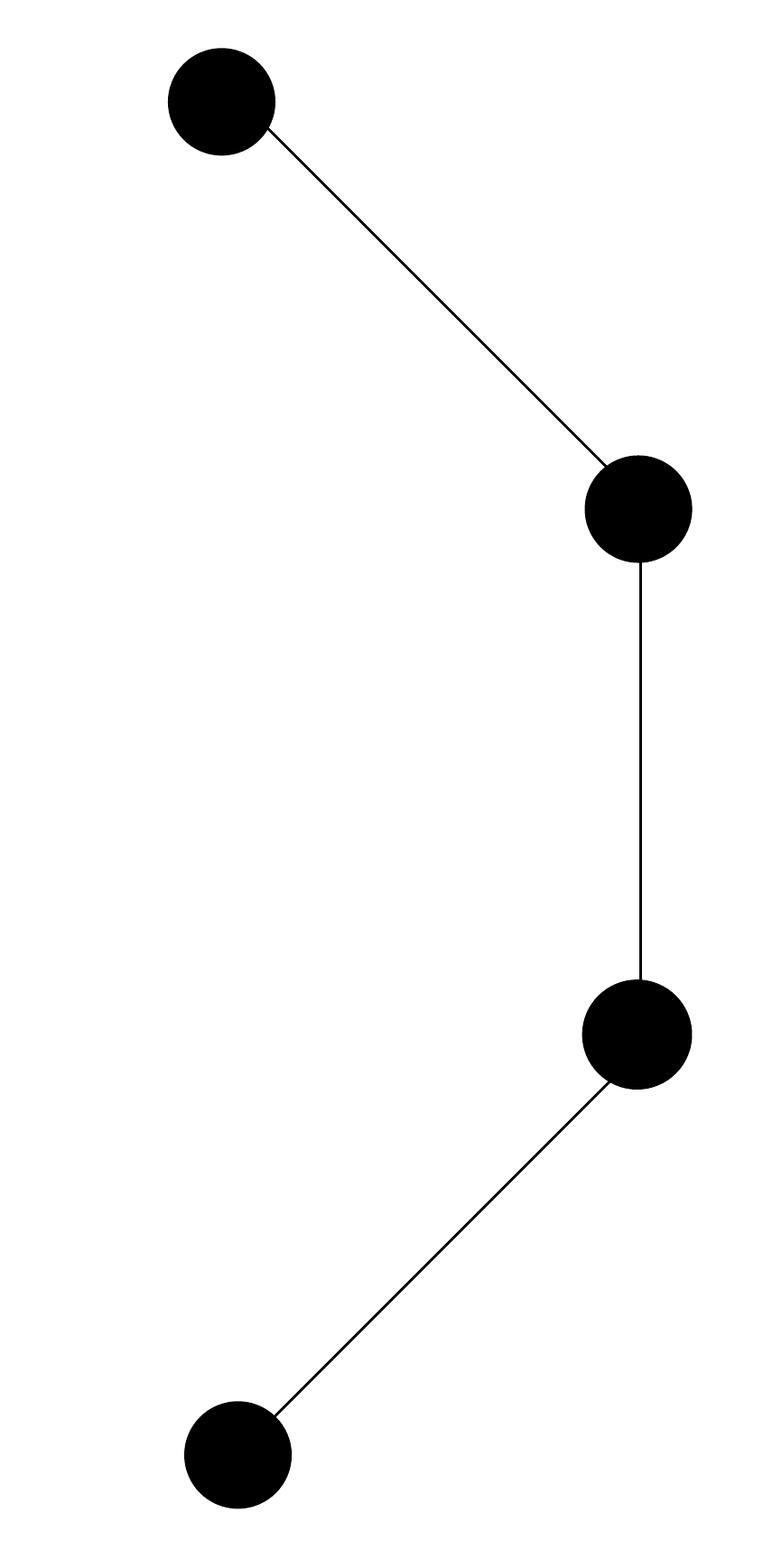}}%
  \end{picture}%
\endgroup%
}} & \parbox[c]{3.em}{\includegraphics[scale=0.12]{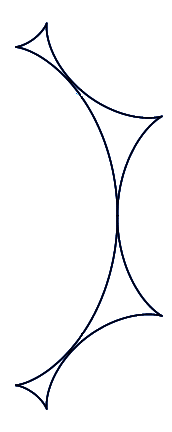}} & $z-\frac{0.6}{z}+\frac{0.3}{z^2}-\frac{0.6}{3z^3}-\frac{1}{5z^5}$ & $\{\{\frac{5}{24},\frac{23}{24}\}, \{\frac{7}{24},\frac{13}{24}\},\{\frac{6}{24},\frac{18}{24}\}\}$\\

    5 & \parbox[c]{4.75em}{\scalebox{.11}{
\begingroup%
  \makeatletter%
  \providecommand\color[2][]{%
    \errmessage{(Inkscape) Color is used for the text in Inkscape, but the package 'color.sty' is not loaded}%
    \renewcommand\color[2][]{}%
  }%
  \providecommand\transparent[1]{%
    \errmessage{(Inkscape) Transparency is used (non-zero) for the text in Inkscape, but the package 'transparent.sty' is not loaded}%
    \renewcommand\transparent[1]{}%
  }%
  \providecommand\rotatebox[2]{#2}%
  \ifx\svgwidth\undefined%
    \setlength{\unitlength}{533.37143785bp}%
    \ifx\svgscale\undefined%
      \relax%
    \else%
      \setlength{\unitlength}{\unitlength * \real{\svgscale}}%
    \fi%
  \else%
    \setlength{\unitlength}{\svgwidth}%
  \fi%
  \global\let\svgwidth\undefined%
  \global\let\svgscale\undefined%
  \makeatother%
  \begin{picture}(1,0.62502678)%
    \put(0,0){\includegraphics[width=\unitlength,page=1]{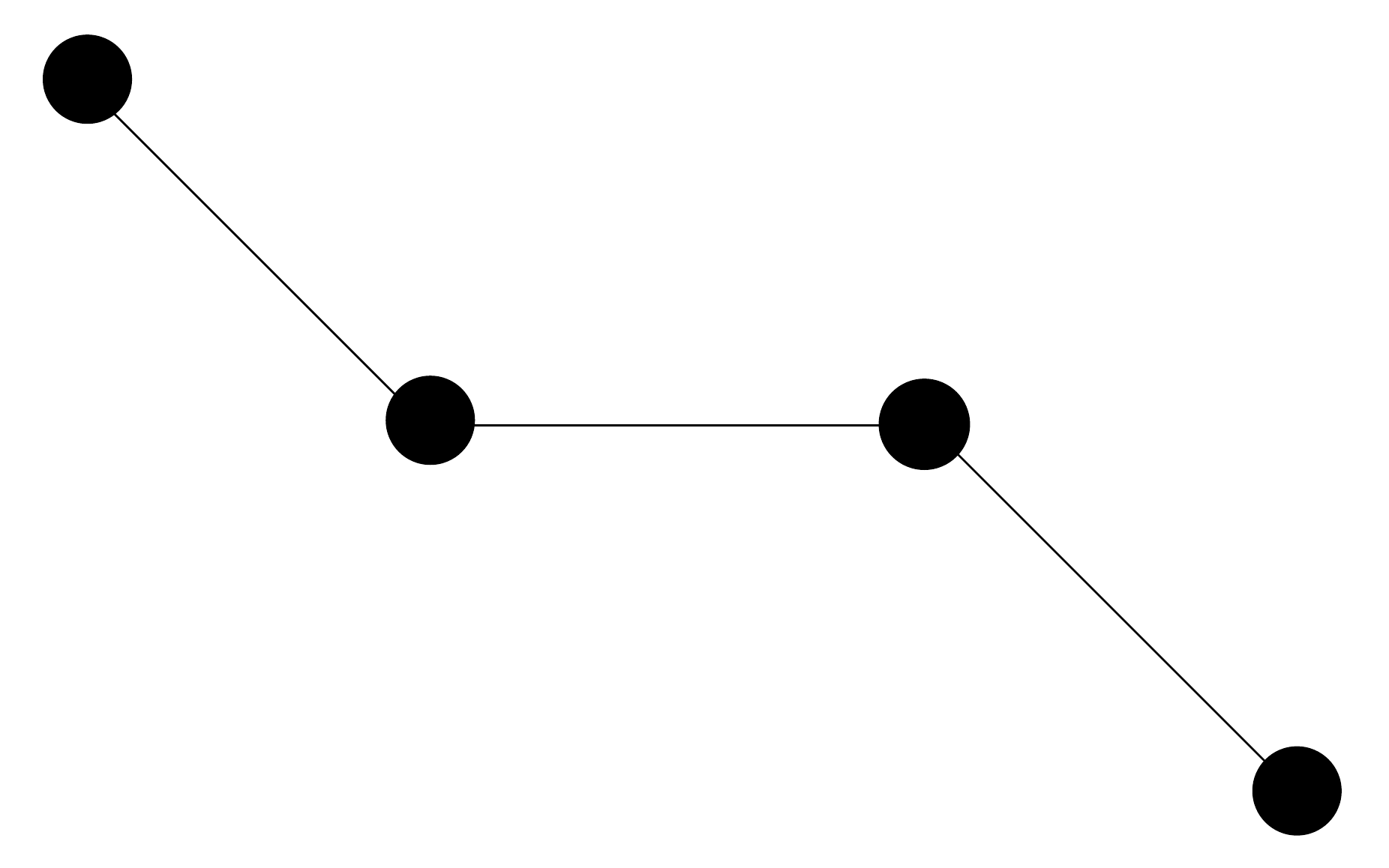}}%
  \end{picture}%
\endgroup%
}} & \parbox[c]{3.em}{ \includegraphics[scale=0.1]{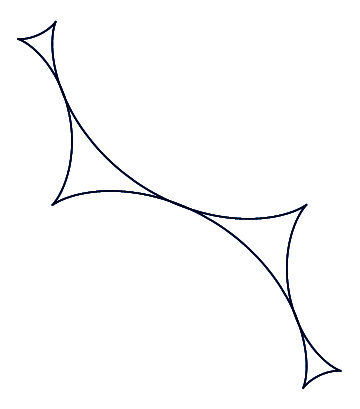}} & $z-\frac{0.71i}{z}+\frac{0.71i}{3z^3}-\frac{1}{5z^5}$ & $\{\{\frac{1}{24},\frac{19}{24}\}, \{\frac{7}{24},\frac{13}{24}\},\{\frac{6}{24},\frac{18}{24}\}\}$\\

    5 & \parbox[c]{4.75em}{\scalebox{.11}{
\begingroup%
  \makeatletter%
  \providecommand\color[2][]{%
    \errmessage{(Inkscape) Color is used for the text in Inkscape, but the package 'color.sty' is not loaded}%
    \renewcommand\color[2][]{}%
  }%
  \providecommand\transparent[1]{%
    \errmessage{(Inkscape) Transparency is used (non-zero) for the text in Inkscape, but the package 'transparent.sty' is not loaded}%
    \renewcommand\transparent[1]{}%
  }%
  \providecommand\rotatebox[2]{#2}%
  \ifx\svgwidth\undefined%
    \setlength{\unitlength}{533.37143785bp}%
    \ifx\svgscale\undefined%
      \relax%
    \else%
      \setlength{\unitlength}{\unitlength * \real{\svgscale}}%
    \fi%
  \else%
    \setlength{\unitlength}{\svgwidth}%
  \fi%
  \global\let\svgwidth\undefined%
  \global\let\svgscale\undefined%
  \makeatother%
  \begin{picture}(1,0.62502678)%
    \put(0,0){\includegraphics[width=\unitlength,page=1]{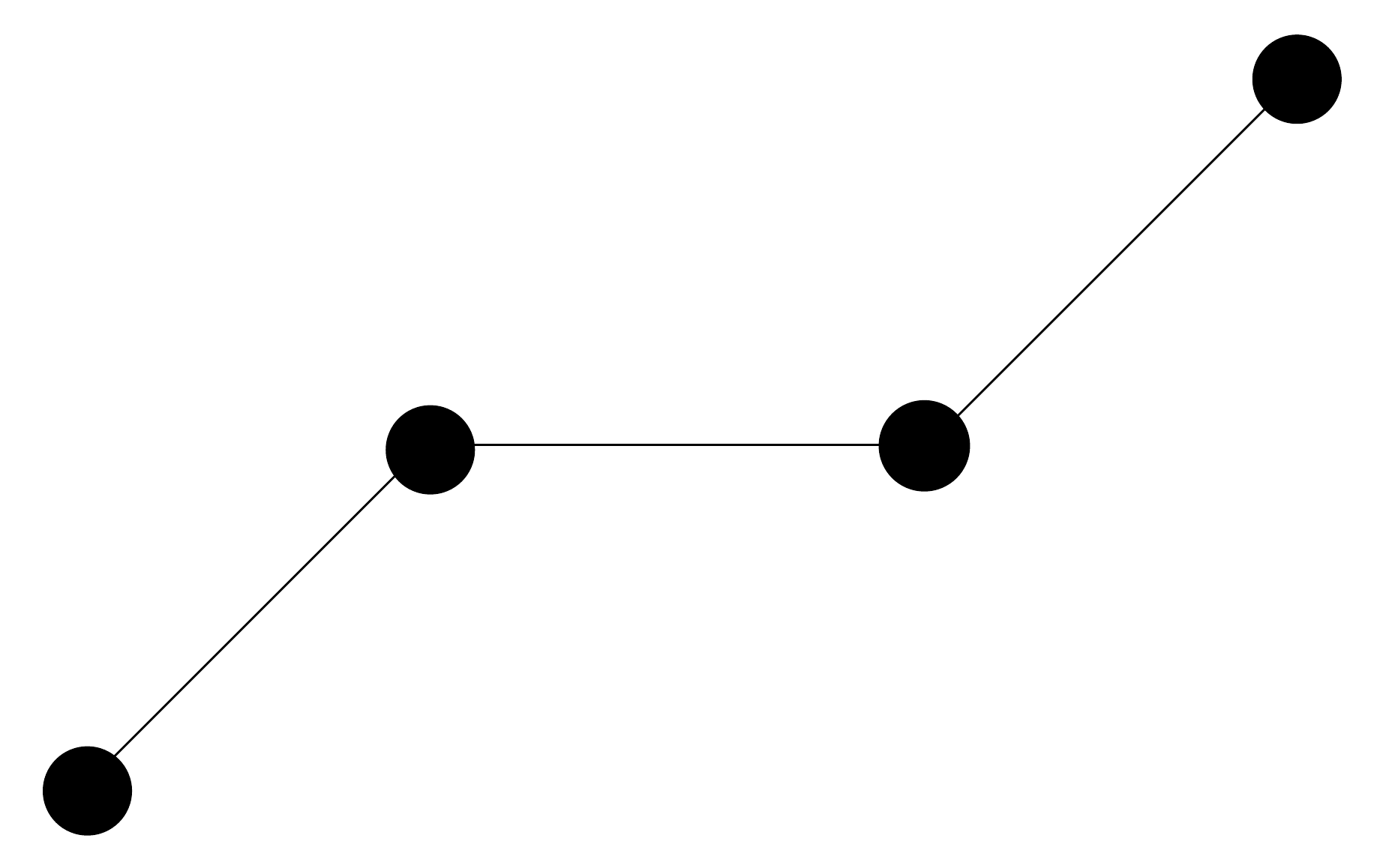}}%
  \end{picture}%
\endgroup%
}} & \parbox[c]{3.em}{ \includegraphics[scale=0.1]{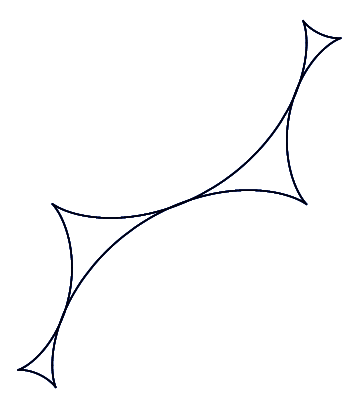}} & $z+\frac{0.71i}{z}-\frac{0.71i}{3z^3}-\frac{1}{5z^5}$ & $\{\{\frac{5}{24},\frac{23}{24}\}, \{\frac{11}{24},\frac{17}{24}\},\{\frac{6}{24},\frac{18}{24}\}\}$ \\

& & & & \\

\hline
\end{tabular}
\end{center}
\end{adjustwidth}
\caption{This table displays all possible bi-angled trees with $d-1$ vertices (up to isomorphism) for $d=2,3,4,5$, and the corresponding droplets (i.e., complement of the corresponding extremal unbounded quadrature domains of order $d$) along with the Suffridge polynomials of degree $d+1$ uniformizing the extremal unbounded quadrature domains (the last three are numerical approximations of the actual Suffridge polynomials). The last column shows the angles of the pairs of external rays (which are radial lines in the B{\"o}ttcher coordinate of the basin of infinity of the Schwarz reflection map, see Subsection~\ref{rigidity_thm_subsec}) landing at the double points of the droplet; these angle pairs generate an equivalence relation on $\R/\Z$ which yields a topological model of the limit set of the corresponding Schwarz reflection map (and also of the Julia set of the corresponding CS anti-polynomial).}
\label{table_1}
\end{table}

The classes $\Sigma_d^*$ are closely related to the more classical spaces of univalent polynomials: \[ S_d^* := \left\{ f(z)= z+a_2{z^2} + \cdots +a_dz^d : a_d=\frac{1}{d}\textrm{ and } f|_{\mathbb{D}} \textrm{ is conformal.}\right\}.   \] We say that $f \in S_d^*$ is a \emph{Suffridge polynomial}, or refer to $f$ as \emph{extremal}, if the curve $f(\mathbb{T})$ has $d-2$ self-intersections. Extremal points among $S_d^*$ have a natural rooted binary tree structure (see Table \ref{table_2} in Section~\ref{S_d^*_section}), and we establish the following existence and uniqueness result for the class $S_d^*$ in Section~\ref{S_d^*_section}:

\begin{thmx}\label{theorem_B} Let $d\geq2$. There is a canonical bijection between

\begin{align} \left\{ f \in S_d^* : f \textrm{ has } d-2 \textrm{ double points} \right\} \big/ \hspace{1mm} \mathbb{Z}_{d-1}  \phantom{} \label{Suffridge_bounded_part} \\ \left\{ \textrm{Rooted binary trees with } d-2 \textrm{ vertices}\right\}. \label{tree_part_S}  \end{align}

\end{thmx}

\noindent The enumeration of rooted binary trees is elementary, and hence we deduce:

\begin{cor*}\label{counting_extremal_bounded_qds} Let $d\geq2$. Then \begin{equation}\label{counting_formula} \# \bigg( \big\{ f \in S_d^* : f \emph{ has } d- 2 \emph{ double points} \big\}\big/ \hspace{1mm} \mathbb{Z}_{d-1}\bigg) = \frac{1}{d-1}{2(d-2) \choose d-2}.   \nonumber \end{equation}\end{cor*}

Sections~\ref{preliminaries} and~\ref{general_dynamics_sec} are devoted to establishing the terminology for, and basic properties of, the objects (\ref{Sigma_d^*}) - (\ref{crofoot-sarason}). Section~\ref{mainthm_proof} is devoted to proving the existence of Suffridge polynomials with prescribed bi-angled tree structure (this is equivalent to surjectivity of the map (\ref{Suffridge_part})$\rightarrow$(\ref{tree_part}) in Theorem \ref{theorem_A}). In Section~\ref{rigidity_qd_sec}, we prove that a Suffridge polynomial is uniquely determined (up to $\mathbb{Z}_{d+1}$) by its bi-angled tree structure (this is equivalent to injectivity of the map (\ref{Suffridge_part})$\rightarrow$(\ref{tree_part}) in Theorem \ref{theorem_A}). In Section ~\ref{crofoot_sarason_sec} we study Crofoot--Sarason anti-polynomials and finish the proof of Theorem \ref{theorem_A}. Lastly, Section~\ref{S_d^*_section} is devoted to the study of $S_d^*$ and a sketch of the proof of Theorem \ref{theorem_B}.

\subsection{Discussion of Proof Methods}

We will now highlight briefly some methods used in the proofs of Theorems \ref{theorem_A} and \ref{theorem_B}. In Section~\ref{mainthm_proof}, we prove the existence of Suffridge polynomials with prescribed bi-angled tree structure (see Theorem \ref{mainthm_qd_terminology}). This is done via the development of a technique of ``pinching'' for the class $\Sigma_d^*$ (see Theorem \ref{pinching_theorem}). One begins with a standard base point $f_0(z)=z-1/dz^d$ in $\Sigma_d^*$ (see Figure \ref{fig:unit_disc}), and successively ``pinches'' edges in $f_0(\mathbb{T})$ to create double points. More precisely, the Schwarz reflection map associated with $f_0$ allows us to employ a quasiconformal deformation argument to increase the moduli of certain carefully chosen topological quadrilaterals. The quasiconformal deformations of $f_0$ stay within the compact family $\Sigma_d^*$, so that we are able to take a limit in $\Sigma_d^*$ (despite the degeneration of the corresponding quasiconformal maps' dilatation), whence we prove that the limiting function has the desired double point structure. Given a tree $\mathcal{T}$, one reads off (from the combinatorics of $\mathcal{T}$) which edges of $f_0(\mathbb{T})$ should be pinched, and carries out the above described pinching procedure step by step. It is worth mentioning that the technique of producing new univalent polynomials from old ones using quasiconformal deformation of the associated Schwarz reflection maps is a novelty of this paper.

The next Section~\ref{rigidity_qd_sec} is devoted to proving that a Suffridge polynomial is uniquely determined (up to $\mathbb{Z}_{d+1}$) by its associated bi-angled tree (see Theorem \ref{rigidity_extremal_qd_thm}). The proof of this theorem is carried out in various steps. We first show that if two extremal unbounded quadrature domains admit isomorphic bi-angled trees, then the corresponding droplets (i.e., the complements of the quadrature domains) are conformally equivalent. Our knowledge of the geometric properties of the cusps and double points on the boundaries of the droplets (see Propositions~\ref{cusp_geometry} and~\ref{double_geometry}) allows us to show that the conformal map between the two droplets is asymptotically linear near the singular points (see Lemma~\ref{asymp_linear}). This is then used to extend the above conformal map to a quasiconformal map of the whole sphere such that it conjugates the Schwarz reflection maps of the two extremal unbounded quadrature domains near the super-attracting fixed point at $\infty$. This is the content of Lemma~\ref{global_qc}. We conclude the proof of Theorem~\ref{rigidity_extremal_qd_thm} by employing a ``pullback argument'' from holomorphic dynamics. More precisely, we lift the quasiconformal homeomorphism of Lemma~\ref{global_qc} by iterates of the Schwarz reflection maps (such that the dilatations do not change under lifting), and pass to a limit which we show to be a conformal conjugacy between the Schwarz reflection maps of the two extremal unbounded quadrature domains.

The aim of Section~\ref{crofoot_sarason_sec} is to prove the Classification Theorem~\ref{crofoot_sarason_bi-angled_bijection_thm} for Crofoot--Sarason polynomials (see \cite{Ge} for a classification of \emph{real-symmetric} Crofoot--Sarason polynomials). We start by investigating some topological properties of the immediate basins (of the finite critical points) of CS anti-polynomials (Propositions~\ref{sarason_bi-angled} and~\ref{connectedness_prop}). Propositions~\ref{sarason_bi-angled} and~\ref{connectedness_prop} also show that the angled Hubbard tree of a CS anti-polynomial can be naturally viewed as a bi-angled tree. A simple application of Poirier's classification of Hubbard trees of anti-polynomials then yields the desired bijection statement in Theorem~\ref{crofoot_sarason_bi-angled_bijection_thm}.
\bigskip

\noindent\textbf{Acknowledgements.} The third author was supported by the Institute for Mathematical Sciences at Stony Brook
University, an endowment from Infosys Foundation, and SERB research grant SRG/2020/000018 during parts of the work on this project. He also thanks Caltech for their support towards the pro ject.

\section{Preliminaries}
\label{preliminaries}

In this Section, we will introduce the notions of quadrature domains and Schwarz reflection maps which are crucial to the proofs of Theorems~\ref{mainthm_qd_terminology} and~\ref{rigidity_extremal_qd_thm}. Subsection~\ref{prelim_1} contains the related definitions of quadrature domain and Schwarz reflection map. Subsection~\ref{prelim_2} discusses singularities on the boundaries of the quadrature domains considered in this paper. In Subsection~\ref{sigma_d_subsec}, we prove some elementary properties about the space $\Sigma_d^*$. Subsection~\ref{prelim_3} introduces the extremal functions in $\Sigma_d^*$, and Subsection~\ref{prelim_4} discusses the bi-angled trees appearing in the statement of Theorem~\ref{theorem_A}. Crofoot--Sarason anti-polynomials are introduced in Subsections~\ref{prelim_5}, and the associated angled Hubbard trees are described in Subsection~\ref{prelim_5.5}. Lastly, Subsection~\ref{prelim_6} recalls the notion of topological quadrilaterals which we will use throughout the course of the proof of Theorem~\ref{mainthm_qd_terminology}.

\vspace{2mm}

\textbf{Notation.} We will denote reflection in the unit circle by $\eta$. An \emph{affine map} $A$ (on the complex plane) is defined as $A(z)=az+b$, for some $a\in\C^*$ and $b\in\C$.

Throughout the paper, we will assume that $d\geq 2$.

\subsection{Quadrature Domains}\label{prelim_1}

\begin{definition}\label{schwarz_func_def}
Let $\Omega\subsetneq\widehat{\C}$ be a domain such that $\infty\notin\partial\Omega$ and $\interior{\overline{\Omega}}=\Omega$. A \emph{Schwarz function} of $\Omega$ is a meromorphic extension of $(z\mapsto\overline{z})\vert_{\partial\Omega}$ to all of $\Omega$. More precisely, a continuous function $S:\overline{\Omega}\to\widehat{\C}$ of $\Omega$ is called a Schwarz function of $\Omega$ if it satisfies the following two properties:
\begin{enumerate}
\item $S$ is meromorphic on $\Omega$,

\item $S(z)=\overline{z}$ for $z\in\partial \Omega$.
\end{enumerate}
\end{definition}

It is easy to see from the definition that a Schwarz function of a domain (if it exists) is unique.

\begin{definition}\label{qd_def}
A domain $\Omega\subsetneq\widehat{\C}$ with $\infty\notin\partial\Omega$ and $\interior{\overline{\Omega}}=\Omega$ is called a \emph{quadrature domain} if $\Omega$ admits a Schwarz function.
\end{definition}

Therefore, for a quadrature domain $\Omega$, the map $\sigma:=(z\mapsto\overline{z})\circ S:\overline{\Omega}\to\widehat{\C}$ is an anti-meromorphic extension of the local reflection maps with respect to $\partial\Omega$ near its non-singular points (the reflection map fixes $\partial\Omega$ pointwise). We will call $\sigma$ the \emph{Schwarz reflection map of} $\Omega$.

Simply connected quadrature domains are of particular interest, and these admit a simple characterization. We only state the result for unbounded simply connected quadrature domains (such that $\infty$ lies in the interior of the domain). For a subset $A$ of the Riemann sphere, we denote its complement $\widehat{\C}\setminus A$ by $A^c$.

\begin{prop}\label{s.c.q.d.}
An unbounded simply connected domain $\Omega\subsetneq\widehat{\C}$ with $\infty\notin\partial\Omega$ and $\interior{\overline{\Omega}}=\Omega$ is a quadrature domain if and only if the Riemann uniformization $f:\widehat{\C}\setminus\overline{\mathbb{D}}\to\Omega$ extends to a rational map on $\widehat{\C}$. In this case, the Schwarz reflection map $\sigma$ of $\Omega$ is given by $f\circ\eta\circ(f\vert_{\widehat{\C}\setminus\overline{\mathbb{D}}})^{-1}$, and if $\deg{f}\geq 2$, we have $\sigma(\overline{\Omega})=\widehat{\C}$. 

Moreover, if the degree of the rational map $f$ is $d+1$, then $\sigma:\sigma^{-1}(\Omega)\to\Omega$ is a branched covering of degree $d$, and $\sigma:\sigma^{-1}(\interior{\Omega^c})\to\interior{\Omega^c}$ is a branched covering of degree $d+1$.
\end{prop}
\begin{proof}
The first part is the content of \cite[Theorem~1]{AS}. The statements about covering properties of $\sigma$ follow from the commutative diagram in Figure~\ref{comm_diag_schwarz}.
\end{proof}

\begin{definition}\label{order_qd_def}
An unbounded simply connected quadrature domain that arises as the univalent image of $\widehat{\C}\setminus\overline{\mathbb{D}}$ under a rational map of degree $d+1$ is said to have \emph{order} $d$.
\end{definition}

\begin{figure}[ht!]
\begin{tikzpicture}
\node[anchor=south west,inner sep=0] at (5.6,4) {\includegraphics[width=0.25\textwidth]{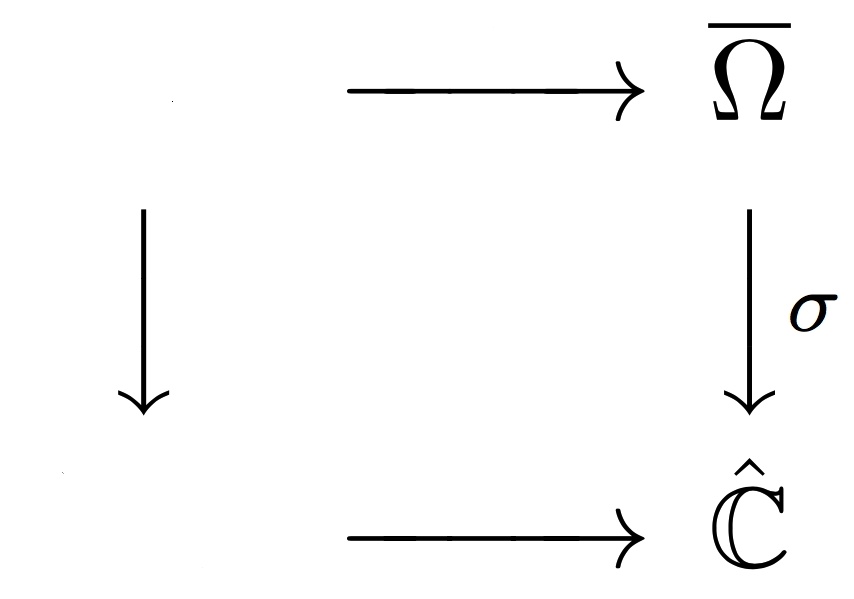}};
\node at (6.32,6.1) {$\widehat{\C}\setminus\overline{\D}$};
\node at (6.2,4.28) {$\D$};
\node at (6,5.2) {$\eta$};
\node at (7.6,6.4) {$f$};
\node at (7.6,4.6) {$f$};
\end{tikzpicture}
\caption{The rational map $f$ semi-conjugates the reflection map $\eta$ of $\D$ to the Schwarz reflection map 
$\sigma$ of $\Omega$.}
\label{comm_diag_schwarz}
\end{figure}

\subsection{Cusps and Double Points}\label{prelim_2}

\begin{definition} Let $\Omega\subset\mathbb{C}$ be an open set. A boundary point $p\in\partial\Omega$ is called \emph{regular} if there is a disc $B=B(p,\varepsilon):=\{z\in\C:\vert z-p\vert<\varepsilon\}$ such that $\Omega\cap B$ is a Jordan domain and $\partial\Omega\cap B$ is a simple non-singular real analytic arc; otherwise $p$ is a \emph{singular} point.
\end{definition}

Note that if the rational map $f:\widehat{\C}\setminus\overline{\D}\to\Omega$ (appearing in Proposition~\ref{s.c.q.d.}) has a critical point $\xi_0\in\mathbb{T}$, then the critical value $\zeta_0=f(\xi_0)$ is a singular point of $\partial\Omega$. We will refer to such a singular point as a \emph{cusp}. Unlike a double point singularity (see below), a cusp $\zeta_0$ has the property that for sufficiently small $\varepsilon>0$, the intersection $B(\zeta_0,\varepsilon)\cap\Omega$ is a Jordan domain. Moreover, by conformality of $f|_{\widehat{\C}\setminus\overline{\D}}$, each cusp on $\partial\Omega$ points in the inward direction towards $\Omega$.

We now define the \emph{type} of a cusp according to the Taylor series expansion of $f$ at $\xi_0$. For reasons that will be clear in Proposition~\ref{crit_points_on_circle}, we will only be concerned with the case where $\xi_0$ is a simple critical point of $f$. After an affine change of coordinates in the domain and the range (if necessary), we can assume that $\xi_0=1$, $\zeta_0=0$, and $$u(e^{it})=t^2+\cdots,\ v(e^{it})=ct^\nu+\cdots,$$ where $f=u+iv$, $\nu\geq3$ is an integer, $t\in(-\delta,\delta)$ for $\delta>0$ sufficiently small, and $c$ is a non-zero constant. By definition, we say then that the cusp at $\zeta_0=0$ is of the type $(\nu,2)$.

\begin{example} Here is an example of a $(\nu,2)$-cusp, with $\nu>3$, arising from a simple critical point of a rational map. Consider the map $f(z):=z^3-3z+2$, which is univalent on the disk $\{z\in\C:\vert z-4\vert<3\}$. Let us parametrize the boundary circle $\{z\in\C:\vert z-4\vert=3\}$ by $z(t)=4-3 e^{it}$, where $-\pi\leq t\leq\pi$. A direct computation now shows that $$\re(f(z(t)))=-27t^2+o(t^2),\ \im(f(z(t)))=-27t^5+o(t^5),\ \textrm{as}\ t\to 0.$$ Therefore, according to the above definition, the $f$-image of the circle $\{z\in\C:\vert z-4\vert=3\}$ (passing through the critical point $\xi_0=1$ of $f$) has a $(5,2)$ cusp at the corresponding critical value $\zeta_0=0$.
\end{example}

Since the domains $\Omega$ considered in this paper are univalent images of $\widehat{\C}\setminus\overline{\D}$ under a rational map, the only non-cusp singular points of $\partial\Omega$ are \emph{double points}: a point $\zeta_0\in\partial\Omega$ is said to be a \emph{double point} if for all sufficiently small $\varepsilon>0$, the intersection $B(\zeta_0,\varepsilon)\cap\Omega$ is a union of two Jordan domains, and $\zeta_0$ is a non-singular boundary point of each of them. In particular, two distinct non-singular (real-analytic) local branches of $\partial\Omega$ intersect tangentially at a double point $\zeta_0$. One can further classify such double points according to the order of contact of these two real-analytic branches.

\subsection{The Space $\Sigma_d^*$}\label{sigma_d_subsec}

In this Subsection we study the class $\Sigma_d^*$ (defined in the Introduction).

\begin{rem}\label{notation_remark} Classical objects in the field of geometric function theory are the classes: \begin{align} S := \{ f(z) = z +\sum_{n=2}^\infty a_nz^n : f \textrm{ is injective in } \mathbb{D}\}  \end{align} of normalized holomorphic functions univalent in the interior of the unit circle, and the related class \begin{align} \Sigma := \{ f(z) = z +\sum_{n=1}^\infty a_n/z^n : f \textrm{ is injective in } \widehat{\mathbb{C}}\setminus\overline{\mathbb{D}}\}  \end{align} of normalized holomorphic functions univalent in the exterior of the unit circle (see \cite[Section~4.7]{MR708494}). Those elements of $S$ defined by a finite sum are polynomials, and so we refer to those elements of $\Sigma$ defined by a finite sum as \emph{external polynomials}. The notations $S_d$ and $\Sigma_d$ classically denote the subspaces of $S$, $\Sigma$ (respectively) consisting of those elements with $a_n=0$, for $n > d$, and the $*$ notation specifies the normalization $-1/d$ of the coefficient for the $z^d$ or $z^{-d}$ term. 

\end{rem}

\noindent Given $f \in \Sigma_d^*$, \cite[Lemma 2.6]{2014arXiv1411.3415L} implies that \[f'\left(\frac1z\right)=z^{d+1}\overline{f'(\overline{z})}.\] A computation then yields that, if $d=2n+1$ is odd, \[ f(z)= z + \sum_{k=1}^{n-1} \frac{a_k}{z^k} + \frac{a_n}{z^n} +  \sum_{k=n+1}^{2n-1} \frac{(2n-k)\overline{a}_{2n-k}}{kz^{k}}-\frac{1}{dz^d}\textrm{ for } a_1, \cdots, a_{n-1}\in\mathbb{C} \textrm{ and }a_n\in\mathbb{R},\] and if $d=2n$ is even, \[ f(z)= z + \sum_{k=1}^{n-1} \frac{a_k}{z^k} + \sum_{k=n}^{2n-2} \frac{(2n-1-k)\overline{a}_{2n-1-k}}{kz^{k}} -  \frac{1}{dz^d} \textrm{ for } a_1, \cdots, a_{n-1}\in\mathbb{C}.\] It follows that the real dimension of $\Sigma_d^*$ is $d-2$.

Let $\omega$ be a primitive $(d+1)$-st root of unity, and $M_{\omega^i}$ be the affine map defined as $M_{\omega^i}(z)=\omega^i z$. There is a natural $\Z_{d+1}\cong\langle\omega\rangle$ action on $\Sigma_d^*$ given by $$f \mapsto M_{\omega^i}\circ f\circ M_{\omega^i}^{-1}.$$

\begin{prop}\label{crit_points_on_circle}
Let $f\in \Sigma_d^*$. Then, $f$ has $d+1$ distinct simple critical points on $\mathbb{T}$.
\end{prop}
\begin{proof}
Note that $f$ has a critical point of multiplicity $d-1$ at the origin. Since a rational map of degree $d+1$ has $2d$ critical points (counting multiplicity), it follows that $f$ has $d+1$ non-zero critical points (counting multiplicity). A straightforward computation using the derivative of $f$ and Vieta's formulas show that the absolute value of the product of all the non-zero critical points of $f$ is $1$. Since $f$ is conformal on $\widehat{\mathbb{C}}\setminus\overline{\mathbb{D}}$, $f$ cannot have a critical point in $\mathbb{C}\setminus\overline{\D}$. Hence, each non-zero critical point of $f$ must have absolute value $1$, and hence lies on $\mathbb{T}$. Conformality of $f\vert_{\widehat{\mathbb{C}}\setminus\overline{\mathbb{D}}}$ implies that each of these critical points of $f$ must be simple. Thus, $f$ has $d+1$ distinct simple critical points on $\mathbb{T}$.
\end{proof}

We will make use of a definition of conformal curvature from \cite{2014arXiv1411.3415L}:

\begin{definition}\label{conf_curv}
Let $f\in\Sigma_d^*$. Then, the \emph{conformal curvature} of $f(\mathbb{T})$ at $\zeta_0=f(e^{it_0})$ is defined as the product of the curvature of the curve $f(\mathbb{T})$ at $\zeta_0$ and the quantity $\vert f'(e^{it_0})\vert$.
\end{definition}

\begin{rem}\label{conf_curv_rem}
Definition \ref{conf_curv} will allow us to study the geometry of the boundary of an unbounded quadrature domain (arising from $\Sigma_d^*$) near its singular points. Observe that conformal curvature of $f(\mathbb{T})$ has the same sign as the usual curvature.
\end{rem}

\begin{prop}\label{cusp_geometry}
Let $f\in \Sigma_d^*$. Then, the $d+1$ cusps of $f(\mathbb{T})$ are of the type $(3,2)$.
\end{prop}
\begin{proof}
Assume that $\xi_0=e^{it_0}\in\mathbb{T}$ is a critical point of $f$, and let $\zeta_0=f(\xi_0)\in f(\mathbb{T})$ denote the resulting cusp. One readily calculates the formula \begin{equation}\label{standard_curvature} \frac{1}{|f'(\xi)|}\textrm{Re}\left( 1+ \frac{\xi f''(\xi)}{f'(\xi)} \right) \end{equation} for the standard curvature of $f(\mathbb{T})$ at the point $f(\xi)$ where $|\xi|=1$ and $f'(\xi)\not=0$. Using (\ref{standard_curvature}), one readily calculates that the conformal curvature is unchanged by an affine change of coordinates in the domain and range. Thus, we may assume that $t_0=0$, $\zeta_0=0$, and 
$$
u(e^{it})=t^2+o(t^2),\ v(e^{it})=ct^\nu+o(t^\nu),
$$ where $f=u+iv$, $t\in(-\delta,\delta)$ for some sufficiently small $\delta>0$, $c\not=0$, and $\nu\geq3$.

We will show that $\nu=3$. A straightforward computation, using the curvature formula \[ \kappa=\frac{u'v''-v'u''}{({u'}^2+{v'}^2)^{3/2}} \] and Definition~\ref{conf_curv}, shows that the conformal curvature of $f(\mathbb{T})$ near the singularity $\zeta_0$ is of the form $$c_1 t^{\nu-3}+o(t^{\nu-3})\ \textrm{ as } t\rightarrow0 \textrm{ where}\ c_1\neq0.$$ Supposing by way of contradiction that $\nu>3$, it is then clear that the conformal curvature of $f(\mathbb{T})$ tends to $0$ as $f(\mathbb{T})\ni\zeta\to\zeta_0$, and this contradicts \cite[Corollary~2.8]{2014arXiv1411.3415L}: the conformal curvature of $f(\mathbb{T})$ is constant, negative at non-cusp points of $f(\mathbb{T})$. Thus $\nu=3$.
\end{proof}

\begin{prop}\label{double_geometry}
Let $f\in \Sigma_d^*$, and $\zeta_0$ be a double point on $f(\mathbb{T})$. Then, the two distinct non-singular local branches of $f(\mathbb{T})$ have non-zero curvature and distinct osculating circles at $\zeta_0$ (in particular, they have a contact of order $1$). Moreover, in suitable local conformal coordinates near $\zeta_0$, the two distinct non-singular local branches of $f(\mathbb{T})$ are of the form $(u_\pm(t),v_\pm(t))=(t+o(t^2),c_\pm t^2+o(t^3))$, for some $c_\pm\in\C^*$ with $c_{+}\neq c_{-}$.
\end{prop}
\begin{proof}
By \cite[Corollary~2.8]{2014arXiv1411.3415L}, the curvature of $f(\mathbb{T})$ is negative at non-cusp points of $f(\mathbb{T})$. It follows that near a double point $\zeta_0$, two distinct non-singular local branches of $f(\mathbb{T})$ lie on opposite sides of their common tangent line $\ell$ at $\zeta_0$. Moreover, since the curvatures of these non-singular branches do not vanish at $\zeta_0$, it follows that the corresponding osculating circles are different from $\ell$. Thus, the osculating circles to these two non-singular branches at $\zeta_0$ lie on opposite sides of the tangent line $\ell$; i.e., the osculating circles do not coincide. In particular, two distinct non-singular branches of $f(\mathbb{T})$ have a contact of order $1$ at the double point $\zeta_0$.

For the last statement, let us first map the common tangent line $\ell$ to the $x$-axis by an affine change of coordinates (which preserves conformal curvature). Then, the two distinct non-singular local branches of $f(\mathbb{T})$ at $\zeta_0$ are of the form $$(u_\pm(t),v_\pm(t))=(t+o(t^2),c_\pm t^{\nu_\pm}+o(t^{\nu_\pm})),$$ where $c_\pm\in\C^*$, $\nu_\pm\geq 2$, and $t\in(-\delta,\delta)$ for some sufficiently small $\delta>0$. A simple computation now shows that the conformal curvatures of these branches near $\zeta_0$ are of the form $$c'_\pm t^{\nu_\pm-2}+o(t^{\nu_\pm-2})\ \textrm{ as } t\rightarrow0 \textrm{ where}\ c'_\pm\neq0.$$ Since these two branches have non-zero conformal curvature at $\zeta_0$, we must have that $\nu_\pm=2$. Finally, since they have a contact of order $1$, it follows that $c_{+}\neq c_{-}$.
\end{proof}

\subsection{Extremal Unbounded Quadrature Domains, and Suffridge Polynomials}\label{prelim_3}

\begin{definition}\label{extremal_qd_def}
An unbounded (simply connected) quadrature domain $\Omega$ of order $d$ is said to be \emph{extremal} if $\infty$ is the only pole of its Schwarz reflection map $\sigma$, and there are $d+1$ cusps and $d-2$ double points on $\partial\Omega$.
\end{definition}

\begin{rem}\label{extremal_rem}
The term \emph{extremal} in Definition~\ref{extremal_qd_def} is justified by \cite[Theorem~A1]{MR3454377} (see also of \cite[Lemma~2.4 ]{2014arXiv1411.3415L}): the boundary of an unbounded (simply connected) quadrature domain $\Omega$ of order $d$ can have at most $d+1$ cusps and $d-2$ double points.
\end{rem}

Two extremal unbounded quadrature domains $\Omega_1$ and $\Omega_2$ are said to be \emph{affinely equivalent} if there exists an affine map $A$ such that $A(\Omega_1)=\Omega_2$. This induces an equivalence relation on the space of extremal unbounded quadrature domains. Note that if an extremal unbounded quadrature domain can be mapped to another by an affine map $A$, the Schwarz reflection maps of the two domains are conjugate via $A$.

By Proposition~\ref{s.c.q.d.}, for every $f\in\Sigma_d^*$, the image $f(\widehat{\mathbb{C}}\setminus\overline{\mathbb{D}})$ is a simply connected unbounded quadrature domain. By Proposition~\ref{crit_points_on_circle}, the boundary of $f(\widehat{\mathbb{C}}\setminus\overline{\mathbb{D}})$ has $d+1$ cusps. Moreover, it follows from the commutative diagram in Figure~\ref{comm_diag_schwarz} that $\infty$ is the only pole of the Schwarz reflection map of $f(\widehat{\mathbb{C}}\setminus\overline{\mathbb{D}})$.

\begin{definition}\label{Suffridge_polynomials} $f \in \Sigma_d^*$ is said to be a \emph{Suffridge polynomial of degree} $d+1$ if the associated quadrature domain $f(\widehat{\mathbb{C}}\setminus\overline{\mathbb{D}})$ is extremal (i.e., $\partial f(\widehat{\C}\setminus\overline{\D})$ has $d-2$ double points). The curve $f(\mathbb{T})$ is called a \emph{Suffridge curve}.
\end{definition}

The set of all Suffridge polynomials of degree $d+1$ are preserved under the $\Z_{d+1}$ action defined in Subsection~\ref{sigma_d_subsec} giving rise to equivalence classes of Suffridge polynomials. The following simple proposition shows that Suffridge polynomials and extremal unbounded quadrature domains are in fact in bijective correspondence.

\begin{prop}\label{suffridge_extremal_qd_equiv_prop}
There is a bijective correspondence between equivalence classes of Suffridge polynomials of degree $d+1$ and affine equivalence classes of extremal unbounded quadrature domains of order $d$.
\end{prop}
\begin{proof}
By Proposition~\ref{s.c.q.d.}, the image of $\widehat{\C}\setminus\overline{\D}$ under a Suffridge polynomial of degree $d+1$ is an extremal unbounded quadrature domain of order $d$. Conjugating the Suffridge polynomial via multiplication by some $(d+1)$-st root of unity clearly produces an affinely equivalent extremal unbounded quadrature domain. Hence, we obtain a map from Suffridge polynomials of degree $d+1$ (up to conjugacy via multiplication by $(d+1)$-st roots of unity) to extremal unbounded quadrature domains of order $d$ (up to affine equivalence).

Let us now prove that this map is surjective. To this end, choose an extremal unbounded quadrature domain $\Omega$ of order $d$. By Proposition~\ref{s.c.q.d.}, there exists a rational map $g$ of degree $d+1$ such that $g:\widehat{\C}\setminus\overline{\D}\to\Omega$ is a conformal isomorphism. We can normalize $g$ (as a Riemann map) such that $g(\infty)=\infty$. Since $\partial\Omega$ has $d+1$ cusps, it follows that $g$ has $d+1$ critical points on $\mathbb{T}$. Moreover, the fact that $\infty$ is the only pole of the Schwarz reflection map $\sigma:=g\circ\eta\circ\left(g\vert_{\widehat{\C}\setminus\overline{\D}}\right)^{-1}:\Omega\to\widehat{\C}$ implies that $g(0)=\infty$, and $g$ has $d-1$ critical points at $0$. Since a rational map of degree $d+1$ has $2d$ critical points (counting multiplicity), it follows that $g$ has no other critical point. Setting $\Omega_1:=A_1(\Omega)$ and $g_1:=A_1\circ g$, where $A_1(z)=az+b$ (for some $a\in\C, b\in\C$), we can assume that $$g_1(z)= z+\frac{a_1}{z} + \cdots +\frac{a_d}{z^d}.$$ A simple computation (using the fact that all critical points of $g_1$ are on $\mathbb{T}$ and at the origin) now shows that $\vert a_d\vert=1/d$. Finally, setting $\Omega_2:=A_2(\Omega_1)$ and $g_2:=A_2\circ g_1\circ A_2^{-1}$, where $A_2(z)=\alpha z$ (for some $\alpha\in\mathbb{T}$), we can further assume that $g_2$ is of the form above with $a_d=-1/d$. Clearly, $g_2\in\Sigma_d^*$, and $\Omega_2=g_2(\widehat{\C}\setminus\overline{\D})$ is affinely equivalent to $\Omega$. In particular, $g_2(\mathbb{T})=\partial\Omega_2$ has $d+1$ cusps and $d-2$ double points; i.e., $g_2$ is the desired Suffridge polynomial of degree $d+1$. 

We now proceed to prove injectivity. Let $g_1$, $g_2$ be Suffridge polynomials of degree $d+1$, and $A$ an affine map such that $\Omega_1=g_1(\widehat{\C}\setminus\overline{\D})$, $\Omega_2=g_2(\widehat{\C}\setminus\overline{\D})$, and $A(\Omega_1)=\Omega_2$. Therefore, $$R:=\left(g_2\vert_{\widehat{\C}\setminus\overline{\D}}\right)^{-1}\circ A\circ g_1:\widehat{\C}\setminus\overline{\D}\to\widehat{\C}\setminus\overline{\D}$$ is a conformal map. Since $A$ is affine, it follows that $R$ fixes $\infty$, and thus is a rotation. Note that since $\left(g_2\vert_{\widehat{\C}\setminus\overline{\D}}\right)^{-1}(w)$ is of the form $w+O(\frac1w)$ near $\infty$, it follows that $A(w)= \alpha w$, for some $\alpha\in\C^*$. Therefore, $g_2=A\circ g_1\circ R^{-1}$. The fact that both $g_1$ and $g_2$ have derivative $1$ at $\infty$ implies that $A\equiv R$. Finally, as the coefficient in front of $1/z^d$ is $-\frac1d$ for both $g_1$ and $g_2$, it follows that $\alpha^{d+1}=1$. We conclude that $g_2=R\circ g_1\circ R^{-1}$, where $R$ is multiplication by a $(d+1)$-st root of unity.
\end{proof}

\begin{definition}\label{deltoid-like} A Jordan curve is \emph{deltoid-like} if it is smooth except for three outward cusps and if (at least) one of the arcs between the cusp points has negative curvature and the total variation of the tangent direction along this arc is less than $\pi$.
\end{definition}

According to \cite[Proposition~4.1]{2014arXiv1411.3415L}, the complement of an extremal unbounded quadrature domain $\Omega$ has $d-1$ interior components, each of which is a Jordan domain bounded by a deltoid-like curve.

\begin{rem}\label{krein_milman_rem} The space $\Sigma_d^*$ is a compact subset of the linear space of all rational functions of degree $d+1$ (given the topology of uniform convergence on compact subsets). We say $f\in\Sigma_d^*$ is an extreme point if $f$ has no representation of the form $f=tg+(1-t)h$ for $g$, $h \in \Sigma_d^*$ and $0<t<1$. By the Krein--Milman Theorem, $\Sigma_d^*$ is contained in the closed convex hull of its extreme points. It was proven in \cite[Theorem~3.6]{2014arXiv1411.3415L} that every extreme point of $\Sigma_d^*$ must be a Suffridge polynomial, however it is not known to us whether every Suffridge polynomial must be an extreme point of $\Sigma_d^*$. \end{rem}

\subsection{Bi-angled Trees}\label{prelim_4}

We now define bi-angled trees, and describe a procedure to assign a bi-angled tree to an extremal unbounded quadrature domain. We will usually denote a tree by $\mathcal{T}$, or $\mathcal{T}=(V,E)$ when we wish to emphasize the vertex set $V$, and the edge set $E$, of the tree $\mathcal{T}$. Given a vertex $v\in V$, we will call the degree of $v$, written $\textrm{deg}(v)$, as the number of edges in $E$ incident to $v$. 

\begin{definition}\label{tiles_def}
Let $\Omega$ be an extremal unbounded quadrature domain of order $d$, $T:=\mathbb{C}\setminus\Omega$ (called the \emph{droplet}), and $T^0:=T\setminus\{$Singular points on $\partial T\}$ (the desingularized droplet). The connected components $T_1, \cdots, T_{d-1}$ of $T^0$ are called \emph{fundamental tiles}. 
\end{definition}

\begin{rem}
The fundamental tiles are neither closed nor open.
\end{rem}

\begin{definition}\label{bi-angled_tree_def} 
A bi-angled tree $\mathcal{T}$ is defined as a tree each of whose vertices are of degree at most $3$, and that is equipped with an angle function $\angle$, defined on ordered pairs of edges incident at a common vertex, and taking values in $\{0, 2\pi/3, 4\pi/3\}$, satisfying the following conditions:
\begin{enumerate}\upshape
\item for each pair of distinct edges $e$ and $e'$ incident at a vertex $v$, we have $\angle_v(e,e')\in\{2\pi/3,4\pi/3\}$, while $\angle_v(e,e)=0$ for each edge $e$ incident at $v$,

\item $\angle_v(e,e')=-\angle_v(e',e)$ (mod $2\pi$), and 

\item $\angle_v(e,e')+\angle_v(e',e'')=\angle_v(e,e'')$ (mod $2\pi$), where $e, e',e''$ are edges incident at a vertex $v$.
\end{enumerate}
\end{definition}

\begin{rem}\label{planar_structure}
Since the function $\angle$ induces a cyclic order among the edges incident at each vertex, there is a preferred (isotopy class of) embedding of a bi-angled tree into the complex plane.
\end{rem}

\noindent We now introduce the notion of isomorphism of bi-angled trees (see also the definition of \emph{angled trees} as in \cite{Poi3}, \cite[Expos{\'e} VI, \S I.2]{orsay}), but first we recall the notion of a tree isomorphism.

\begin{definition}\label{tree_isomorphism} 
We say that two trees $\mathcal{T}_1=(V_1, E_1)$, $\mathcal{T}_2=(V_2, E_2)$ are \emph{isomorphic} if there exists a bijection $h: (V_1, E_1)\to (V_2, E_2)$ such that the vertices $v, w \in V_1$ are connected by an edge $e\in E_1$ if and only if the vertices $h(v)$, $h(w)\in V_2$ are connected by the edge $h(e)\in E_2$. 
\end{definition}

\begin{definition}\label{isom_def}
Two bi-angled trees $\mathcal{T}_1=(V_1, E_1)$ and $\mathcal{T}_2=(V_2, E_2)$ are said to be \emph{isomorphic} if there is a tree isomorphism $h:\mathcal{T}_1\to\mathcal{T}_2$ that is compatible with the corresponding angle functions; i.e., $\angle_{h(v)}(h(e),h(e'))=\angle_v(e,e')$, for each pair of distinct edges $e,e'\in E_1$ incident at a vertex $v\in V_1$. 
\end{definition}

\begin{rem}\label{bi-angled_equiv_def}
While the abstract Definitions~\ref{bi-angled_tree_def},~\ref{isom_def} will be used in Section~\ref{crofoot_sarason_sec} to establish the bijection between objects \eqref{bi-angled_trees} and \eqref{crofoot-sarason}, the following equivalent definition will suffice in Sections~\ref{mainthm_proof} and \ref{rigidity_qd_sec}.

Let $\mathcal{T}=(V,E)$ be a tree embedded in the plane. We say that $\mathcal{T}$ is \emph{bi-angled} if $\deg(v)\leq3$ for all $v \in V$, and $E$ consists of linear segments which meet at angles $2\pi/3$ or $4\pi/3$. In this setting, the angle function $\angle(e,e')$ is simply the Euclidean angle between $e$, $e'$ (measured counter-clockwise from $e$).

With this definition, two bi-angled trees $\mathcal{T}_1$, $\mathcal{T}_2$ are isomorphic if there exists a tree isomorphism $h: \mathcal{T}_1 \rightarrow \mathcal{T}_2$ with the property that if $e$, $e'$ are edges in $\mathcal{T}_1$ sharing a common endpoint, then the angle between $h(e)$, $h(e')$ (measured counterclockwise from $h(e)$) is equal to the angle between $e$, $e'$ (measured counterclockwise from $e$).

\end{rem}

\begin{definition}\label{associated_tree} For an extremal unbounded quadrature domain $\Omega$, we define a bi-angled tree $\mathcal{T}(\Omega)=(V_\Omega, E_\Omega)$ associated to $\Omega$ by the following recursive procedure (see also Figure \ref{fig:tree}). Denote by $T_1, \cdots, T_k$ the fundamental tiles of $\Omega$. Associate to the component $T_1$ a vertex at $0$, and denote by $c_1, c_2, c_3$ the cusp points of $\partial T_1$, ordered counter-clockwise. For $1\leq j \leq 3$, we include in $V_\Omega$ the point $\exp(i (2\pi j/3) )$ and in $E_\Omega$ the linear segment connecting $0$ to $\exp(i (2\pi j/3) )$ if and only if the cusp point $c_j$ is a boundary point of two fundamental tiles. This defines at most three new vertices and edges in $\mathcal{T}(\Omega)$. We perform a similar procedure for each new vertex, allowing for edges to decrease in length in successive steps of this procedure so as to avoid self-intersection. This recursively defines the bi-angled tree $\mathcal{T}(\Omega)$. 
\end{definition}

\begin{figure}
\centering
\scalebox{.12}{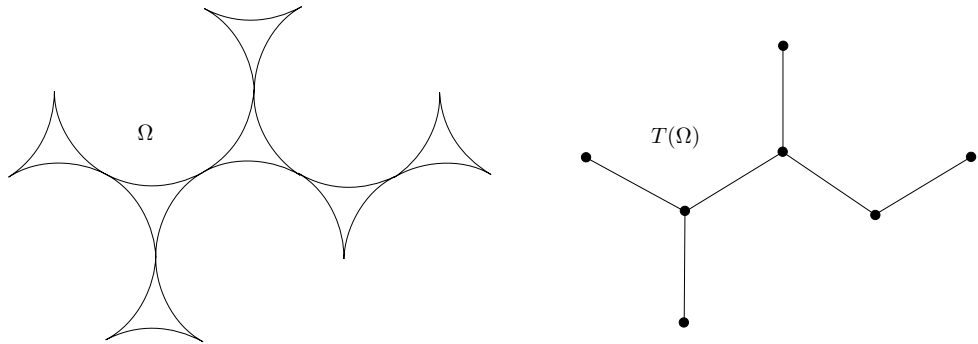}
\caption{Illustrated is Definition~\ref{associated_tree} in which a bi-angled tree $\mathcal{T}(\Omega)$ is associated to a given unbounded quadrature domain $\Omega$ by associating vertices to components of $\textrm{int}(\mathbb{C}\setminus\Omega)$ and connecting two vertices by an edge if and only if the corresponding components share a boundary point. We note that this figure is merely meant to illustrate Definition~\ref{associated_tree}: $\Omega$ was neither explicitly nor numerically computed.}
\label{fig:tree}      
\end{figure}

\subsection{Crofoot--Sarason Polynomials}\label{prelim_5}

\begin{definition}\label{sara_cro_def} Let $q(z)$ be a holomorphic polynomial of degree $d$. We say that $q$ is a \emph{Crofoot--Sarason polynomial} (\emph{CS polynomial} in short) if the $d-1$ critical points of $p(z):=\overline{q(z)}$ in $\C$ are distinct, and each critical point is fixed under $p$ (note that $p$ and $q$ have the same critical points). The anti-holomorphic polynomial $p$ is called a \emph{CS anti-polynomial}.
\end{definition}

\begin{rem}\label{CS_equiv_def_rem}
In \cite{KhSw}, CS polynomials were defined by the equivalent condition that all critical points of the polynomial (in $\C$) are distinct, and are mapped to their complex conjugates. It will be more convenient for us to use Definition~\ref{sara_cro_def} since the defining condition is dynamically natural. 
\end{rem}

The importance of CS polynomials stems from the fact that they realize the upper bound of the following theorem \cite{KhSw,Ge}.

\begin{thm}\label{KhSw} \cite{KhSw} Let $q(z)$, $\deg(q)=d>1$, be an analytic polynomial. Then 

\[ \#\left\{ z \in \mathbb{C} : \overline{q(z)}=z \right\} \leq 3d-2 \]

\end{thm}

\begin{rem}\label{lefschetz_count_rem}
While a complex polynomial of degree $d$ with no neutral fixed point has exactly $d$ distinct fixed points in the plane, the situation for anti-polynomials is quite different. In fact, it follows from an application of the Lefschetz-Hopf fixed point theorem that the number of planar fixed points of a degree $d$ anti-polynomial $p$ with no neutral fixed point depends on the number of planar \emph{attracting} fixed points of $p$ in the following way: \begin{equation}\label{hyperbolic_count}
\#\left\{ z \in \mathbb{C} : p(z)=z \right\}= 2\cdot \# \{ z \in \mathbb{C} : p(z)=z,\ \vert p'(z)\vert<1 \} + d.
\end{equation}
(See \cite[Lemma~6.1]{2014arXiv1411.3415L}.) The proof of Theorem~\ref{KhSw} combines (\ref{hyperbolic_count}) and Fatou's count of attracting fixed points (see \cite[\S 3.2]{CG1}). The formula (\ref{hyperbolic_count}) also explains why existence of CS polynomials demonstrates sharpness in the upper bound of Theorem \ref{KhSw} (see also \cite[Section~5]{KhSw}). Indeed, if $p=\overline{q}$ as in Definition \ref{sara_cro_def}, then the left hand-side of (\ref{hyperbolic_count}) is $3d-2$ as needed, where we note that there can be no neutral fixed point for $p$ as any neutral fixed point would necessarily imply the existence of an infinite critical orbit, but all critical points of $p$ are fixed.
\end{rem}

Following \cite[\S 1]{Ge}, let us now introduce a notion of equivalence for CS polynomials. 

\begin{definition}\label{sara_cro_equiv_def}
Two CS polynomials $q_1$ and $q_2$ (of the same degree) are said to be \emph{equivalent} if there exist affine maps $A_1$ and $A_2$ of the complex plane such that $A_1\circ q_1\circ A_2 = q_2$. 
\end{definition}

\begin{rem}\label{sara_cro_equiv_rem}
It is straightforward to check that the above definition produces an equivalence relation on the space of CS polynomials of a given degree.
\end{rem}

We conclude this subsection by interpreting the above equivalence relation on CS polynomials in terms of a conjugacy equivalence relation on CS anti-polynomials.

\begin{prop}\label{equiv_conj_sara_cro_prop}
Two CS polynomials $q_1$ and $q_2$ are equivalent if and only if the corresponding CS anti-polynomials $p_1$ and $p_2$ (respectively) are affinely conjugate.
\end{prop}
\textbf{Notation:} For an affine map $A$, we define $\tilde{A}(z):=\overline{A(\overline{z})}$. Clearly, $\tilde{A}$ is also an affine map.
\begin{proof}
Note that for $d=2$, the only CS polynomial (up to equivalence) is $z^2$, and the result is trivially true in this case. Thus, we will assume in the rest of the proof that $d\geq 3$. We will adopt the notation $\iota(z):=\overline{z}$ for this proof. 

Let us first suppose that $p_1(z)=\overline{q_1(z)}$ and $p_2(z)=\overline{q_2(z)}$ are conjugate via an affine map $A$; i.e., $A\circ p_1\circ A^{-1}=p_2$. Then, $\left(\iota\circ A\circ\iota\right)\circ\left(\iota\circ p_1\right)\circ A^{-1}=\iota\circ p_2$, and hence $\tilde{A}\circ q_1\circ A^{-1}= q_2$. It follows that $q_1$ and $q_2$ are equivalent in the sense of Definition~\ref{sara_cro_equiv_def}.

Conversely, let $A_1\circ q_1\circ A_2 = q_2$, where $A_1$ and $A_2$ are affine. Then, $\left(\iota\circ A_1\circ\iota\right)\circ\left(\iota\circ q_1\right)\circ A_2=\iota\circ q_2$; i.e., 
\begin{equation}
\tilde{A}_1\circ p_1\circ A_2=p_2. 
\label{equiv_rel_eqn}
\end{equation} 
It follows from \eqref{equiv_rel_eqn} that $c$ is a critical point of $p_1$ if and only if $A_2^{-1}(c)$ is a critical point of $p_2$. Moreover, since $p_1$ and $p_2$ fix each of their critical points, we also conclude from Relation~\eqref{equiv_rel_eqn} that $\tilde{A}_1(c)=A_2^{-1}(c)$, for every critical point $c$ of $p_1$. As $p_1$ has at least two distinct finite critical points, and $A_1, A_2$ are affine maps, we must have that $\tilde{A}_1\equiv A_2^{-1}$. Relation~\eqref{equiv_rel_eqn} now implies that $A_2^{-1}\circ p_1\circ A_2=p_2$; i.e., $p_1$ and $p_2$ are affinely conjugate.
\end{proof}

\subsection{Angled Hubbard Trees for CS Anti-polynomials}\label{prelim_5.5}

In this Subsection, we describe the construction of \emph{angled Hubbard trees} for CS anti-polynomials. This will be used in the proof of the Classification Theorem~\ref{crofoot_sarason_bi-angled_bijection_thm}.

Since each critical point of a CS anti-polynomial $p$ is fixed, it follows that each bounded Fatou component $U$ (of $p$) containing a critical point is invariant under $p$, and admits a \emph{B{\"o}ttcher coordinate} that conjugates $p\vert_U$ to the power map $\overline{z}^2\vert_{\D}$. The pre-images of radial lines in $\mathbb{D}$ under this B{\"o}ttcher coordinate are called \emph{internal rays} in $U$. Pulling the internal rays in $U$ back by iterates of $p$, we obtain internal rays in all bounded Fatou components of $p$. Note that since all critical points of $p$ (in $\C$) are fixed, the filled Julia set $K(p)$ (i.e., the set of all points in $\C$ that have bounded forward orbit under iteration of $p$) is connected (this follows, for instance, by applying \cite[Theorem~9.5]{MR2193309} on the holomorphic polynomial $p^{\circ 2}$). Moreover, the same property of $p$ also implies that $p^{\circ 2}$ is hyperbolic, and hence the filled Julia set of $p^{\circ 2}$, which is equal to $K(p)$, is a full continuum with locally connected boundary \cite[Lemma~9.4, Theorems~19.1, 19.2]{MR2193309}. Hence, there exists a unique arc (in $K(p)$) connecting any two points of $K(p)$ such that the intersection of the arc with every Fatou component is contained in the union of two internal rays (compare \cite[Expos{\'e}~II, \S 6, Proposition 6]{orsay}). Such arcs are called \emph{allowable}. The union of the allowable arcs connecting the post-critical set of $p$ in $\C$ (which is simply the set of all critical points of $p$ in $\C$) is a tree, and we call it the \emph{Hubbard tree} $\mathcal{T}(p)$ of the CS anti-polynomial $p$. The critical points of $p$ and the branch points of the tree $\mathcal{T}(p)$ are defined to be the vertices of the Hubbard tree (in fact, we will show later in Section~\ref{crofoot_sarason_sec} that all branch points of $\mathcal{T}(p)$ are critical points of $p$). 

Locally near each critical point $c$ (of $p$), the edges of $\mathcal{T}(p)$ meeting at $c$ are internal rays of the corresponding Fatou component, and we record the counter-clockwise angle between each pair of such internal rays measured in the B{\"o}ttcher coordinate. This endows the tree $\mathcal{T}(p)$ with an \emph{angle data}, and $\mathcal{T}(p)$ equipped with the angle data is called the \emph{angled Hubbard tree} of $p$. For a more detailed (and general) discussion of angled Hubbard trees, we refer the readers to \cite[Expos{\'e} VI]{orsay}, \cite{Poi3}.

We will also need an abstract counterpart of angled Hubbard trees. We collect the basic definitions of these objects here, and refer the readers to \cite[\S 5]{Poi3} for a more detailed account including a realization theorem for these abstract objects. 

Let $\mathcal{T}$ be a finite topological tree equipped with an angle function $\angle$ that assigns a number $\angle_v(e,e')\in\Q/2\pi\Z$ to each pair of edges $e$ and $e'$ incident at a vertex $v$. Suppose further that the angle function is skew-symmetric, non-degenerate (i.e., $\angle_v(e,e')=0$ if and only if $e=e'$), and additive. Then, the pair $(\mathcal{T},\angle)$ is called an \emph{angled tree}. Angled trees may be thought of as \emph{formal} angled Hubbard trees.

To model the dynamics of an anti-polynomial on its Hubbard tree, we need to define a notion of orientation-reversing angled tree maps. To this end, let 
$$
\textswab{d}:V(\mathcal{T})\to\N
$$ 
be a function, which we call a \emph{local degree} function, where $V(\mathcal{T})$ stands for the vertex set of $\mathcal{T}$. We caution the reader not to confuse $\textswab{d}(v)$ with the graph-theoretic degree of a vertex $v$, which is denoted by $\deg(v)$. A self-map $$\tau:\mathcal{T}\to\mathcal{T}$$ of the tree (which sends vertices to vertices, and is a homeomorphism of each edge onto its image) is called an \emph{orientation-reversing angled tree map} if it satisfies the condition 
$$
\angle_{\tau(v)}(\tau(e), \tau(e'))=-\textswab{d}(v)\angle_v(e,e')\ (\textrm{mod}\ 2\pi),
$$ 
for all edges $e, e'$ incident at $v\in V(\mathcal{T})$. The number $$d_\tau =1+\sum_{v\in V(\mathcal{T})}(\textswab{d}(v)-1)$$ is the \emph{degree} of the tree map $\tau$. We say that a vertex is \emph{critical} if $\textswab{d}(v)>1$. Finally, a vertex is said to be of \emph{Fatou type} if it maps to a periodic critical vertex under some iterate of $\tau$. It is a special case of the realization theorem \cite[Theorem~5.1]{Poi3} that an orientation-reversing angled tree map $\tau$ for which all vertices are of Fatou type can be realized as the dynamics associated to a postcritically finite anti-polynomial (i.e., an anti-polynomial whose critical points have finite forward orbits) of the same degree.

\subsection{Topological Quadrilaterals}\label{prelim_6}

We conclude this Section by summarizing some elementary facts about topological quadrilaterals that we will need in the course of the proof of Theorem~\ref{mainthm_qd_terminology} (for further details, see \cite[Chapter 1]{MR0344463}). 

\begin{definition} A \emph{quadrilateral} $Q(z_1, z_2, z_3, z_4)$ consists of a Jordan domain $Q\subset\widehat{\mathbb{C}}$ together with four distinct points $z_1, z_2, z_3, z_4 \in \partial Q$, oriented positively with respect to $Q$. The two subarcs $\arc{z_1z_2}$, $\arc{z_3z_4} \subset \partial Q$ are called the \emph{a-sides} of $Q$, and $\arc{z_2z_3}$, $\arc{z_4z_1}\subset \partial Q$ are called the \emph{b-sides} of $Q$. 
\end{definition}

\noindent Any topological quadrilateral $Q$ can be mapped conformally (preserving the four vertices) to some Euclidean rectangle $[0,a]\times[0,b]$ (see, for instance, \cite[\S I.2.4]{MR0344463}). The modulus $M(Q)$ of the quadrilateral is then defined to be $a/b$. Lastly, we recall a notion of convergence for quadrilaterals (see \cite[\S I.4.9]{MR0344463}). 

\begin{definition}\label{definition_of_quadrilateral_convergence} A sequence of quadrilaterals $Q_n$ (with a-sides $a_i^n$ and b-sides $b_i^n$, $i=1,2$, $n\in\mathbb{N}$) is said to converge to the quadrilateral $Q$ (with a-sides $a_i$ and b-sides $b_i$, $i=1,2$) if to every $\varepsilon>0$ there corresponds an $n_{\varepsilon}$ such that for $n\geq n_{\varepsilon}$, every point of $a_i^n$, $b_i^n$, $i=1,2$, and every interior point of $Q_n$ has a spherical distance of at most $\varepsilon$ from $a_i$, $b_i$, and $Q$, respectively. 
\end{definition}

\begin{thm}\label{convergence_of_quadrilaterals} If a sequence of quadrilaterals $Q_n$ converges to the quadrilateral $Q$, then \[ \lim_{n\rightarrow\infty} M(Q_n)=M(Q). \]
\end{thm}

\noindent Lastly, inserting an appropriate mass function into the extremal length characterization of modulus yields the following theorem (see \cite[\S I.4.4]{MR0344463}):

\begin{thm}\label{modulus_estimate} Let $Q$ be a quadrilateral. Then: \begin{equation}\nonumber \frac{1}{\pi}\frac{(\log(1+2s_b/s_a))^2}{1+2\log(1+2s_b/s_a)} \leq M(Q) \leq \pi\frac{1+2\log(1+2s_a/s_b)}{(\log(1+2s_a/s_b))^2}, \end{equation} where $s_a$, $s_b$ denote the Euclidean path distance in $Q$ between the a-sides, b-sides, respectively, of the topological quadrilateral $Q$.
\end{thm}

\section{Dynamics of Schwarz Reflections Arising from $\Sigma_d^*$}\label{general_dynamics_sec}

The function $f$ of Theorem~\ref{mainthm_qd_terminology} is obtained via a sequence of quasiconformal deformations applied to a canonical map in $\Sigma_d^*$. Subsection~\ref{Exteriors} introduces this canonical map, and Subsection~\ref{Dynamical_Partition} discusses the dynamical plane of a Schwarz reflection map associated to a function in $\Sigma_d^*$.

\subsection{Exteriors of Hypocycloids are Quadrature Domains}\label{Exteriors}

\begin{prop}\label{hypocycloid_quadrature} The map $f_0(z):=z-1/dz^d$ is injective on $\widehat{\mathbb{C}}\setminus\mathbb{D}$. In particular, $f_0\in\Sigma_d^*$.
\end{prop}

\begin{proof} 
Note that $f_0$ has a $(d-1)$-fold critical point at the origin, and $d+1$ simple critical points on $\mathbb{T}$. In particular, $f_0$ has no critical point in $\widehat{\mathbb{C}}\setminus\overline{\mathbb{D}}$. So it suffices to show that $f_0(\mathbb{T})$ is a Jordan curve (indeed, this implies that $f_0(\widehat{\mathbb{C}}\setminus\overline{\mathbb{D}})$ is simply connected, and hence uniqueness of analytic continuation yields an inverse branch of $f_0$ that maps $f_0(\widehat{\mathbb{C}}\setminus\overline{\mathbb{D}})$ onto $\widehat{\mathbb{C}}\setminus\overline{\mathbb{D}}$).

We now prove injectivity of $f_0\vert_{\mathbb{T}}$. To this end, let $z,w\in\mathbb{T}$ and note that:
\begin{align*}
f_0(z)=f_0(w) \implies \displaystyle\sum_{j=0}^{d-1} z^{d-1-j}w^j=-dz^dw^d\implies \left\vert \displaystyle \sum_{j=0}^{d-1}z^{d-1-j}w^j\right\vert=d.
\end{align*}
By the triangle inequality and the fact that $z,w\in\mathbb{T}$, we now conclude that all the complex numbers $z^{d-1-j}w^j$ (for $j\in\{0, \cdots, d-1\}$) have the same argument. But this implies that $z=w$, so that $f_0$ is injective on $\mathbb{T}$, and hence $f_0(\mathbb{T})$ is a Jordan curve.
\end{proof}

The Jordan curve $f_0(\mathbb{T})$ is a so-called hypocycloid curve (see Figure~\ref{fig:unit_disc}). Since $f_0$ commutes with multiplication by the $(d+1)-$st roots of unity, it follows that $\Omega_0:=f_0(\widehat{\C}\setminus\overline{\D})$ is symmetric under rotation by $\frac{2\pi}{(d+1)}$. Moreover, the $d+1$ simple critical points of $f_0$ on $\mathbb{T}$ produce $d+1$ cusps  of the type $(3,2)$ on the boundary $\partial\Omega_0$. 

\subsection{Dynamical Partition for Schwarz Reflections Arising from $\Sigma_d^*$}\label{Dynamical_Partition}

We will now study some basic dynamical properties of Schwarz reflection maps associated with an arbitrary $f\in\Sigma_d^*$. Recall that for such an $f$, the domain $\Omega:=f(\widehat{\C}\setminus\overline{\D})$ is an unbounded quadrature domain with associated Schwarz reflection map \[\sigma=f\circ\eta\circ\left(f\vert_{\widehat{\C}\setminus\overline{\D}}\right)^{-1}.\] 

The map $\sigma$ has a $d$-fold pole at $\infty$, and no other critical point in $\Omega$. It also follows from Proposition~\ref{s.c.q.d.} that $\sigma:\sigma^{-1}(\Omega)\to\Omega$ is a proper branched covering map of degree $d$ (branched only at $\infty$), and $\sigma:\sigma^{-1}(\interior{\Omega^c})\to\interior{\Omega^c}$ is a degree $d+1$ covering map.

As in Section~\ref{preliminaries}, we define $T=\widehat{\mathbb{C}}\setminus \Omega$, $T^0=T\setminus\{$The singular points on $\partial T\}$, and $$T^\infty(\sigma):=\bigcup_{n\geq0} \sigma^{-n}(T^0).$$ We will call $T^\infty(\sigma)$ the \emph{tiling set} of $\sigma$. For any $n\geq0$, the connected components of $\sigma^{-n}(T^0)$ are called \emph{tiles} of rank $n$. Note that two distinct tiles have disjoint interior. Let us denote the union of the tiles of rank $\leq k$ by $E^k(\sigma)$.

Observe furthermore that $\infty$ is a super-attracting fixed point of $\sigma$: more precisely, $\infty$ is a fixed critical point of $\sigma$ of multiplicity $d-1$. We denote the basin of attraction of $\infty$ by $\mathcal{B}_\infty(\sigma)$. Clearly, $$\mathcal{B}_\infty(\sigma)\subset \widehat\C\setminus\overline{T^\infty(\sigma)}.$$

The next proposition discusses some basic topological properties of the tiling set, and the basin of infinity of $\sigma$.

\begin{prop}\label{basin_topology}
1) The tiling set $T^\infty(\sigma)$ is open. Its closure $\overline{T^\infty(\sigma)}$ is a compact, connected set.

2) The basin of infinity $\mathcal{B}_\infty(\sigma)$ is a simply connected, completely invariant domain.
\end{prop}

\begin{figure}
\centering
\scalebox{.075}{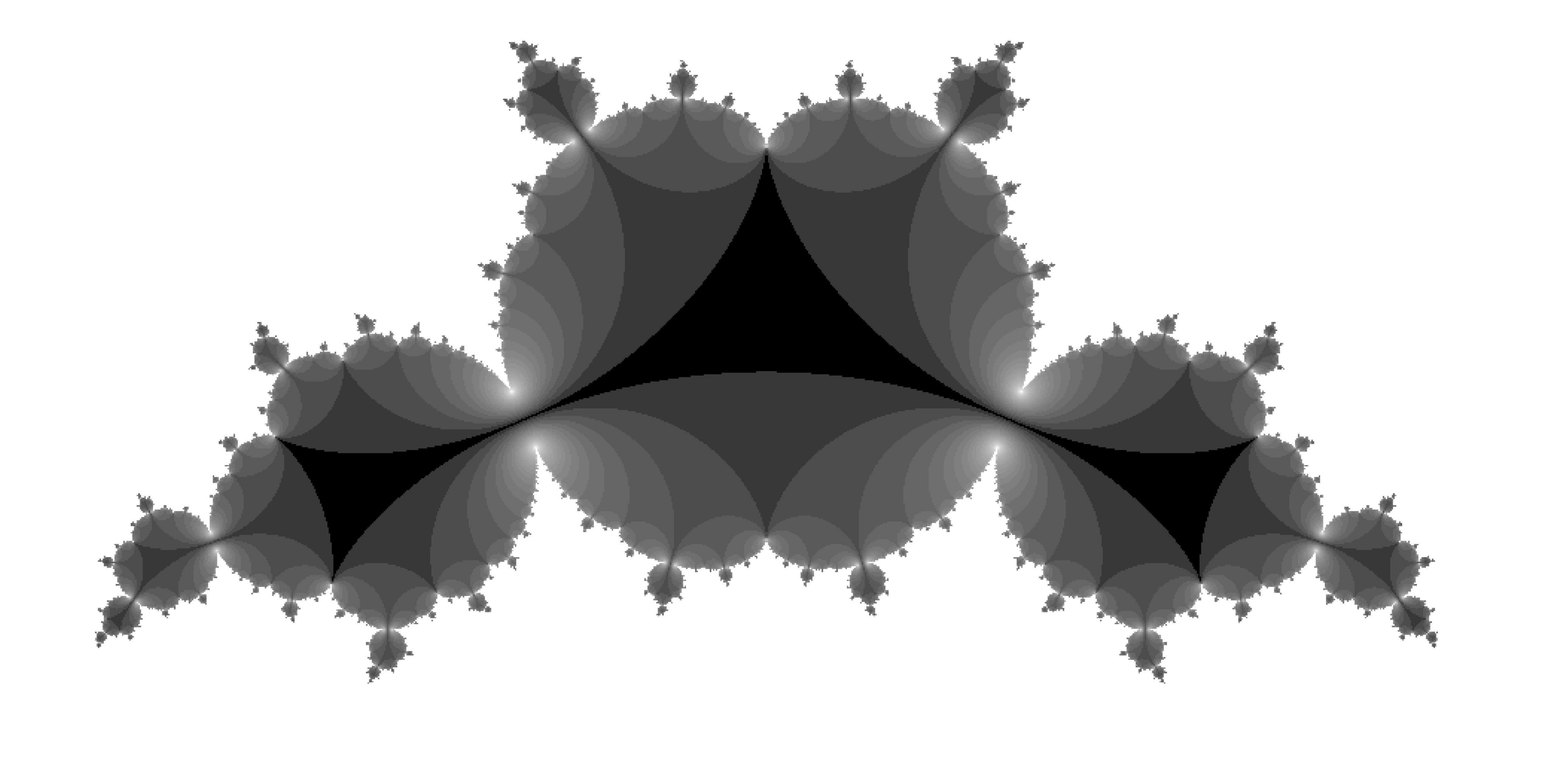}
\caption{ Pictured is the dynamical plane of the Schwarz reflection map arising from a Suffridge polynomial in $\Sigma_4^*$. The numbers denote the ranks of various tiles. }
\label{tiles_various_ranks_fig}
\end{figure}

\begin{proof}
1) If $z\in T^\infty(\sigma)$ belongs to the interior of a tile, then it clearly belongs to $\interior{T^\infty(\sigma)}$. On the other hand, if $z\in T^\infty(\sigma)$ belongs to the boundary of a tile of rank $k$ (here the boundary is taken in the relative topology of $T^\infty(\sigma)$), then $z$ lies in $\interior{E^{k+1}(\sigma)}\subset\interior{T^\infty(\sigma)}$ (see Figure~\ref{tiles_various_ranks_fig}). Hence, $T^\infty(\sigma)$ is open.

To see that $\overline{T^\infty(\sigma)}$ is connected, first note that $$\overline{T^\infty(\sigma)}=\overline{\bigcup_{k\geq0}\overline{E^k(\sigma)}}.$$ Now, $\overline{E^k(\sigma)}$ is a connected set for each $k\geq0$, and hence so is their increasing union $\bigcup_{k\geq0}\overline{E^k(\sigma)}$ (see Figure~\ref{tiles_various_ranks_fig}). Since the closure of a connected set is connected, we conclude that $\overline{T^\infty(\sigma)}$ is connected.

2) Since $\mathcal{B}_\infty(\sigma)$ is the basin of attraction of a super-attracting fixed point, it is necessarily open and completely invariant. 

Let $\mathcal{B}^\textrm{imm}_\infty(\sigma)$ be the connected component of $\mathcal{B}_\infty(\sigma)$ containing $\infty$ (this is usually called the immediate basin of attraction, which justifies the superscript `imm'). If $\mathcal{B}_\infty(\sigma)\setminus \mathcal{B}^\textrm{imm}_\infty(\sigma)\neq\emptyset$, then every connected component of $\mathcal{B}_\infty(\sigma)\setminus \mathcal{B}^\textrm{imm}_\infty(\sigma)$ would eventually map onto $\mathcal{B}^\textrm{imm}_\infty(\sigma)$ under some iterate of $\sigma$. Therefore, every connected component of $\mathcal{B}_\infty(\sigma)\setminus \mathcal{B}^\textrm{imm}_\infty(\sigma)$ must contain an iterated pre-image of $\infty$. But $\sigma^{-1}(\infty)=\{\infty\}$. Hence $\mathcal{B}^\textrm{imm}_\infty(\sigma)$ must be equal to $\mathcal{B}_\infty(\sigma)$. This shows that $\mathcal{B}_\infty(\sigma)$ is connected and completely invariant.

It remains to prove simple connectivity of $\mathcal{B}_\infty(\sigma)$. Let $U\subset \mathcal{B}_\infty(\sigma)$ be a small neighborhood of $\infty$. Clearly, $\mathcal{B}_\infty(\sigma)$ is the increasing union of the domains $\left\{\sigma^{-k}(U)\right\}_{k\geq0}$. Since $\infty$ is the only critical point of $\sigma$, it follows from the Riemann-Hurwitz formula that each $\sigma^{-k}(U)$ is simply connected. Thus, $\mathcal{B}_\infty(\sigma)$ is an increasing union of simply connected domains, and hence itself is such.
\end{proof}

We will now prove a proposition that will allow us to talk about quasiconformal deformations of the Schwarz reflection maps associated with quadrature domains arising from $\Sigma_d^*$. 

\begin{prop}\label{qc_def_prop}
Let $g\in\Sigma_d^*$, $\Omega:=g(\widehat{\C}\setminus\overline{\D})$, and $\sigma$ the Schwarz reflection map of $\Omega$. Further, let $\mu$ be a $\sigma$-invariant Beltrami coefficient on $\widehat{\C}$, and $\mathbf{\Phi}:(\widehat{\mathbb{C}},\infty)\rightarrow(\widehat{\mathbb{C}},\infty)$ be any quasiconformal map satisfying $\mathbf{\Phi}_{\overline{z}}/\mathbf{\Phi}_{z}=\mu$ a.e.. Then $\mathbf{\Phi}(\Omega)$ is a simply connected unbounded quadrature domain. There exists a normalization for $\mathbf{\Phi}$ (specified in the proof) with which we have $\mathbf{\Phi}(\Omega)=h(\widehat{\C}\setminus\overline{\D})$ for some $h\in\Sigma_d^*$, and $\mathbf{\Phi}\circ\sigma\circ\mathbf{\Phi}^{-1}$ is the Schwarz reflection map of $\mathbf{\Phi}(\Omega)$.
\end{prop}
\begin{proof}
Let $\mathbf{\Phi}_2$ be a quasiconformal map integrating $\mu$ satisfying $\mathbf{\Phi}_2(\infty)=\infty$. To prove that $\widehat{\Omega}:=\mathbf{\Phi}_2(\Omega)$ is an unbounded quadrature domain, it suffices to show that 
\begin{equation*} \widehat{\sigma}:=\mathbf{\Phi}_2 \circ \sigma \circ \mathbf{\Phi}_2^{-1} : \overline{\widehat{\Omega}} \rightarrow \widehat{\mathbb{C}} \end{equation*}
\noindent defines a Schwarz reflection map for the domain $\widehat{\Omega}$. Indeed, $\widehat{\sigma}$ is anti-meromorphic (on $\widehat{\Omega}$) by invariance of $\mu$ under the action of $\sigma$, and that $\widehat{\sigma}$ is the identity on the boundary of $\widehat{\Omega}$ follows since $\sigma|_{\partial \Omega}=\textrm{id}$. Since $\mathbf{\Phi}_2(\infty)=\infty$, we see that $\widehat{\Omega}$ is a simply connected unbounded quadrature domain.

Since $g\in\Sigma_d^*$, it follows that $\sigma=g\circ\eta\circ (g\vert_{\widehat{\C}\setminus\overline{\D}})^{-1}$ has a pole of order $d$ at $\infty$. As $\mathbf{\Phi}_2$ fixes $\infty$, the Schwarz reflection map $\widehat{\sigma}$ of $\widehat{\Omega}$ also has a pole of order $d$ at $\infty$.

The fact that $\Omega$ has order $d$, combined with Proposition~\ref{s.c.q.d.}, implies that $\widehat{\sigma}:\widehat{\sigma}^{-1}(\widehat{\Omega})\to\widehat{\Omega}$ is a branched cover of degree $d$. The same proposition now provides us with a rational map $h_2$ of degree $d+1$ such that $h_2(\widehat{\C}\setminus\overline{\D})=\widehat{\Omega}$, and $h_2$ is conformal on $\widehat{\C}\setminus\overline{\D}$. We may normalize $h_2$ so that $h_2(\infty)=\infty$. Since the Schwarz reflection map $\widehat{\sigma}$ of $\widehat{\Omega}$ has a pole of order $d$ at $\infty$, it follows from the commutative diagram in Figure~\ref{comm_diag_schwarz} that $h_2$ has a pole of order $d$ at the origin. Hence, $h_2$ is of the form 
\[h_2(z)=c_1z+c_0+\frac{c_{1}}{z}+\cdots+\frac{c_{d}}{z^d}. \] 
We now define an affine map $A: w\mapsto (w-c_0)/c_1$, and set $\mathbf{\Phi}_1:=A\circ\mathbf{\Phi}_2$. Then, $h_1:=A\circ h_2$ uniformizes the quadrature domain $\mathbf{\Phi}_1(\Omega)=A(\widehat{\Omega})$, and we have
\[h_1(z)=z+\frac{b_{1}}{z}+\cdots+\frac{b_{d}}{z^d}. \] 
Now recall that by Proposition~\ref{crit_points_on_circle}, the map $g$ has $d+1$ distinct critical points on $\mathbb{T}$, and hence $\partial\Omega$ has $d+1$ cusps. Since quasiconformal maps preserve cusps, $h_1$ must have $d+1$ critical points on $\partial\mathbb{D}$. The only remaining critical point is of multiplicity $d-1$ at the origin. Thus by considering the equation $h_1'(z)=0$, we see that since each non-zero critical point of $h_1$ has modulus $1$, it must be the case that $|b_{d}|=1/d$. Finally, we define a rotation $R: w\mapsto \lambda w$ (where $\lambda^{d+1}=-b_{d}/\vert b_{d}\vert$), and set $\mathbf{\Phi}:=R\circ\mathbf{\Phi}_1$, $h:=R\circ h_1\circ R^{-1}$. Then, $h(\widehat{\C}\setminus\overline{\D})=\mathbf{\Phi}(\Omega)$, where 
\[h(z)=z+\frac{a_{1}}{z}+\cdots-\frac{1}{d\cdot z^d}. \]
Hence, $h\in\Sigma_d^*$, and $\mathbf{\Phi}\circ\sigma\circ\mathbf{\Phi}^{-1}$ is the Schwarz reflection map of $\mathbf{\Phi}(\Omega)$.
\end{proof}

We will now complete the dynamical description of the Schwarz reflection map $\sigma$ by showing that the forward orbit of all points outside the closure of the tiling set of $\sigma$ converges to $\infty$.

\begin{prop}\label{dynamical_partition_schwarz} 
$\widehat{\C}=\overline{T^\infty(\sigma)}\sqcup\mathcal{B}_\infty(\sigma)$.
\end{prop}
\begin{proof}
The classical proof of non-existence of wandering Fatou components for rational maps \cite[Theorem~1]{Sul} can be adapted for the current setting to show that each component of $\widehat{\C}\setminus\overline{T^\infty(\sigma)}$ is eventually periodic. Indeed, if $\sigma$ were to have a wandering Fatou component (since $\overline{T^\infty(\sigma)}$ is connected, such a component would necessarily be simply connected), then one could construct an infinite-dimensional space of nonequivalent quasiconformal deformations of $\sigma$ (supported on the grand orbit of the wandering component), all of which would be Schwarz reflection maps associated with quadrature domains of the form $h(\widehat{\C}\setminus\overline{\D})$, where $h\in\Sigma_d^*$ (by Proposition~\ref{qc_def_prop}). Evidently, this would contradict the fact that the parameter space of $\Sigma_d^*$ and hence the parameter space of the resulting Schwarz reflection maps is finite-dimensional.

One can now use the arguments of \cite[Propositions~6.25, 6.26]{LLMM1} combined with the fact that $\sigma$ has no critical point other than the super-attracting fixed point $\infty$ to rule out the existence of bounded complementary components of $\overline{T^\infty(\sigma)}$. More precisely, such arguments would show that each periodic component of $\widehat{\C}\setminus\overline{T^\infty(\sigma)}$ other than $\mathcal{B}_\infty(\sigma)$ is either the immediate basin of attraction of a (super-)attracting or parabolic cycle, or a Siegel disk. However, the closure of each such component must intersect the closure of the post-critical set of $\sigma$. This is clearly impossible as the only critical point of $\sigma$ is in fact fixed, and lies in $\mathcal{B}_\infty(\sigma)$.

Therefore, $\widehat\C\setminus\overline{T^\infty(\sigma)}\ni\infty$ is an invariant domain, and hence $\left\{\sigma^{\circ n}\right\}_n$ forms a normal family there. It follows that the sequence of functions $\left\{\sigma^{\circ n}\right\}_n$ converges locally uniformly to $\infty$ on $\widehat\C\setminus\overline{T^\infty(\sigma)}$; i.e., $\mathcal{B}_\infty(\sigma)=\widehat\C\setminus\overline{T^\infty(\sigma)}$.
\end{proof}

\begin{cor}\label{tiling_set_full}
$\overline{T^\infty(\sigma)}$ is a full continuum.
\end{cor}

\begin{rem}
By Proposition~\ref{cusp_geometry} (respectively, Proposition~\ref{double_geometry}), the cusps (respectively, the double points) on $\partial\Omega$ are of the type $(3,2)$ (respectively, are intersection points of two distinct non-singular branches of $\partial\Omega$ with a contact of order $1$). Using this, one can apply the arguments of \cite[Proposition~6.17]{LLMM1} to show that at a cusp (respectively, at a double point) on $\partial\Omega$, the tiling set $T^\infty(\sigma)$ contains a circular sector with angle $2(\pi-\delta)$ (respectively, two circular sectors with angle $(\pi-\delta)$ each), for some $\delta>0$ small enough (see Figure~\ref{tiles_various_ranks_fig}). Moreover, the singular points on $\partial\Omega$ repel nearby points in $\widehat{\C}\setminus\overline{T^\infty(\sigma)}$ under iterates of $\sigma$. Since these facts will not be used in this paper, we skip the proofs here.
\end{rem}

Let us denote the boundary of $T^\infty(\sigma)$ by $\mathcal{L}(\sigma)$, and call it the \emph{limit set} of $\sigma$. We end this subsection with a result on the area of the limit set that will be useful later.

\begin{prop}\label{limit_schwarz_zero_area}
The limit set $\mathcal{L}(\sigma)$ has zero area.
\end{prop}
\begin{proof}
The proof is similar to that of \cite[Corollary~7.3]{LLMM1}. Indeed, since $\mathcal{L}(\sigma)$ is nowhere dense, one can use a standard `hyperbolic zoom' argument (which first appeared in \cite{Lyu83}, \cite[Expos{\'e}~V]{orsay}) to conclude that no point on $\mathcal{L}(\sigma)$, except possibly the countably many points that eventually land on the singular points on $\partial T$, is a point of Lebesgue density for $\mathcal{L}(\sigma)$. It now follows from Lebesgue's density theorem that $\mathcal{L}(\sigma)$ has zero area.
\end{proof}

\section{Extremal Quadrature Domains from Bi-angled Trees}\label{mainthm_proof}

In this section we consider the problem of realizing extremal unbounded quadrature domains with prescribed bi-angled tree structure (see Definition~\ref{associated_tree}). In particular we prove the following:

\begin{thm}\label{mainthm_qd_terminology} 
Let $\mathcal{T}$ be a bi-angled tree. There exists an extremal unbounded quadrature domain $\Omega$ whose associated bi-angled tree is isomorphic to $\mathcal{T}$.
\end{thm}

\begin{rem}\label{thm_equiv_1_rem} By Proposition~\ref{suffridge_extremal_qd_equiv_prop}, Theorem~\ref{mainthm_qd_terminology} is equivalent to surjectivity of the map (1)$\mapsto$(2) in the statement of Theorem~\ref{theorem_A}. \end{rem}

\noindent There are several principles we will use repeatedly in the proofs of Section~\ref{mainthm_proof}, and it will be convenient for us to formalize them here. We will refer to Figure \ref{fig:constant_ratio_principle} for Principle \ref{constant_ratio_principle}.

\begin{principle}\label{constant_curvature_principle} {\bf (Constant curvature principle)} If $f\in\Sigma_d^*$, the curve $f(\mathbb{T})$ has constant conformal curvature (except for cusps). In particular the sign of the usual curvature of $f(\mathbb{T})$ is constant (except for cusps).
\end{principle}

\begin{proof} See \cite[Corollary~2.8]{2014arXiv1411.3415L}.
\end{proof}

\begin{principle}\label{constant_ratio_principle} {\bf (Constant ratio principle)} Let $R=[0,a]\times[0,b]$ be a Euclidean rectangle. Let $z_1$, $z_2 \in \partial R$ be points lying on a vertical boundary segment of $\partial R$, and let $w_1$, $w_2 \in \partial R$ lie on a common linear segment of $\partial R$ such that the vertical segment connecting $z_1$, $z_2$ does not contain either of $w_1$, $w_2$. Assume $z_1, z_2, w_1, w_2$ have been labeled so that $(z_1, z_2, w_1, w_2)$ is oriented positively with respect to $R$. Define $L_t:\mathbb{C}\rightarrow\mathbb{C}$ by $L_t(x,y):=(x, y/t)$ for $t>1$, and consider the quadrilateral $Q:=R(z_1, z_2, w_1, w_2)$. Then the modulus of the quadrilateral $L_t(Q)$ is bounded away from $\infty$ independently of $t$. 
\end{principle}

\begin{proof} Let $s_a(t)$, $s_b(t)$, as in Theorem \ref{modulus_estimate}, denote the Euclidean path-distance between the a-sides, b-sides (respectively) of the quadrilateral $L_t(Q)$. Note that if $w_1$ lies on the vertical segment containing $z_1$, $z_2$, then $s_a(t)=\min\{ t|z_2-w_1|, t|z_1-w_2| \}$, $s_b(t)=\min\{ t|z_1-z_2|, t|w_1-w_2| \}$ and so $s_a(t)/s_b(t)$ is independent of $t$. Otherwise, $s_a(t) \not \rightarrow 0$ but $s_b(t)\rightarrow 0$ as $t\rightarrow\infty$. In either case, Theorem \ref{modulus_estimate} applied to $L_t(Q)$ yields the desired bound (and in fact in the second case, Theorem \ref{modulus_estimate} yields that $M(L_t(Q))\rightarrow0$ as $t\rightarrow\infty$). 

\begin{figure}
\centering
\scalebox{.5}{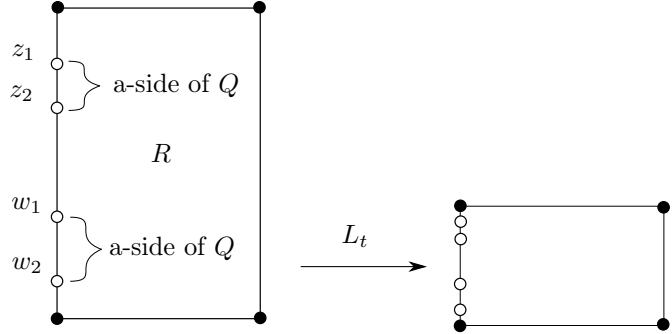}
\caption{ Illustrated is the Euclidean rectangle $R$ and the quadrilateral $Q=R(z_1, z_2, w_1, w_2)$ of Principle \ref{constant_ratio_principle}. The a-sides of the quadrilateral $Q$ have been labeled. }
\label{fig:constant_ratio_principle}      
\end{figure}
\end{proof}

\subsection{The Astroid Case}\label{astroid}

In this subsection, we prove Theorem~\ref{mainthm_qd_terminology} in the case of the astroid (corresponding to the class $\Sigma_3^*$). In fact an explicit formula may be given for the rational map corresponding to this extremal quadrature domain (see \cite[\S 2.2]{2014arXiv1411.3415L}), but our goal is rather to illustrate the main arguments used to prove Theorem~\ref{mainthm_qd_terminology} in this simplified setting. The argument is illustrated in Figure~\ref{fig:astroid}. We will use the following notation throughout this section: \[ f_0(z)=z-\frac{1}{3z^3}\textrm{, }\Omega_0=f_0(\widehat{\mathbb{C}}\setminus\overline{\D}),\ \sigma_0:=\ \textrm{Schwarz reflection map of}\ \Omega_0,\] \[Q:=\widehat{\mathbb{C}}\setminus\Omega_0, \textrm{and}\ Q^0:=Q\setminus\{\textrm{Singular points on}\ \partial Q\}. \] 

\begin{prop}\label{square} There exists a conformal map:

\[ \Psi: Q  \rightarrow [-1,1]\times[-1,1] \] 

\noindent mapping the critical values of $f_0$ to the four vertices of $[-1,1]\times[-1,1]$.

\end{prop}

\begin{proof} Consider $Q$ as a topological quadrilateral with vertices at the four critical values of $f_0$, and horizontal a-sides and vertical b-sides. Since $f_0(\overline{z})=\overline{f_0(z)}$ and $f_0(i\overline{z})=i\overline{f_0(z)}$, the map $z\rightarrow iz$ maps $Q$ conformally to the quadrilateral $Q$ with reversed a-sides and b-sides, so that $M(Q)=1/M(Q)$, and hence $M(Q)=1$. Any topological quadrilateral can always be conformally mapped to some Euclidean rectangle such that vertices are preserved (see, for instance, \cite[\S I.2.4]{MR0344463}), and since $M(Q)=1$, the claim follows.  
\end{proof}

\begin{figure}
\centering
\scalebox{.36}{\hspace{-25mm}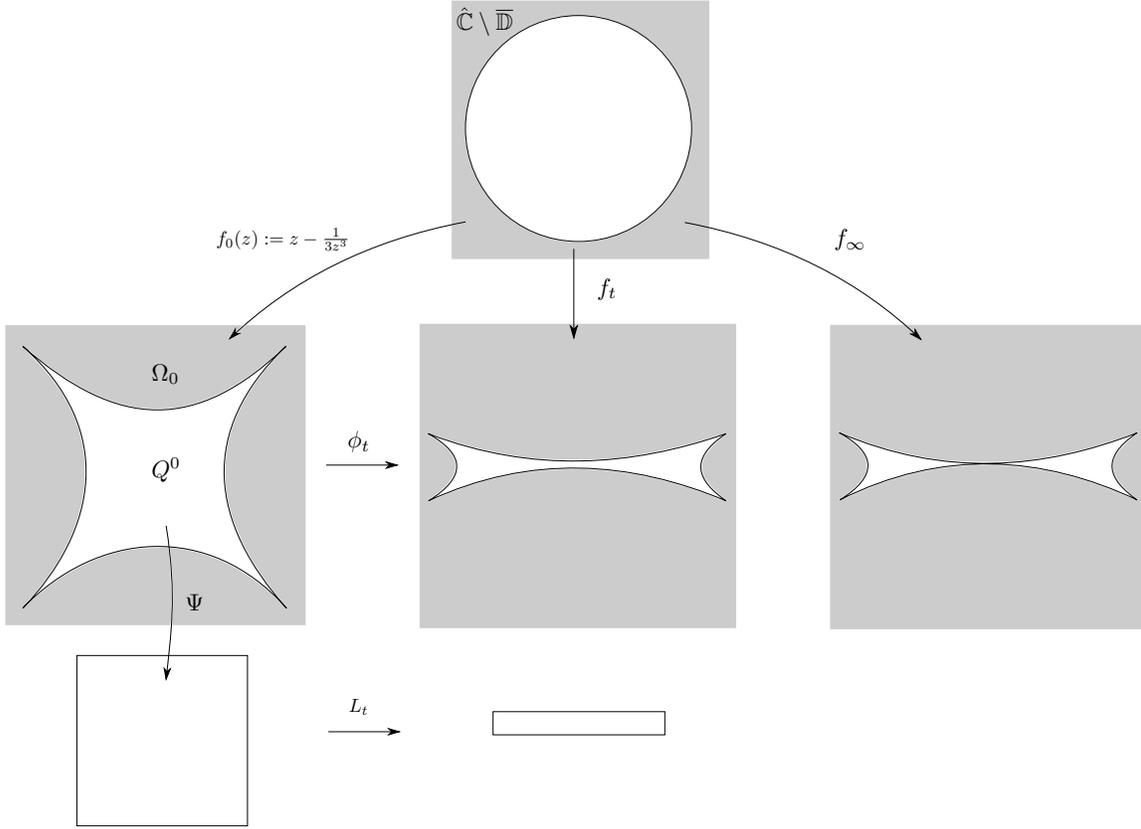}
\caption{ This Figure summarizes Section~\ref{astroid}. The image of $\widehat{\mathbb{C}}\setminus\overline{\mathbb{D}}$ under the map $f_0(z)= z-1/(3z^3)$ is a quadrature domain $\Omega_0$. $Q=\widehat{\mathbb{C}}\setminus\Omega_0$ is mapped conformally to $[-1,1]\times[-1,1]$ by $\Psi$. The map $L_t(x,y)=(x,\frac{y}{t})$ for $1<t<\infty$ determines a non-zero Beltrami coefficient on $[-1,1]\times[-1,1]$ which is pulled back under $\Psi$, spread by the Schwarz reflection map, and then straightened by a quasiconformal map $\phi_t$ which is conformal off of the tiling set. This determines a family of quadrature domains $f_t(\widehat{\mathbb{C}}\setminus\overline{\mathbb{D}})$ for $t\in[1,\infty)$. The desired extremal quadrature domain corresponds to a normal limit $f_\infty$ of the family $(f_t)_{t\nearrow+\infty}$.     }
\label{fig:astroid}      
\end{figure}

We now define a family of Beltrami coefficients $\mu_t\in L^{\infty}(\mathbb{C})$ depending on a real parameter $t\in[1,\infty)$ as follows. Let $L_t(x,y)=(x,y/t)$ for $1\leq t<\infty$, so that 
\[ \nu_t:=(L_t)_{\overline{z}}/(L_t)_{z}\equiv(t-1)/(t+1).\] 

\noindent For $z\in Q^0$, define $\mu_t(z)$ by pulling back $(L_t)_{\overline{z}}/(L_t)_{z}$ under $\Psi$; i.e., $\mu_t(z):=\Psi^*( \nu_t(\Psi(z)) )$. Next we pull back the Beltrami coefficient $\mu_t|_{Q^0}\in L^{\infty}(Q^0)$ under the Schwarz reflection map $\sigma_0$: 
\[ \mu_t(\sigma_0^{-n}(z) ):= (\sigma_0^{\circ n})^*( \mu_t(z) ) \textrm{ for } z\in Q^0 \textrm{ and } n\geq 1. \]
(See~\cite[Exercise~1.2.2]{BF14} for an explicit formula for pullbacks of Beltrami coefficients under anti-holomorphic maps.)

\noindent Lastly, we define:
\[ \mu_t(z)=0 \textrm{ for } z\not\in \cup_{n\geq 0} \sigma_0^{-n}(Q^0). \]

\noindent Since $\sigma_0$ is anti-holomorphic and $||\mu_t||_{L^{\infty}(Q^0)}=(t-1)/(t+1)$, it follows that $||\mu_t||_{L^{\infty}(\mathbb{C})}=(t-1)/(t+1)<1$.

\begin{prop}\label{normalizing_rational_maps} There exist a family of quasiconformal maps $(\phi_t)_{t\in[1,\infty)}:\widehat{\C}\rightarrow\widehat{\C}$, and a family of rational maps $(f_t)_{t\in[1,\infty)}$ such that $(\phi_t)_{\overline{z}}/(\phi_t)_{z}=\mu_t$ a.e., $\phi_t(\infty)=\infty$, $f_t\in\Sigma_3^*$, and $f_t(\widehat{\mathbb{C}}\setminus\overline{\D})=\phi_t(\Omega_0)$ for all $t\in[1,\infty)$.
\end{prop}

\begin{proof}
For any $t\in[1,\infty)$, we may apply the measurable Riemann mapping theorem to obtain a quasiconformal map $\phi_t:\widehat{\C}\to\widehat{\C}$ with $(\phi_t)_{\overline{z}}/(\phi_t)_{z}=\mu_t$ a.e., and $\phi_t(\infty)=\infty$. By construction, $\mu_t$ is invariant under $\sigma_0$. Hence, Proposition~\ref{qc_def_prop} implies that $\phi_t(\Omega_0)$ is an unbounded simply connected quadrature domain. Moreover, the same proposition ensures a normalization of $\phi_t$ and the existence of a rational map $f_t\in\Sigma_3^*$ such that $f_t(\widehat{\C}\setminus\overline{\D})=\phi_t(\Omega_0)$.
\end{proof}

\begin{prop}\label{conformal_limit} There exist a map $f_\infty\in\Sigma_3^*$ and a sequence of positive real numbers $(t_n)\nearrow+\infty$ such that \[(f_{t_n})\xrightarrow[]{n\rightarrow\infty} f_\infty \textrm{ uniformly on   }\widehat{\mathbb{C}} \] with respect to the spherical metric. 
\end{prop}

\begin{proof} We first prove that there is a constant $M>0$ such that for \[ f_t(z)=z+\frac{a^t_{1}}{z}+\frac{a^t_{2}}{z^2}-\frac{1}{3z^3},  \] 
\noindent we have $|a_j^t|\leq M$ for $j\in\{0,1\}$ and all $t\in[1,\infty)$. Indeed, this follows from the triangle inequality and the factorization \[ z^4\cdot f_t'(z)=z^4-a_1^tz^2-2a_2^tz+1=\prod_{j=1}^4(z-\xi_j^t)\] 
\noindent where $|\xi_j^t|=1$ for $1\leq j\leq4$ and all $t\in[1,\infty)$. Thus \[ |f_t(z)-z| < \frac{M}{R} + \frac{M}{R^2} + \frac{1}{3R^3} \textrm{ for } |z|>R. \]
\noindent Hence, given $\varepsilon>0$, we may take $R$ sufficiently large such that \begin{equation}\label{close_to_id}|f_t(z)-z|<\varepsilon \textrm{ for } |z| > R \textrm{ and all } t\in[1,\infty). \end{equation}
\noindent Since each $f_t$ is conformal in $\widehat{\mathbb{C}}\setminus\overline{\mathbb{D}}$, it follows from, for instance  of  of \cite[Theorem~5.1, \S II.5.2]{MR0344463}, that the family $(f_t|_{\widehat{\mathbb{C}}\setminus\overline{\mathbb{D}}})_{t\in[1,\infty)}$ is normal (with respect to the spherical metric). Let $(t_n)\nearrow\infty$ and denote by $f_\infty$ a normal limit of $(f_{t_n}|_{\widehat{\mathbb{C}}\setminus\overline{\mathbb{D}}})_{n=1}^{\infty}$. Note that $f_\infty|_{\widehat{\mathbb{C}}\setminus\overline{\mathbb{D}}}$ is conformal by (\ref{close_to_id}) and \cite[Theorem~II.5.2]{MR0344463}. It follows from the Cauchy integral theorem that the coefficients in the Laurent series expansion for $(f_{t_n})_{n=1}^{\infty}$ (at $\infty$) converge to the corresponding coefficients in the Laurent series expansion for $f_\infty$ (at $\infty$). Hence \[ f_\infty(z)=z+\frac{a_{1}}{z}+\frac{a_{2}}{z^2}-\frac{1}{3z^3} \]
\noindent for some $a_1,a_2\in\mathbb{C}$, and so $f_\infty\in\Sigma_3^*$. Moreover, since we have $a_{1}^{t_n}\rightarrow a_1$ and  $a_{2}^{t_n}\rightarrow a_2$ as $n\rightarrow\infty$, it follows that the convergence $f_{t_n}\rightarrow f_\infty$ is uniform on $\widehat{\mathbb{C}}$ (with respect to the spherical metric), as needed.
\end{proof}

\begin{rem}\label{normal_rem} 
Here is an alternative argument for normality of $(f_t|_{\widehat{\mathbb{C}}\setminus\overline{\mathbb{D}}})_{t\in[1,\infty)}$. Note that each map $f_t$ is univalent on $\widehat{\mathbb{C}}\setminus\overline{\mathbb{D}}$, fixes $\infty$, and has derivative $1$ at $\infty$. Thus, by \cite[Theorem~1.10]{CG1}, the restrictions of $f_t$ on $\widehat{\mathbb{C}}\setminus\overline{\mathbb{D}}$ form a normal family, and a normal limit also has the same properties.
\end{rem}

\noindent Note that by Propositions~\ref{conformal_limit} and~\ref{s.c.q.d.}, the image $\Omega_\infty:=f_\infty(\widehat{\C}\setminus\overline{\D})$ is an unbounded simply connected quadrature domain.

\begin{prop}\label{4_cusp_1_double}
The boundary of the quadrature domain $\Omega_\infty$ has $4$ cusps and $1$ double point.
\end{prop}

\begin{rem}\label{double_cut_rem} We will use without further comment that any double point for $f\in\Sigma_d^*$ must be a cut point for $\widehat{\mathbb{C}}\setminus\Omega_{\infty}$ since $\Omega_{\infty}$ is connected. We also note that a cusp point of $f(\mathbb{T})$ for $f\in\Sigma_d^*$ cannot also be a double point of $f(\mathbb{T})$ by conformality of $f|_{\widehat{\mathbb{C}}\setminus\overline{\mathbb{D}}}$.
\end{rem}

\begin{proof}[Proof of Proposition~\ref{4_cusp_1_double}]  That $\partial\Omega_\infty$ has $4$ cusps follows from Proposition~\ref{crit_points_on_circle}. It remains to show that $\Omega_\infty$ has one double point.

 Proposition~\ref{crit_points_on_circle} also implies that each $f_t$ has $4$ distinct critical points on $\mathbb{T}$. Let $\xi_1^n$ denote the critical point of $f_{t_n}$ in the upper-right-half plane, and enumerate the remaining critical points in counter-clockwise order as $\xi_2^n$, $\xi_3^n$, $\xi_4^n$. Denote the corresponding critical values by $\zeta^n_1, \cdots, \zeta^n_4$. We may assume, by taking a subsequence in $n$ if necessary, that $\xi^n_j\rightarrow \xi^\infty_j\in\partial\mathbb{D}$ for $1\leq j\leq 4$. Since $f_{t_n}\rightarrow f_\infty$ uniformly on $\widehat{\mathbb{C}}$, it follows that $f_\infty'(\xi^{\infty}_j)=0$. Moreover, $\xi_1^\infty$, $\xi_2^\infty$, $\xi_3^\infty$, $\xi_4^\infty$ are all distinct by Proposition \ref{crit_points_on_circle} as $f_\infty\in\Sigma_3^*$.

We will denote by $\arc{z_1z_2}$ the closed subarc of $\partial\mathbb{D}$ with endpoints $z_1$, $z_2 \in \partial\mathbb{D}$ such that the triple $(z_1, e^{i\theta}, z_2)$ is oriented positively with respect to $\mathbb{D}$ for any $e^{i\theta} \in \arc{z_1z_2}$. We will also employ the notation \[Q_n=\widehat{\C}\setminus f_{t_n}(\widehat{\C}\setminus\overline{\D})\textrm{, } Q_\infty=\widehat{\C}\setminus f_{\infty}(\widehat{\C}\setminus\overline{\D}).\] Consider the quadrilateral $Q_n(\zeta^n_1, \zeta^n_2, \zeta^n_3, \zeta^n_4)$. Note that the map \[ L_{t_n}\circ\Psi\circ\phi_{t_n}^{-1}: Q_{n} \rightarrow [-1,1]\times[-1/t_n,1/t_n] \] is conformal for each $n\geq1$ and preserves vertices, so that \begin{equation}\label{some_convergence} M(Q_n)=t_n \rightarrow \infty \textrm{ as } n\rightarrow\infty. \end{equation}

We will show that  \begin{equation}\label{a-sides-intersect} f_\infty(\arc{\xi^\infty_1\xi^\infty_2}) \cap f_\infty( \arc{\xi^\infty_3\xi^\infty_4} )\not=\emptyset, \end{equation} 

\noindent whence the claim that $\Omega_\infty$ has exactly one double point will follow since $\Omega_\infty$ can have \emph{at most} one double point (see \cite[Lemma 2.4]{2014arXiv1411.3415L}). First we show that the only possible self-intersection of $\partial\Omega_\infty$ is between the arcs  $f_\infty(\arc{\xi^\infty_1\xi^\infty_2})$, $f_\infty( \arc{\xi^\infty_3\xi^\infty_4} )$. Indeed, all other a priori possible self-intersections of $\partial Q_{\infty}$ fall into the following three categories, and we prove that each, in turn, is impossible:

\begin{itemize}

\item An arc $f_\infty(\arc{\xi_j^\infty\xi_{j+1}^\infty})$ has a self-intersection.

\subitem This is disallowed because of Principle \ref{constant_curvature_principle} (constant curvature principle).

\item Two adjacent arcs $f_{\infty}(\arc{\xi_j^\infty\xi_{j+1}^\infty})\textrm{, } f_{\infty}(\arc{\xi_{j+1}^\infty\xi_{j+2}^\infty})$ intersect at more than their common endpoint.

\subitem This is disallowed because of Principle \ref{constant_curvature_principle} (constant curvature principle).

\item The arcs $f_\infty(\arc{\xi^\infty_2\xi^\infty_3})$, $f_\infty( \arc{\xi^\infty_4\xi^\infty_1} )$ have non-empty intersection.

\subitem This follows from Theorem \ref{modulus_estimate}. Indeed, let $s_a(n)$, $s_b(n)$ denote the Euclidean path-distance between the $a$, $b$-sides of $Q_n$, respectively. If $f_\infty(\arc{\xi^\infty_2\xi^\infty_3})$, $f_\infty( \arc{\xi^\infty_4\xi^\infty_1} )$ intersect, then $s_{b}(n)\rightarrow 0$ but $s_{a}(n)$ stays bounded away from $0$ as $n\rightarrow\infty$ (this follows from the local geometry of a double point singularity). Theorem \ref{modulus_estimate} then implies $M(Q_n)\rightarrow 0$ as $n\rightarrow\infty$, and this contradicts (\ref{some_convergence}). 
\end{itemize}

Thus if we suppose, by way of contradiction, the failure of (\ref{a-sides-intersect}),  $\partial Q_{\infty}$ must be a Jordan curve. It follows then that \[Q_\infty(\zeta^\infty_1, \zeta^\infty_2, \zeta^\infty_3, \zeta^\infty_4) \] is a topological quadrilateral, and moreover, \[Q_n \rightarrow Q_\infty \textrm{ as } n\rightarrow\infty \] in the sense of Definition \ref{definition_of_quadrilateral_convergence}. Hence \[M(Q_n)\rightarrow M(Q_\infty) \textrm{ as } n\rightarrow\infty \] by Theorem \ref{convergence_of_quadrilaterals}. However $M(Q_\infty)<\infty$ since $Q_\infty$ is a topological quadrilateral, which contradicts (\ref{some_convergence}), as needed.
\end{proof}

\subsection{The General Case}\label{mainthm_general_proof_subsec}

Let $\mathcal{T}$ be a bi-angled tree. We define \begin{equation}\label{degree_from_tree} d_\mathcal{T}:=1+\#\{\textrm{vertices in the tree } \mathcal{T}\}.\end{equation}

\begin{rem} As will be seen below, the quantity $d_\mathcal{T}$ indicates in which degree we will look for an extremal function whose associated bi-angled tree is isomorphic to $\mathcal{T}$. For instance, the tree $\mathcal{T}$ of Subsection \ref{astroid} consists of two vertices with a connecting edge, for which $d_\mathcal{T}=3$, and the associated extremal function is indeed in the family $\Sigma_3^*$.
\end{rem}

\noindent We set $d=d_\mathcal{T}$. Denote by $\xi_1, \cdots, \xi_{d+1}$ the critical points of the map \begin{equation}\label{z^d} f_0(z)= z-\frac{1}{dz^d} \end{equation} on $\partial\mathbb{D}$, enumerated counter-clockwise starting with $\xi_1:=\exp(2\pi i/(d+1))$. Label the arc of $\partial\mathbb{D}$ connecting $\xi_j$ to $\xi_{j+1}$ as $I_j$, for $1\leq j \leq d+1$ (see Figure \ref{fig:unit_disc}). Our strategy in the proof of Theorem~\ref{mainthm_qd_terminology} will be to pinch pairs of arcs from $(f_0(I_j))_{j=1}^{d+1}$ as in Section~\ref{astroid}, so that the tree associated to the resulting pinched quadrature domain will be isomorphic to $\mathcal{T}$. For now, we need to determine, from the combinatorics of the tree $\mathcal{T}$, for which pairs of indices $\{j,k\}$ will we pinch the associated arcs. To this end, it will be useful to introduce an \emph{augmented tree} $\widehat{\mathcal{T}}$ from the bi-angled tree $\mathcal{T}$ (see Figure \ref{angled_augmented_tree}).

\begin{figure}
\centering
\scalebox{.5}{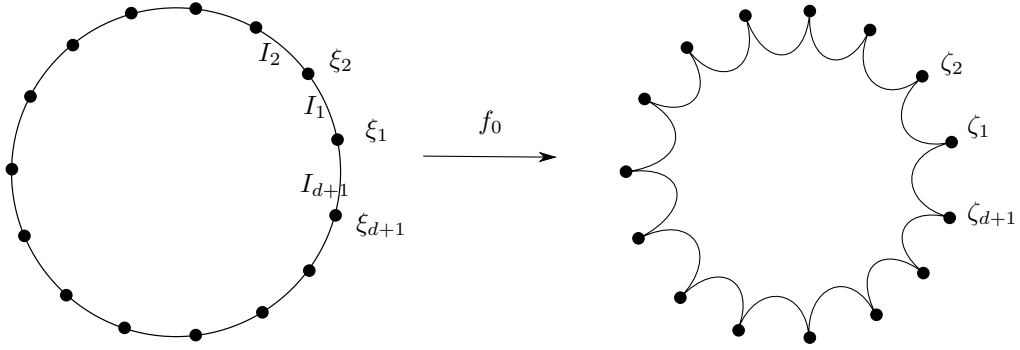}
\caption{ This Figure illustrates the image of the unit circle under the mapping $f_0$. The map $f_0$ has $d+1$ critical points $\xi_1, \cdots, \xi_{d+1}$ on the unit circle with corresponding critical values $\zeta_1, \cdots, \zeta_{d+1}$. }
\label{fig:unit_disc}      
\end{figure}

To do so, let us first color the vertices of $\mathcal{T}$ red. Now, for each vertex $v$ of degree $j$ ($j=1,2$) in $\mathcal{T}$, we add $3-j$ new edge(s) that connect(s) $v$ to $3-j$ new vertices, which we color blue. We add these $3-j$ new vertices and edges such that if $e$ and $e'$ are two edges incident at a common vertex $v$, and are consecutive in the counter-clockwise circular order around $v$, then the counter-clockwise oriented angle between $e$ and $e'$ is $\frac{2\pi}{3}$. Finally, given any edge $e$ in $\mathcal{T}$ (by definition, it connects two red vertices), subdivide it into two edges by introducing a blue vertex in the middle of the edge $e$. We will call the resulting tree the \emph{augmented tree of} $\mathcal{T}$, and denote it by $\widehat{\mathcal{T}}$.

Now suppose that $\mathcal{T}=\mathcal{T}(\Omega)$, where $\Omega$ is an extremal unbounded quadrature domain of order $d$. Then, $\mathcal{T}$ has $d-1$ vertices. It follows from the construction of $\widehat{\mathcal{T}}$ that $\widehat{\mathcal{T}}$ has $d-1$ red vertices (each of degree $3$), and $2d-1$ blue vertices (of which $d+1$ have degree $1$, and $d-2$ have degree $2$). Every edge of $\widehat{\mathcal{T}}$ connects a blue and a red vertex. Moreover, each red vertex of $\widehat{\mathcal{T}}$ represents an interior component of $T=\widehat{\C}\setminus\Omega$, each blue vertex of degree $1$ (respectively, of degree $2$) stands for a cusp (respectively, a double point) on $\partial T$. There is an edge in $\widehat{\mathcal{T}}$ connecting a red and a blue vertex if and only if the singular point corresponding to the blue vertex lies on the boundary of the interior component of $T$ corresponding to the red vertex.

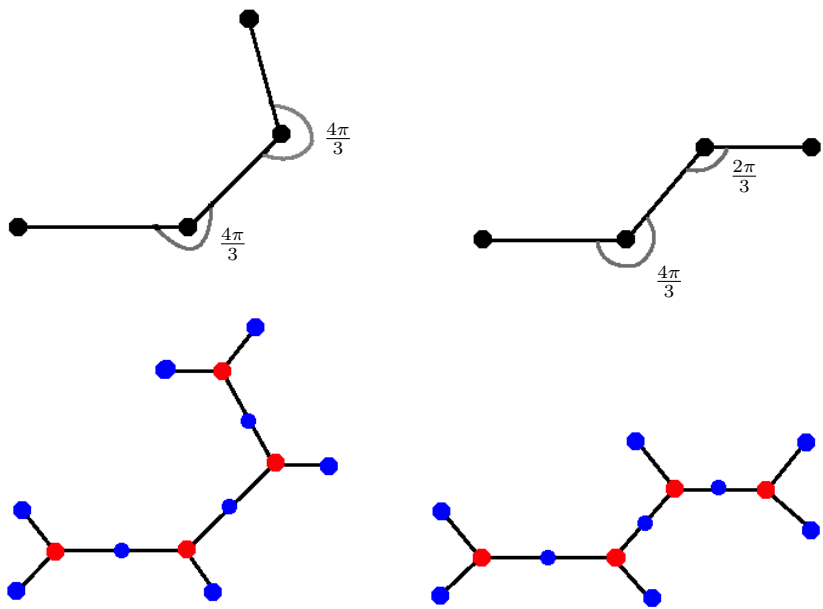
\begin{figure}[ht!]
\begin{tikzpicture}
\node[anchor=south west,inner sep=0] at (0,4.5) {\includegraphics[width=0.8\textwidth]{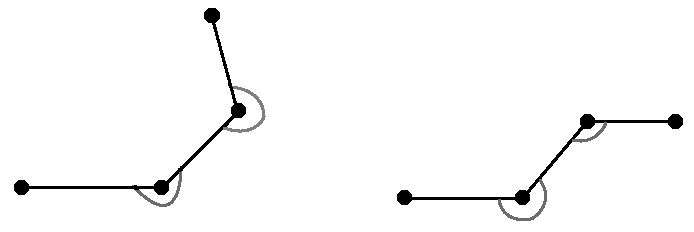}};
\node[anchor=south west,inner sep=0] at (0,0) {\includegraphics[width=0.8\textwidth]{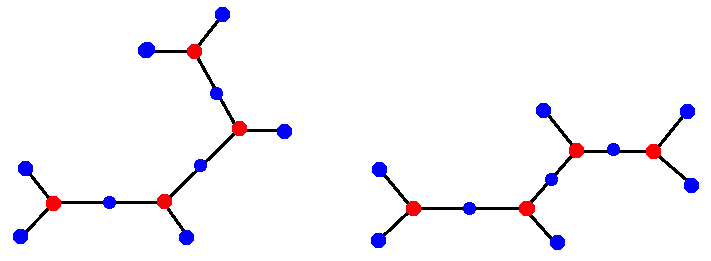}};
\node at (9,4.5) {$\frac{4\pi}{3}$};
\node at (10,5.9) {$\frac{2\pi}{3}$};
\node at (4.6,6.4) {$\frac{4\pi}{3}$};
\node at (3.2,5) {$\frac{4\pi}{3}$};
\end{tikzpicture}
\caption{Top: Two non-isomorphic bi-angled trees (for $d=5$) whose underlying topological trees are isomorphic. Bottom: No isomorphism between the corresponding augmented trees preserves the circular order of edges meeting at each vertex.}
\label{angled_augmented_tree}
\end{figure}

The bi-angled tree $\mathcal{T}$ and the corresponding augmented tree $\widehat{\mathcal{T}}$ contain precisely the same amount of information; i.e., one can be recovered from the other. Indeed, the bi-angled tree $\mathcal{T}$ is obtained from the augmented tree $\widehat{\mathcal{T}}$ by removing all the blue vertices along with the edges of $\widehat{\mathcal{T}}$ terminating at the blue vertices of degree $1$. Thus, the red vertices of $\widehat{\mathcal{T}}$ are the only vertices of $\mathcal{T}$. This is illustrated in Figure~\ref{angled_augmented_tree}.

\begin{definition}\label{isom_def_1}
Two augmented trees $\widehat{\mathcal{T}}_1$, $\widehat{\mathcal{T}}_2$ are said to be \emph{isomorphic} if there exists a tree isomorphism between them preserving the color of each vertex and the (counter-clockwise) cyclic order of edges meeting at each vertex.
\end{definition}

The above discussion directly implies the following lemma.

\begin{lem}\label{angled_augmented_equiv_lem}
Two bi-angled trees $\mathcal{T}_1$ and $\mathcal{T}_2$ are isomorphic if and only if the corresponding augmented trees $\widehat{\mathcal{T}}_1$ and $\widehat{\mathcal{T}}_2$ are isomorphic.
\end{lem}

With the introductory remarks of this section in mind, we will now record, given a tree $\mathcal{T}$, which edges $(f_0(I_j))_{j=1}^{d+1}$ we will need to ``pinch'' in order to obtain an extremal quadrature domain whose bi-angled tree is isomorphic to $\mathcal{T}$. See Figure \ref{fig:orientation} for an illustration of Definition \ref{definition_of_S}.

\begin{definition}\label{definition_of_S} Given a tree $\mathcal{T}$, we define a subset \begin{equation}\label{combinatorics_from_tree}S=S_\mathcal{T}\subset \{1, \cdots, d+1\}\times \{1, \cdots, d+1\} \end{equation} as follows. Consider the augmented tree $\widehat{\mathcal{T}}$, and a Riemann map \begin{equation} \phi: \widehat{\mathbb{C}} \setminus \mathbb{D} \rightarrow \widehat{\mathbb{C}} \setminus \widehat{\mathcal{T}}  \end{equation} fixing infinity. There are $d+1$ preimages (under $\phi$) of the degree $1$ blue vertices of $\widehat{\mathcal{T}}$. We label them $\eta_1$, ..., $\eta_{d+1}$ in counter-clockwise order starting with an arbitrary choice of $\eta_1$. Denote the closed arc of $\mathbb{T}$ connecting $\eta_k$ to $\eta_{k+1}$ (where $k+1$ is taken mod $d+1$) by $J_k$. Lastly, let $j$, $k \in \{1, ..., d+1\}$ with $j\not=k\pm1$ (mod $d+1$). We define \begin{equation} \{j, k\} \in S \textrm{ if and only if  } \phi(J_j)\cap\phi(J_k)\not=\emptyset. \end{equation}

\end{definition}

\begin{proof}[Proof of Theorem~\ref{mainthm_qd_terminology}] 

Let $\mathcal{T}$ be as in the statement of Theorem~\ref{mainthm_qd_terminology}. Define $d=d_\mathcal{T}$ as in (\ref{degree_from_tree}) and $S=S_\mathcal{T}$ as in Definition \ref{definition_of_S}. Consider the map \[ f_0(z)= z-\frac{1}{dz^d}. \] Let $\{j, k\} \in S$, and define $\zeta_l:=f_0(\xi_l)$ to be the critical value of $f_0$ corresponding to the critical point $\xi_l\in\partial\mathbb{D}$, for $1\leq l \leq d+1$. Consider the topological quadrilateral \[ Q\left( \zeta_j, \zeta_{j+1}, \zeta_{k}, \zeta_{k+1} \right), \textrm{ where } Q=\widehat{\mathbb{C}}\setminus f_0(\widehat{\mathbb{C}} \setminus\overline{\D}  ). \]

\noindent Denote by $\Psi\equiv\Psi_s$ the conformal map of $Q$ to the Euclidean rectangle $[0, M(Q)]\times[0,1]$ which preserves vertices. We now proceed as in Section~\ref{astroid}: see Figure~\ref{fig:quasiconformal_deformation}.

\begin{figure}
\centering
\scalebox{.36}{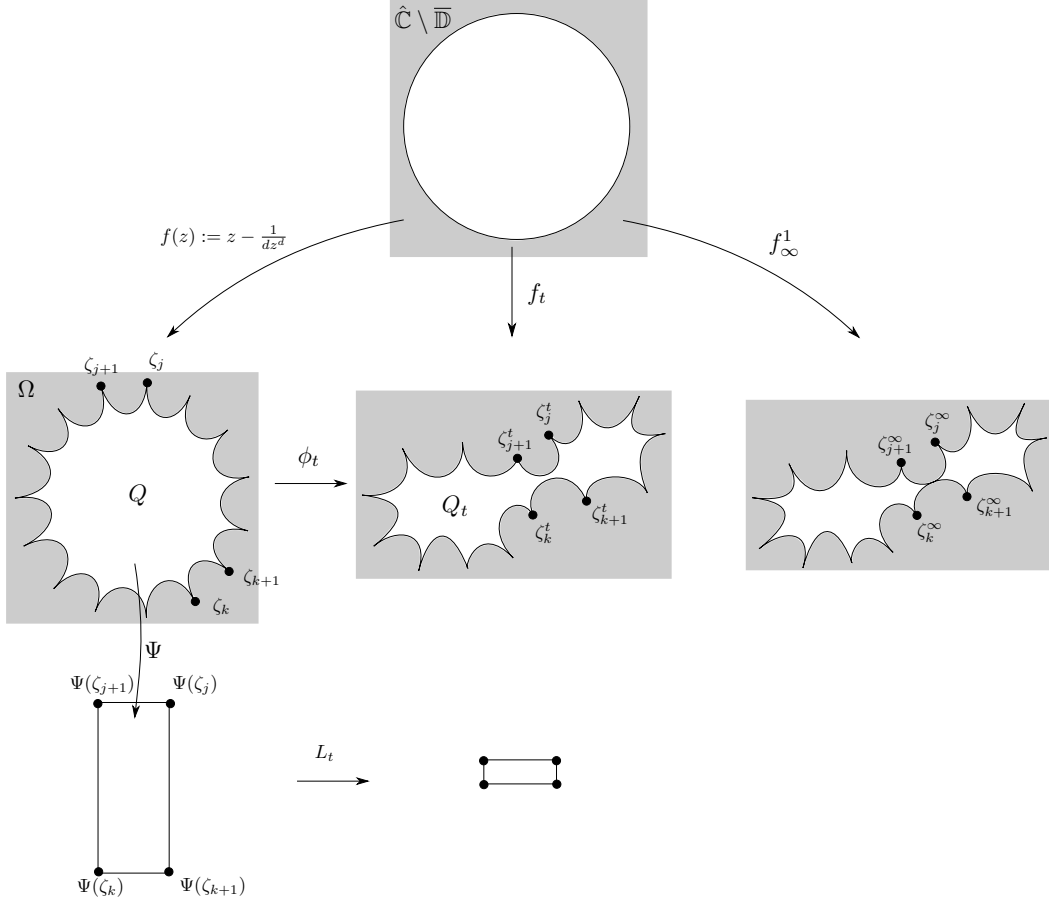}
\caption{ This Figure summarizes the first step in the proof of Theorem~\ref{mainthm_qd_terminology}. One considers a pair of indices $\{j, k\} \in S_T$ and pinches the curve $f(I_j)$ to the curve $f(I_k)$ by quasiconformally perturbing the quadrilateral $Q$ as in Section~\ref{astroid}. }

\label{fig:quasiconformal_deformation}      
\end{figure}

Let $L_t$ be the linear map defined by $L_t(x,y)=(x,y/t)$ for $1\leq t<\infty$, so that \[ \nu_t=(L_t)_{\overline{z}}/(L_t)_{z}\equiv(t-1)/(t+1).\] For $z\in Q$, define $\mu_t(z)$ by pulling back $(L_t)_{\overline{z}}/(L_t)_{z}$ under $\Psi$; i.e., $\mu_t(z)=\Psi^*( \nu_t(\Psi(z)) )$. Next we pull back the Beltrami coefficient $\mu_t|_{Q}\in L^{\infty}(Q)$ under the Schwarz reflection map $\sigma_0$ associated to $\Omega_0=f_0(\widehat{\mathbb{C}}\setminus\overline{\D})$: \[ \mu_t(\sigma_0^{-n}(z) )= (\sigma_0^{\circ n})^*( \mu_t(z) ) \textrm{ for } z\in Q^0 \textrm{ and } n\geq 1. \] Lastly, we define: \[ \mu_t(z)=0 \textrm{ for } z\not\in \bigcup_{n\geq0} \sigma_0^{-n}(Q^0). \]

\noindent The same proofs as for Propositions~\ref{normalizing_rational_maps} and~\ref{conformal_limit} yield the following two propositions:

\begin{prop}\label{normalizing_rational_maps_general} There exist a family of quasiconformal maps $(\phi_t)_{t\in[1,\infty)}:\widehat{\C}\rightarrow\widehat{\C}$, and a family of rational maps $(f_t)_{t\in[1,\infty)}$ such that $(\phi_t)_{\overline{z}}/(\phi_t)_{z}=\mu_t$ a.e., $\phi_t(\infty)=\infty$, $f_t\in\Sigma_d^*$, and $f_t(\widehat{\mathbb{C}}\setminus\overline{\D})=\phi_t(\Omega_0)$ for all $t\in[1,\infty)$.
\end{prop}

\begin{prop}\label{conformal_limit_general} There exist a map ${f}^1_{\infty}\in\Sigma_d^*$ and a sequence of positive real numbers $(t_n)\nearrow\infty$ such that \[(f_{t_n})\xrightarrow[]{n\rightarrow\infty} f^1_{\infty} \textrm{  uniformly on   }\widehat{\mathbb{C}} \] with respect to the spherical metric. 
\end{prop}

\begin{rem}\label{superscript_rem} We note that by Propositions~\ref{conformal_limit_general} and~\ref{s.c.q.d.}, $\Omega^1_\infty:=f^1_{\infty}(\widehat{\mathbb{C}}\setminus\overline{\mathbb{D}})$ is an unbounded, simply connected quadrature domain. The superscript $1$ is used in the notation ${\Omega}^1_\infty$, ${f}^1_{\infty}$ to emphasize that $\partial{\Omega}^1_\infty$ will have exactly one double point, as proven below in Proposition~\ref{d+1_cusp_1_double}. Subsequent steps in the proof of Theorem~\ref{mainthm_qd_terminology} will yield functions ${f}^j_{\infty}$, $2\leq j \leq |S|$ such that $\partial{\Omega}^j_\infty$ has exactly $j$ double points (see also Figure~\ref{fig:quasiconformal_deformation_step_two}).
\end{rem}

\begin{prop}\label{d+1_cusp_1_double}
The boundary of the quadrature domain $\Omega^1_\infty$ has $d+1$ cusps and $1$ double point.
\end{prop}

\begin{proof} That  $\partial\Omega^1_\infty$ has $d+1$ cusps follows from Proposition~\ref{crit_points_on_circle} since ${f}^1_{\infty} \in \Sigma_d^*$. Proposition~\ref{crit_points_on_circle} also implies that each $f_t$ has $d+1$ distinct critical points on $\mathbb{T}$. We label these critical points in counterclockwise order as $\xi_1^t, \cdots, \xi_{d+1}^t$, with $\xi_1^t$ having the minimal argument where we take the branch $0\leq\arg(z)<2\pi$. We denote $\zeta_i^t:=f_t(\xi_i^t)$, and recall the notation $\zeta_i:=f_0(\xi_i)$. We claim that \begin{equation}\label{dependence_on_t}\zeta_i^t=\phi_t(\zeta_i) \textrm{ for } 1\leq i \leq d+1 \textrm{ and } 1\leq t<\infty.\end{equation} Indeed, the relation clearly holds for $t=1$, whence for $t>1$ we note that each $\phi_t(\zeta_i)$ must be a critical value of $f_t$ since quasiconformal maps preserve cusps, and so the claim follows by continuous dependence of $\phi_t$ on the parameter $t$ (see, for instance, \cite[Theorem~7.5]{CG1}). By taking a further subsequence in $n$ if necessary, we have that \[ \xi_1^{t_n} \xrightarrow{n\rightarrow\infty} \xi_1^\infty, \cdots, \xi_{d+1}^{t_n} \xrightarrow{n\rightarrow\infty} \xi_{d+1}^\infty\textrm{} \] for some $\xi_1^\infty, \cdots, \xi_{d+1}^\infty \in \partial\mathbb{D}$. Since $f_{t_n}\rightarrow f^1_{\infty}$ uniformly in the spherical metric, it follows that $\xi_1^\infty, \cdots, \xi_{d+1}^\infty$ are critical points for $f^1_{\infty}$. Moreover,  $\xi_1^\infty, \cdots, \xi_{d+1}^\infty$ are distinct since $f^1_{\infty}\in\Sigma_d^*$.

We consider \[ Q_n\big( \zeta_j^{t_n}, \zeta_{j+1}^{t_n}, \zeta_k^{t_n}, \zeta_{k+1}^{t_n} \big), \] where \[Q_n:=\widehat{\C}\setminus f_{t_n}(\widehat{\C}\setminus\overline{\D})\textrm{ and } {Q}^1_\infty:=\widehat{\C}\setminus{f}^1_{\infty}(\widehat{\C}\setminus\overline{\D}). \] Note that the map \begin{equation}\label{conformal_invariance} L_{t_n}\circ\Psi\circ\phi_{t_n}^{-1}: Q_n\big( \zeta_j^{t_n}, \zeta_{j+1}^{t_n}, \zeta_k^{t_n}, \zeta_{k+1}^{t_n} \big) \rightarrow [0,M(Q)]\times[0,1/t_n] \end{equation} is conformal and preserves vertices by (\ref{dependence_on_t}), so that \begin{equation}\label{needed_contradiction} M(Q_n)\rightarrow\infty \textrm{ as } n\rightarrow\infty. \end{equation} 

We will show that \begin{equation}\label{want_to_show_intersection} f^1_{\infty}(\arc{\xi_j^\infty\xi_{j+1}^\infty}) \cap f^1_{\infty}(\arc{\xi_k^\infty\xi_{k+1}^\infty}) \not=\emptyset, \end{equation} but first we need to prove that the only possible self-intersection of $\partial Q^1_\infty$ is between the arcs $f^1_{\infty}(\arc{\xi_j^\infty\xi_{j+1}^\infty}), f^1_{\infty}(\arc{\xi_k^\infty\xi_{k+1}^\infty})$. As in Proposition \ref{4_cusp_1_double}, we list all other a priori possibilities and rule out each in turn:

\begin{itemize}

\item An arc $f^1_\infty(\arc{\xi^i_\infty\xi^{i+1}_\infty})$ has a self-intersection.

\subitem This is disallowed because of Principle \ref{constant_curvature_principle} (constant curvature principle). 

\item Two adjacent arcs $f^1_{\infty}(\arc{\xi^i_\infty\xi^{i+1}_\infty})\textrm{, } f^1_{\infty}(\arc{\xi^{i+1}_\infty\xi^{i+2}_\infty})$ intersect at more than their common endpoint.

\subitem This is disallowed because of Principle \ref{constant_curvature_principle} (constant curvature principle).

\item Non-adjacent arcs $f^1_{\infty}(\arc{\xi_m^\infty\xi_{m+1}^\infty})$,  $f^1_{\infty}(\arc{\xi_l^\infty\xi_{l+1}^\infty})$ have non-empty intersection, where $\{m,l\}\not=\{j,k\}$.

Consider the quadrilateral $Q_n(\zeta_m^{t_n}, \zeta_{m+1}^{t_n}, \zeta_l^{t_n}, \zeta_{l+1}^{t_n})$.  Let $s_a(n)$, $s_b(n)$ denote the Euclidean path-distance between the a, b-sides of $Q_n(\zeta_m^{t_n}, \zeta_{m+1}^{t_n}, \zeta_l^{t_n}, \zeta_{l+1}^{t_n})$, respectively. If $f^1_{\infty}(\arc{\xi_m^\infty\xi_{m+1}^\infty})$,  $f^1_{\infty}(\arc{\xi_l^\infty\xi_{l+1}^\infty})$ have non-empty intersection, then $s_a(n) \rightarrow 0$ but $s_b(n) \not \rightarrow 0$ as $n\rightarrow\infty$. Thus by Theorem \ref{modulus_estimate}, \[ M(Q_n(\zeta_m^{t_n}, \zeta_{m+1}^{t_n}, \zeta_l^{t_n}, \zeta_{l+1}^{t_n})) \rightarrow \infty  \textrm{ as } t\rightarrow\infty.\]

Thus to rule out the intersection of $f^1_{\infty}(\arc{\xi_m^\infty\xi_{m+1}^\infty})$,  $f^1_{\infty}(\arc{\xi_l^\infty\xi_{l+1}^\infty})$, it will suffice to show that in fact the family $\Gamma_t$ of rectifiable curves in $\overline{Q_t}$ with endpoints one on each of \[ f_t(\arc{\xi_m^t\xi_{m+1}^t})\textrm{, } f_t(\arc{\xi_l^t\xi_{l+1}^t}) \] satisfies \begin{equation}\label{desired_extremal_length_estimate}\textrm{mod}(\Gamma_t)<R<\infty \textrm{ for some } R>0 \textrm{ independent of }t.\end{equation} Note that \begin{equation}\label{extremal_length}\textrm{mod}(\Gamma_t)=\textrm{mod}\left(L_{t}\circ\Psi\circ\phi_{t}^{-1}(\Gamma_t)\right),\end{equation} since \[L_{t}\circ\Psi\circ\phi_{t}^{-1}: Q_t \rightarrow [0,M(Q)]\times[0,t^{-1}]\] is conformal (see, for instance, Theorem IV.3.2 of \cite{MR0344463}). The path family on the right-hand side of (\ref{extremal_length}) corresponds to the family of curves in $[0,M(Q)]\times[0,t^{-1}]$ connecting the linear segment between $t^{-1}\Psi(\zeta_m)$, $t^{-1}\Psi(\zeta_{m+1})$ with the non-adjacent linear segment connecting $t^{-1}\Psi(\zeta_l)$, $t^{-1}\Psi(\zeta_{l+1})$ (see Figure \ref{fig:extremal_length_argument}). Since $\{m,l\}\not=\{j,k\}$, at least one of these linear segments is vertical. The argument is thus finished by applying Principle \ref{constant_ratio_principle} (Constant ratio principle).

\end{itemize}

\begin{figure}
\centering
\scalebox{.45}{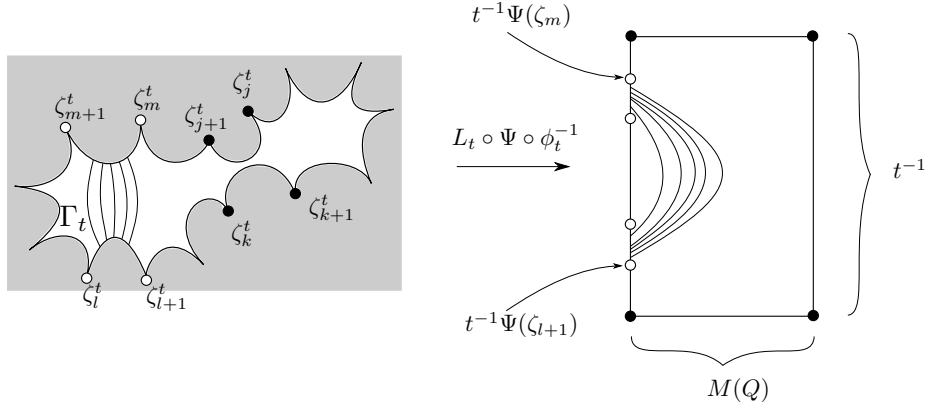}
\caption{ Illustrated is the argument that the modulus of the path family $\Gamma_t$ stays bounded away from $\infty$ as $t\rightarrow\infty$.  }
\label{fig:extremal_length_argument}      
\end{figure}

We have established that the only possible self-intersection of $\partial Q^1_\infty$ is between the arcs $f^1_{\infty}(\arc{\xi_j^\infty\xi_{j+1}^\infty}), f^1_{\infty}(\arc{\xi_k^\infty\xi_{k+1}^\infty})$. Thus if we suppose by way of contradiction that (\ref{want_to_show_intersection}) fails, $\partial Q_{\infty}$ must be a Jordan curve. We now follow the same strategy as in the proof of Proposition \ref{4_cusp_1_double}. Since $\partial Q_{\infty}$ is a Jordan curve, \[Q_\infty(\zeta^\infty_j, \zeta^\infty_{j+1}, \zeta^\infty_k, \zeta^\infty_{k+1}) \] is a topological quadrilateral, and moreover, \[M(Q_n) \rightarrow M(Q_\infty) < \infty \textrm{ as } n\rightarrow\infty, \] contradicting (\ref{needed_contradiction}).

Thus we have shown that $\partial\Omega^1_\infty$ has at least one double point: an intersection point of the curves $f^1_{\infty}(\arc{\xi_j^\infty\xi_{j+1}^\infty})$, $f^1_{\infty}(\arc{\xi_k^\infty\xi_{k+1}^\infty})$. It remains to be shown that there are no other double points for $\partial\Omega^1_\infty$. The only possibility of a second double point which we have not already ruled out is:

\begin{itemize}

\item The arcs $f^1_{\infty}(\arc{\xi_j^\infty\xi_{j+1}^\infty})$,  $f^1_{\infty}(\arc{\xi_k^\infty\xi_{k+1}^\infty})$ intersect at more than one point.

\subitem This is disallowed because of Principle \ref{constant_curvature_principle} (constant curvature principle).

\end{itemize}
\end{proof}

We summarize our discussion in the proof of Theorem~\ref{mainthm_qd_terminology} thus far. Given a bi-angled tree $\mathcal{T}$, we have quasiconformally modified the function $f_0(z)=z-1/(dz^d)$ to produce a function $f^1_{\infty}\in\Sigma_d^*$ so that the associated quadrature domain has exactly one double point corresponding to pinching an arbitrarily chosen pair of edges indexed by $\{j,k\}\in S=S_\mathcal{T}$. The remainder of the proof of Theorem \ref{mainthm_qd_terminology} proceeds in a recursive fashion: we will quasiconformally modify the function $f^1_{\infty}$ to produce another pair of pinched edges specified by $S$ (see Figure \ref{fig:quasiconformal_deformation_step_two}), and repeat this procedure until all of the pinchings specified by $S$ are achieved. We will finish the proof of Theorem \ref{mainthm_qd_terminology} by detailing the second step in this recursive procedure: subsequent steps are completely analogous. 

\begin{figure}
\centering
\scalebox{.36}{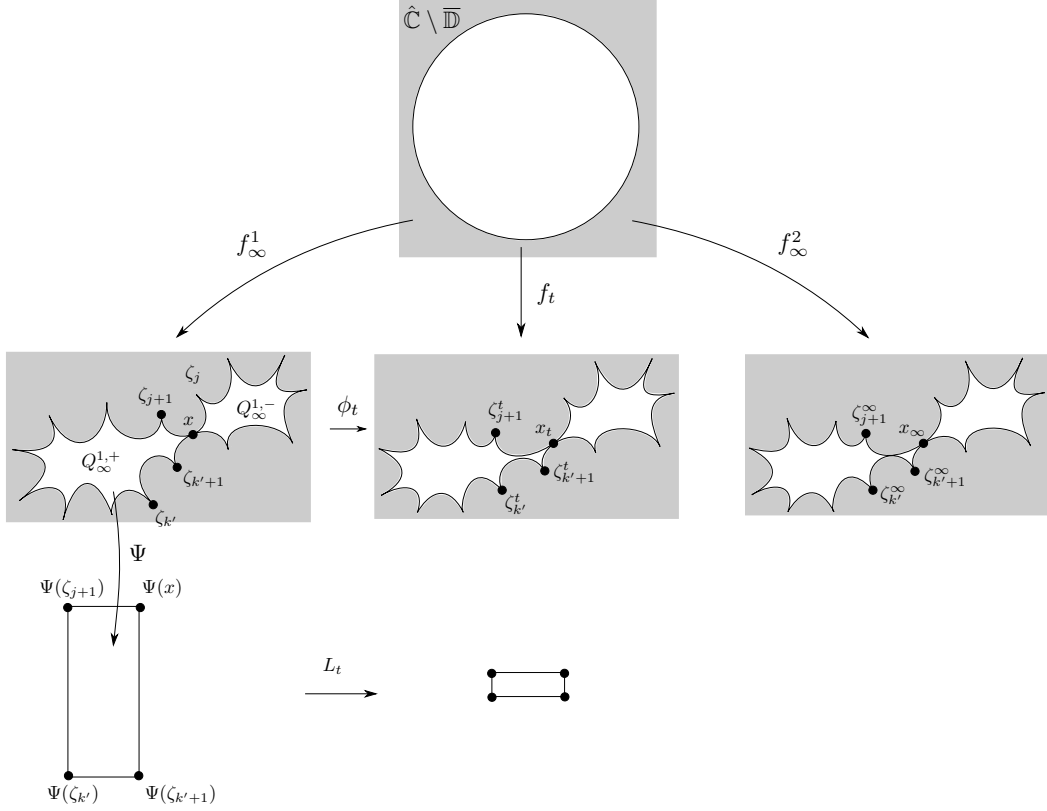}
\caption{ This Figure summarizes the second step in the proof of Theorem~\ref{mainthm_qd_terminology}. }
\label{fig:quasiconformal_deformation_step_two}      
\end{figure}

Denote by ${Q}^{1,+}_{\infty}$, ${Q}^{1,-}_\infty$ the two components of $\textrm{int}(\widehat{\mathbb{C}}\setminus\Omega^1_\infty)$. To ease notation, we will relabel the critical points of $f^1_{\infty}$ as $\xi_1, \cdots, \xi_{d+1}$ (previously denoted $\xi_1^{\infty}, \cdots, \xi_{d+1}^{\infty}$). Denote the corresponding critical values by $\zeta_1, \cdots, \zeta_{d+1}$ (previously denoted $\zeta_1^{\infty}, \cdots, \zeta_{d+1}^{\infty}$). Let $\{j',k'\}\in S\setminus\{\{j,k\}\}$. From the definition of $S$, it follows that both the curves \[ f^1_{\infty}(\arc{\xi_{j'}\xi_{j'+1}})\textrm{, } f^1_{\infty}(\arc{\xi_{k'}\xi_{k'+1}}) \] must have non-empty intersection with the boundary of a common component of $\textrm{int}(\widehat{\mathbb{C}}\setminus\Omega^1_\infty)$, say ${Q}_{\infty}^{1,+}$. We consider the quadrilateral ${Q}_{\infty}^{1,+}$ with vertices \[{Q}_{\infty}^{1,+}(\zeta_{j'}, \zeta_{j'+1}, \zeta_{k'}, \zeta_{k'+1}) \textrm{ if } j', k' \not\in\{j,k\} \] and, otherwise, if say $j'=j$, we consider the quadrilateral \[ {Q}_{\infty}^{1,+}(x, \zeta_{j'+1}, \zeta_{k'}, \zeta_{k'+1}), \] where $x$ is the unique double point of Proposition~\ref{d+1_cusp_1_double}. We henceforth consider the case that $j'=j$ (the other setting is similar). Again we consider a conformal map \[\Psi: {Q}_{\infty}^{1,+} \rightarrow [0,M({Q}_{\infty}^{1,+})]\times[0,1] \] which preserves vertices, and define a Beltrami coefficient $\mu_t$ in ${Q}_{\infty}^{1,+}$ by pulling back \[ \nu_t=(L_t)_{\overline{z}}/(L_t)_{z}\equiv(t-1)/(t+1)\textrm{, } t\in[1,\infty)\] under $\psi$: $\mu_t(z)=\psi^*(\nu_t(\psi(z)))$ for $z\in{Q}_{\infty}^{1,+}$. For $z\in Q_{\infty}^{1,-}$, define $\mu_t(z)=0$. Given the definition of $\mu_t$ in $\textrm{int}(\widehat{\mathbb{C}}\setminus\Omega_\infty^1)$, we define as before: \[ \mu_t(\sigma^{-n}(z) )= (\sigma^{\circ n})^*( \mu_t(z) ) \textrm{ for } z\in\textrm{int}(\widehat{\mathbb{C}}\setminus\Omega_\infty^1)  \textrm{ and } n\geq 1, \] where $\sigma$ is the Schwarz reflection map associated to $\Omega^1_\infty$. Lastly, $\mu_t$ is defined to be $0$ elsewhere. Analogous  proofs of Propositions \ref{normalizing_rational_maps_general} - \ref{conformal_limit_general} hold to yield:

\begin{prop}\label{final_normalizing_rational_maps_general} With notation as above, there exist a family of quasiconformal maps $(\phi_t)_{t\in[1,\infty)}:\widehat{\C}\rightarrow\widehat{\C}$, and a family of rational maps $(f_t)_{t\in[1,\infty)}$ such that $(\phi_t)_{\overline{z}}/(\phi_t)_{z}=\mu_t$ a.e., $\phi_t(\infty)=\infty$, $f_1=f_\infty^1$, $f_t\in\Sigma_d^*$, and $f_t(\widehat{\mathbb{C}}\setminus\overline{\D})=\phi_t(\Omega_\infty^1)$ for all $t\in[1,\infty)$. Moreover, there exist a map ${f}^2_{\infty}\in\Sigma_d^*$ and a sequence of positive real numbers $(t_n)\nearrow\infty$ such that \[(f_{t_n})\xrightarrow[]{n\rightarrow\infty} f^2_{\infty} \textrm{  uniformly on   }\widehat{\mathbb{C}} \] with respect to the spherical metric. \end{prop}

\noindent We define $\Omega^2_\infty:=f^2_{\infty}(\widehat{\mathbb{C}}\setminus\overline{\mathbb{D}})$, and show:

\begin{prop}\label{d+1_cusp_2_double}
The boundary of the quadrature domain $\Omega^2_\infty$ has $d+1$ cusps and $2$ double points.
\end{prop}

\begin{proof} We introduce some notation (see also Figure \ref{fig:quasiconformal_deformation_step_two}).  Proposition~\ref{crit_points_on_circle} implies that each $f_t$ has $d+1$ distinct critical points on $\mathbb{T}$. We label these critical points in counterclockwise order as $\xi_1^t, \cdots, \xi_{d+1}^t$, with $\xi_1^t$ having the minimal argument where we take the branch $0\leq\arg(z)<2\pi$. We denote $\zeta_i^t=f_t(\xi_i^t)$, and recall the notation $\zeta_i=f_{\infty}^1(\xi_i)$. As in the proof of Proposition \ref{d+1_cusp_1_double}, we have that \begin{equation}\label{dependence_on_t2}\zeta_i^t=\phi_t(\zeta_i) \textrm{ for } 1\leq i \leq d+1 \textrm{ and } 1\leq t<\infty.\end{equation} By taking a further subsequence in $n$ if necessary, we have that \[ \xi_1^{t_n} \xrightarrow{n\rightarrow\infty} \xi_1^\infty, \cdots, \xi_{d+1}^{t_n} \xrightarrow{n\rightarrow\infty} \xi_{d+1}^\infty\textrm{} \] for some $\xi_1^\infty, \cdots, \xi_{d+1}^\infty \in \partial\mathbb{D}$. Since $f_{t_n}\rightarrow f^2_{\infty}$ uniformly in the spherical metric, it follows that $\xi_1^\infty, \cdots, \xi_{d+1}^\infty$ are critical points for $f^2_{\infty}$. Moreover,  $\xi_1^\infty, \cdots, \xi_{d+1}^\infty$ are distinct since $f^2_{\infty}\in\Sigma_d^*$.

Let $x_t:=\phi_t(x)$. Note that $x_t$ is a double point of the quadrature domain $\phi_t(\Omega^1_\infty)$ for each $t$. Let $\xi_+^t$, $\xi_-^t \in \mathbb{T}$ denote the preimages of $x_t$ under $f_t$. By taking a further subsequence in $n$ if necessary, \[ \xi^{t_n}_+ \rightarrow \xi^{\infty}_+\textrm{, and } \xi^{t_n}_- \rightarrow \xi^{\infty}_- \textrm{ as } n\rightarrow \infty,\] for $ \xi^{\infty}_+$, $\xi^{\infty}_- \in \mathbb{T}$ respectively.  By uniform convergence of $f_{t_n}\rightarrow f^2_{\infty}$, it follows that $f^2_{\infty}(\xi^{\infty}_+)=f^2_{\infty}(\xi^{\infty}_-)$. We claim that $\xi^{\infty}_+\not=\xi^{\infty}_-$. Indeed, $\arc{\xi^{t_n}_+\xi^{t_n}_-}$ contains at least two critical points of $f_{t_n}$ for each $n$, so that $\xi^{\infty}_+=\xi^{\infty}_-$ would imply that the common value $\xi^{\infty}_+=\xi^{\infty}_-$ is a double critical point for $f^2_{\infty}$, however this violates conformality of $f^2_\infty$ in $\mathbb{C}\setminus\mathbb{D}$. Thus $\xi^{\infty}_+\not=\xi^{\infty}_-$, and so $x_\infty:=f^2_{\infty}(\xi^{\infty}_+)=f^2_{\infty}(\xi^{\infty}_-)$ is a double point of $\partial\Omega^2_\infty$.

Curvature considerations yield that neither $\xi^{\infty}_+$, nor $\xi^{\infty}_-$ can be cusps, and since $\arc{\xi^{t_n}_+\xi^{t_n}_-}$ contains at least two critical points of $f_{t_n}$ for each $n$, it follows that $\xi^{\infty}_+$, $\xi^{\infty}_-$ are distinct, non-critical points on the unit circle. Thus $f^2_{\infty}(\xi^{\infty}_+)=f^2_{\infty}(\xi^{\infty}_-)$ is a double point of $\phi_\infty(\Omega^1_\infty)$.

The same arguments as in Proposition \ref{d+1_cusp_1_double} apply verbatim to show that precisely one double point is introduced in \[ \phi_{t_n}(Q^{1,+}_{\infty})\left( x_{t_n}, \zeta_{j'+1}^{t_n}, \zeta_{k'}^{t_n}, \zeta_{k'+1}^{t_n} \right) \] as $n\rightarrow\infty$ corresponding to an intersection point of the arcs \[   f_\infty^2(\arc{\xi_{+}^{\infty} \xi_{j'+1}^{\infty}}), f_\infty^2(\arc{\xi_{k'}^{\infty}\xi_{k'+1}^{\infty})}.   \] Thus it remains to be shown that no double points are introduced in  $\phi_{t_n}(Q^{1,-}_{\infty})$  as  $n\rightarrow\infty$. As before, we consider all such a priori possible double points and rule out each, in turn:

\begin{itemize} 

\item An arc $f^1_\infty(\arc{\xi^i_\infty\xi^{i+1}_\infty})$ has a self-intersection.

\subitem This is disallowed because of Principle \ref{constant_curvature_principle} (constant curvature principle). 

\item Two adjacent arcs $f^1_{\infty}(\arc{\xi^i_\infty\xi^{i+1}_\infty})\textrm{, } f^1_{\infty}(\arc{\xi^{i+1}_\infty\xi^{i+2}_\infty})$ intersect at more than their common endpoint.

\subitem This is disallowed because of Principle \ref{constant_curvature_principle} (constant curvature principle). 

\item Non-adjacent arcs $f^2_{\infty}(\arc{\xi_m^\infty\xi_{m+1}^\infty})$,  $f^2_{\infty}(\arc{\xi_l^\infty\xi_{l+1}^\infty})$ have non-empty intersection, where both curves $f^1_{\infty}(\arc{\xi_m\xi_{m+1}})$,  $f^1_{\infty}(\arc{\xi_l\xi_{l+1}})$ lie on the boundary of $Q_\infty^{1,-}$.

\subitem This follows from conformal invariance of extremal length. Indeed, consider the quadrilateral $\phi_{t_n}(Q_\infty^{1,-})(\zeta_m^{t_n}, \zeta_{m+1}^{t_n}, \zeta_l^{t_n}, \zeta_{l+1}^{t_n})$, and let $s_a(n)$, $s_b(n)$ denote the Euclidean path-distance between its a, b-sides. If $f^2_{\infty}(\arc{\xi_m^\infty\xi_{m+1}^\infty})$,  $f^2_{\infty}(\arc{\xi_l^\infty\xi_{l+1}^\infty})$ have non-empty intersection, then $s_a(n) \rightarrow 0$ but $s_b(n) \not \rightarrow 0$ as $n\rightarrow\infty$. Thus by Theorem \ref{modulus_estimate}, \[ M\left(\phi_{t_n}(Q_\infty^{1,-})(\zeta_m^{t_n}, \zeta_{m+1}^{t_n}, \zeta_l^{t_n}, \zeta_{l+1}^{t_n})\right) \rightarrow \infty  \textrm{ as } t\rightarrow\infty.\] However, $\phi_{t_n}$ is conformal on $Q_\infty^{1,-}$ for each $n$, so that \begin{equation}\label{conformal_invariance} M\left(\phi_{t_n}(Q_\infty^{1,-})(\zeta_m^{t_n}, \zeta_{m+1}^{t_n}, \zeta_l^{t_n}, \zeta_{l+1}^{t_n})\right) = M \left( Q_\infty^{1,-}(\zeta_m, \zeta_{m+1}, \zeta_l, \zeta_{l+1}) \right) \end{equation} for all $n$, and this is our needed contradiction since the right-hand side of (\ref{conformal_invariance}) is finite and does not depend on $n$.

\item One of the arcs $f^2_{\infty}(\arc{\xi_{-}^{\infty} \xi_{k+1}^{\infty}})$, $f^2_{\infty}(\arc{\xi_{j}^{\infty} \xi_{+}^{\infty}})$ has nonempty intersection with a non-adjacent arc $f^2_\infty(\arc{\xi_i^\infty\xi_{i+1}^\infty})$.

\subitem The same argument as above applies.

\end{itemize}
\end{proof}

As already mentioned, iterating the above procedure $|S|$ many times yields a rational function \[f_\infty\equiv f^{|S|}_\infty\in\Sigma_d^*\] such that the quadrature domain \[\Omega_\infty:=f_\infty(\widehat{\mathbb{C}}\setminus\overline{\mathbb{D}})\] has precisely $d+1$ cusps and $|S|$ double points. 

Let $\mathcal{T}_\infty$ denote the tree associated to $\Omega_\infty$ (see Definition \ref{associated_tree}). We finish the proof of Theorem \ref{mainthm_qd_terminology} by showing that $\mathcal{T}_{\infty}$ is isomorphic to $\mathcal{T}$ (see also Figure \ref{fig:orientation}). For each fundamental tile $T$ of $\Omega_\infty$, there is a unique set of three distinct integers $j$, $k$, $l \in \{1, \cdots, d+1\}$ such that the pairwise intersections of $f_\infty(I_{j})$, $f_\infty(I_{k})$, $f_\infty(I_{l})$ form the three singular points of $\partial T$. Moreover, by construction, $f_\infty(I_{j})$, $f_\infty(I_{k})$, $f_\infty(I_{l})$ have pairwise non-empty intersection if and only if $\phi(J_j)$, $\phi(J_k)$, $\phi(J_l)$ intersect at a common vertex $v$ of $\mathcal{T}$ (see Definition~\ref{definition_of_S}). Thus, this defines a natural bijection $T \mapsto v$ between the vertices of $\mathcal{T}_\infty$ and the vertices of $\mathcal{T}$. Moreover, this bijection naturally extends to the vertex sets of the augmented trees. It is straightforward to check that this bijection is indeed an isomorphism of the augmented trees (namely that it preserves edge structure, and cyclic ordering of edges). By Lemma \ref{angled_augmented_equiv_lem}, it follows that $\mathcal{T}_\infty$ and $\mathcal{T}$ are isomorphic.

\begin{figure}
\centering
\scalebox{.22}{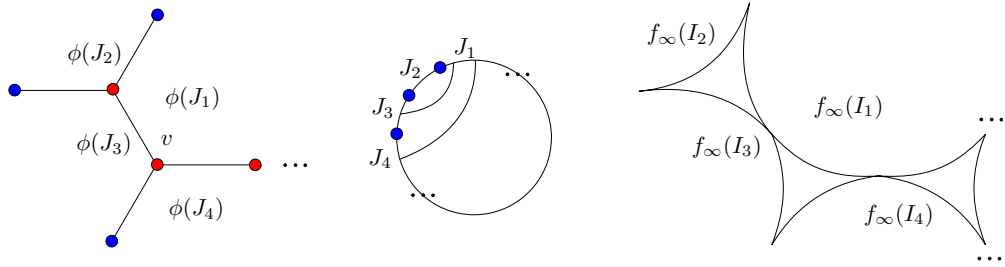}
\caption{Pictured on the left is part of a bi-angled tree $\mathcal{T}$. Also shown is $\widehat{\mathcal{T}}$, where only the degree 1 blue vertices of $\widehat{\mathcal{T}}$ have been illustrated for clarity. In the middle are shown the arcs $J_k$ on $\mathbb{T}$ and the pre-images (under $\phi$) of the degree 1 blue vertices of $\widehat{\mathcal{T}}$. Elements of $S=S_{\mathcal{T}}$ have been illustrated by lines connecting the arcs $J_k$. On the right is shown part of the quadrature domain $\Omega_\infty$ obtained in Theorem \ref{mainthm_qd_terminology}. Note that the correct oriented angle between edges in $T_\infty$ is ensured by the identifications in $S$.}
\label{fig:orientation}      
\end{figure}

\end{proof}

Lastly, we remark that the technical work done in Section~\ref{mainthm_proof} may be summarized as proving the following Theorem:

\begin{thm}\label{pinching_theorem} \emph{(Pinching Theorem)} Let $f\in\Sigma_d^*$. Choose a counter-clockwise labelling $\xi_1, \cdots, \xi_{d+1}$ of the (necessarily distinct) critical points of $f$ on $\mathbb{T}$. Suppose that \[f(\arc{\xi_{j}\xi_{j+1}}), f(\arc{\xi_{k}\xi_{k+1}})\] do not intersect, but both intersect the boundary of a common bounded component of $\widehat{\C}\setminus f(\mathbb{T})$. 

Then there exist $f_{\infty}\in\Sigma_d^*$ and a counter-clockwise labelling  $\xi_1^\infty, \cdots, \xi_{d+1}^\infty$ of the (necessarily distinct) critical points of $f_\infty$ on $\mathbb{T}$ such that the arcs \[f_\infty(\arc{\xi_{j}^\infty\xi_{j+1}^\infty}), f_\infty(\arc{\xi_{k}^\infty\xi_{k+1}^\infty})\] have a (necessarily unique) intersection point. Furthermore, if $\{m, n\}\not=\{j,k\}$, then \[f_\infty(\arc{\xi_{m}^\infty\xi_{m+1}^\infty})\textrm{, } f_\infty(\arc{\xi_{n}^\infty\xi_{n+1}^\infty}) \textrm{ intersect if and only if } f(\arc{\xi_{m}\xi_{m+1}})\textrm{, } f(\arc{\xi_{n}\xi_{n+1}}) \textrm{ intersect. }\]

\end{thm}

\section{Bi-angled Trees Determine Extremal Quadrature Domains}\label{rigidity_qd_sec}

The goal of this section is to prove that the bi-angled tree associated with an extremal unbounded quadrature domain uniquely determines the quadrature domain up to an affine map.

\begin{thm}\label{rigidity_extremal_qd_thm}
Let $\widetilde{\Omega}$ and $\Omega$ be two unbounded extremal quadrature domains such that $\mathcal{T}( \widetilde{\Omega})$ and $\mathcal{T}(\Omega)$ are isomorphic as bi-angled trees. Then, there exists an affine map $A$ with $A( \widetilde{\Omega})=\Omega$.
\end{thm}

\begin{rem}\label{thm_equiv_2_rem} By Proposition~\ref{suffridge_extremal_qd_equiv_prop}, Theorem~\ref{rigidity_extremal_qd_thm} is equivalent to injectivity of the map (1)$\mapsto$(2) in Theorem~\ref{theorem_A}.
\end{rem}

\begin{rem}\label{conf_remove_rem}
It follows from a conformal removability result that the only extremal unbounded quadrature domain of order $2$ (up to M{\"o}bius equivalence) is the exterior of a deltoid curve (compare \cite[\S 5.1]{LLMM1}). More precisely, if $\widetilde{\Omega}$ is an unbounded extremal quadrature domains of order $2$ and $\Omega$ is the exterior of the deltoid curve, we can first map the complement of one to that of the other by a conformal map that preserves cusps (note that there are only $3$ cusps and no double points on their boundaries), and then lift it using the corresponding Schwarz reflections to obtain a conformal isomorphism between the tiling sets. Furthermore, there is a conformal isomorphism between the basins of attractions (of $\infty$) of the two Schwarz reflection maps that conjugates the dynamics. These two conformal maps agree continuously on the boundary of the tiling set (which is a Jordan curve, see \cite[Theorem~5.11]{LLMM1}) producing a homeomorphism of $\widehat{\C}$ that is conformal off the boundary of the tiling set. It now follows from conformal removability of the boundary of the tiling set (see \cite[Corollary~5.22]{LLMM1}) that this homeomorphism is, in fact, a M{\"o}bius map; and hence, the quadrature domains $\widetilde{\Omega}$ and $\Omega$ are M{\"o}bius equivalent. 

On the other hand, the lack of corresponding conformal removability results for $d\geq 3$ forces us to use dynamical tools to prove the above rigidity theorem. More precisely, we prove rigidity of an extremal unbounded quadrature domain (of order $d\geq3$) by demonstrating rigidity of the corresponding Schwarz reflection map using a `pullback argument'. 
\end{rem}

\subsection{Proof of Theorem~\ref{rigidity_extremal_qd_thm}}\label{rigidity_thm_subsec}

Let $\Omega$ be an extremal unbounded quadrature domain of order $d$. By Proposition~\ref{suffridge_extremal_qd_equiv_prop}, possibly after replacing $\Omega$ by its image under an affine map, we can assume that there exists a rational map $g\in\Sigma_d^*$ of degree $d+1$ (in fact, a Suffridge polynomial) such that $g:\widehat{\C}\setminus\overline{\D}\to\Omega$ is a conformal isomorphism. Since $g$ has a $(d-1)$-fold pole at the origin, it follows that the Schwarz reflection map $\sigma$ of $\Omega$ fixes $\infty$, and has a critical point of order $d-1$ at $\infty$. In particular, $\infty$ is a super-attracting fixed point of $\sigma$. Moreover, as all critical points of $g$ lie on $\mathbb{T}$ and at the origin, $\infty$ is the only critical point of the Schwarz reflection map $\sigma$.

Let us set $T=\widehat{\C}\setminus\Omega$, and $T^0=T\setminus\{$Singular points on $\partial T\}$. The connected components of $T^0$ are labelled as $T_1,\cdots,T_{d-1}$. It follows from Proposition~\ref{s.c.q.d.} that $\sigma:\sigma^{-1}(\Omega)\to\Omega$ is a proper branched covering of degree $d$ (branched only at $\infty$), and $\sigma:\sigma^{-1}(\interior{T})\to\interior{T}$ is a covering of degree $d+1$.

By Proposition~\ref{basin_topology} and Corollary~\ref{tiling_set_full}, the tiling set \[T^\infty\equiv T^\infty(\sigma)=\displaystyle\bigcup_{j=0}^\infty\sigma^{-j}(T^0)\] is open, and its closure is a full continuum (in fact, for $d\geq 3$, the desingularized droplet $T^0$ has more than one connected component, and hence the tiling set $T^\infty$ has infinitely many connected components; but adding the singular points $\displaystyle\cup_{j=0}^\infty\sigma^{-j}(T\setminus T^0)$ to $T^\infty$ makes it connected). 

Recall that we denote the basin of attraction of the fixed point $\infty$ by $\mathcal{B}_\infty(\sigma)$. By Proposition~\ref{basin_topology}, $\mathcal{B}_\infty(\sigma)$ is a simply connected, completely invariant domain (under $\sigma$). Hence, there exists a conformal isomorphism (which we call a ``B{\"o}ttcher  coordinate'') $$\tau:\widehat{\C}\setminus\overline{\D}\to\mathcal{B}_\infty(\sigma)$$ that conjugates $\overline{z}^d$ to $\sigma$. The conformal map $\tau$ is unique up to pre-composition by a $(d+1)$-st root of unity. Moreover, by Proposition~\ref{dynamical_partition_schwarz}, we have that $$\widehat{\C}=\mathcal{B}_\infty(\sigma)\sqcup\overline{T^\infty}$$ (see Figure~\ref{extremal_QD_schwarz}).

Let us now assume that $\Omega,  \widetilde{\Omega}$ are two extremal unbounded quadrature domains of order $d$. For any object associated with the domain $\Omega$, the corresponding object associated with $ \widetilde{\Omega}$ will be denoted with a $``\ \widetilde{}\ $''.

\begin{lem}\label{angled_tree_conformal_map}
Let  $ \widetilde{\Omega}$ and $\Omega$ be two unbounded extremal quadrature domains such that $\mathcal{T}( \widetilde{\Omega})$ and $\mathcal{T}(\Omega)$ are isomorphic as bi-angled trees. Then, there exists a cusp-preserving homeomorphism $\mathbf{\Psi}:\widetilde{T}\to T$ that maps $\interior{\widetilde{T}_i}$ conformally onto $\interior{T_i}$ for $1\leq i \leq d-1$.
\end{lem}
\begin{proof}
By the assumption and Lemma~\ref{angled_augmented_equiv_lem}, the corresponding augmented trees are also isomorphic (in the sense of Definition~\ref{isom_def_1}). Since each red vertex of the augmented trees represents some $T_i$ or $\widetilde{T}_i$, we can re-label the components $\{\widetilde{T}_1,\cdots,\widetilde{T}_{d-1}\}$ such that the red vertex corresponding to $\widetilde{T}_i$ is mapped to the red vertex corresponding to $T_i$ by the isomorphism between their augmented trees.

As the isomorphism between the augmented trees preserves the circular order of edges meeting at each vertex, it induces a homeomorphism $$\mathbf{\Psi}:\widetilde{T}=\displaystyle\bigcup_{i=1}^{d-1}\overline{\widetilde{T}_i}\to T=\displaystyle\bigcup_{i=1}^{d-1}\overline{T_i}$$ that maps $\partial\widetilde{T}_i$ onto $\partial T_i$ as an orientation preserving homeomorphism (note that each $\partial\widetilde{T}_i$ and $\partial T_i$ is a topological triangle). In particular, $\mathbf{\Psi}$ sends the three singular points on $\partial\widetilde{T}_i$ to the three singular points on $\partial T_i$ preserving their circular order. Since each $\overline{\widetilde{T}_i}$ and $\overline{T_i}$ is a closed Jordan disk, the Riemann mapping theorem (along with Carath{\'e}odory's theorem) implies that $\mathbf{\Psi}$ can be chosen to be conformal on $\interior{\widetilde{T}_i}$ and preserve cusps, for each $i\in\{1,\cdots, d-1\}$.
\end{proof}

We continue to work with the hypotheses of the previous proposition. It follows from the construction of $\mathbf{\Psi}$ that it maps the singular points on $\partial\widetilde{T}$ onto those of $\partial T$. The next proposition studies the asymptotics of the map $\mathbf{\Psi}$ near the singular points on $\partial\widetilde{T}$.

\begin{lem}\label{asymp_linear}
$\mathbf{\Psi}$ is asymptotically linear near the singular points on $\partial\widetilde{T}$. More precisely, if $\widetilde{\zeta}_0$ is a singular point on $\partial\widetilde{T}$ with $\zeta_0=\mathbf{\Psi}(\widetilde{\zeta}_0)$, then 
\begin{equation}
\mathbf{\Psi}(\widetilde{\zeta})=\zeta_0+c_1(\widetilde{\zeta}-\widetilde{\zeta}_0)+o((\widetilde{\zeta}-\widetilde{\zeta}_0)) \textrm{ as } \widetilde{\zeta}\rightarrow\widetilde{\zeta}_0
\label{asymp_linear_1}
\end{equation}
where $c_1\in\C^\ast$.
\end{lem}
\begin{proof}
The proof is similar to that of \cite[Lemma~6.10]{LLMM2}. However, as we are dealing with a different setting here, we work out the details for completeness.

Let us first prove the statement for a cusp point $\widetilde{\zeta}_0\in\partial\widetilde{T}$. We set $\zeta_0=\mathbf{\Psi}(\widetilde{\zeta}_0)$. By Proposition~\ref{cusp_geometry}, $\widetilde{\zeta}_0$ (respectively, $\zeta_0$) is a $(3,2)$ cusp. Thus, we can send $\widetilde{\zeta}_0$ (respectively, $\zeta_0$) to $\infty$ by a M{\"o}bius map such that the image of $\widetilde{T}$ near $\widetilde{\zeta}_0$ (respectively, the image of $T$ near $\zeta_0$) is a curvilinear strip $\widetilde{\mathfrak{S}}$ (respectively, $\mathfrak{S}$) bounded by the non-singular real-analytic curves $y=v_1(x)=k_1+k_2\sqrt{x}+k_3/\sqrt{x}+o(1/\sqrt{x})$ and $y=v_2(x)=k_1-k_2\sqrt{x}-k_3/\sqrt{x}+o(1/\sqrt{x})$, where $x\geq x_0$ for some large $x_0>0$ (with possibly different constants for $\mathfrak{S}$). Note that as $v_i$ is a real-algebraic curve (of negative curvature), we have that: \begin{itemize}
\item\label{property_1} $v_i'$ is continuous and of bounded variation on $\left[x_0,+\infty\right]$ ($i=1,2$), and

\item\label{property_2} the strip $\widetilde{\mathfrak{S}}$ (respectively, $\mathfrak{S}$) has \emph{boundary inclination} $0$ at $\infty$ (in the sense of \cite{War}) since $v_i'(x)\to 0$ as $x\to+\infty$ ($i=1,2$). 
\end{itemize}
We define $\theta_1(x):=v_1(x)-v_2(x)=2k_2\sqrt{x}+2k_3/\sqrt{x}+o(1/\sqrt{x})$, and $\theta_2(x):=(v_2(x)+v_1(x))/2=k_1+o(1/\sqrt{x})$.   

Let us now consider a conformal map $\widetilde{\mathfrak{b}}$ (respectively, $\mathfrak{b}$) from the curvilinear strip $\widetilde{\mathfrak{S}}$ (respectively, $\mathfrak{S}$) to the right-half of the horizontal strip $\{y\in\C:\vert y\vert<\pi/2\}$. Since the boundary curves $v_i$ ($i=1,2$) of the strip $\widetilde{\mathfrak{S}}$ (respectively, $\mathfrak{S}$) satisfy the above two properties, we can invoke \cite{War} to conclude that the conformal map $\widetilde{\mathfrak{b}}$ (respectively, $\mathfrak{b}$) from the curvilinear strip $\widetilde{\mathfrak{S}}$ (respectively, $\mathfrak{S}$) to the horizontal strip is of the form
\begin{equation}
z=x+iy\mapsto k_4+\pi\int_{x_0}^x\frac{1+(\theta_2'(t))^2}{\theta_1(t)} dt +i\pi\frac{y-\theta_2(x)}{\theta_1(x)}+o(1)=k_5\sqrt{z}+O(1),
\label{asymp_riemann_cusp}
\end{equation}
as $\re(z)\to+\infty$, where $k_5\in\C^\ast$ (with possibly different constants for $\mathfrak{b}$). Moreover, as $v_i'$ is continuous and of bounded variation on $\left[x_0,+\infty\right]$, the asymptotic expression~(\ref{asymp_riemann_cusp}) holds uniformly with respect to $\im(z)$.

It follows that the conformal map $\mathfrak{b}^{-1}\circ\widetilde{\mathfrak{b}}$ between the curvilinear strips $\widetilde{\mathfrak{S}}$ and $\mathfrak{S}$ admits the asymptotics $k_6z+O(1)$ as $\re(z)\to+\infty$,  for some $k_6\in\C^\ast$ (uniformly in $\im(z)$). Due to local connectivity of the above curvilinear strips, the conformal map  $\mathfrak{b}^{-1}\circ\widetilde{\mathfrak{b}}$ extends continuously to the boundary. Furthermore, since the above asymptotics is uniform in $\im(z)$, the boundary extension admits the same asymptotics $k_6z+O(1)$ as $\re(z)\to+\infty$. Now going back to $\widetilde{T}$ and $T$ by the M{\"o}bius maps used earlier, we conclude that $\mathbf{\Psi}$ is asymptotically linear near $\widetilde{\zeta}_0$; i.e., 
$$
\mathbf{\Psi}(\widetilde{\zeta})=\zeta_0+k_7(\widetilde{\zeta}-\widetilde{\zeta}_0)+o((\widetilde{\zeta}-\widetilde{\zeta}_0)),
$$ 
for $\partial\widetilde{T}\ni\widetilde{\zeta}\to\widetilde{\zeta}_0$, and some constant $k_7\in\C^\ast$.

Let us now prove the assertion for a double point $\widetilde{\zeta}_0\in\partial\widetilde{T}$. As before, we set $\zeta_0=\mathbf{\Psi}(\widetilde{\zeta}_0)$. By Proposition~\ref{double_geometry}, the two distinct non-singular branches of $\partial\widetilde{T}$ (respectively, of $\partial T$) have distinct osculating circles at $\widetilde{\zeta}_0$ (respectively, at $\zeta_0$). Let us denote the two (distinct) osculating circles to $\partial\widetilde{T}$ (respectively, $\partial T$) at $\widetilde{\zeta}_0$ (respectively, $\zeta_0$) by $\widetilde{\mathbf{C}}_{1}$ and $\widetilde{\mathbf{C}}_2$ (respectively, $\mathbf{C}_1$ and $\mathbf{C}_2$). We can send $\widetilde{\zeta}_0$ (respectively, $\zeta_0$) to $\infty$ by a M{\"o}bius map such that the images of $\widetilde{\mathbf{C}}_{1}$ and $\widetilde{\mathbf{C}}_2$ (respectively, $\mathbf{C}_1$ and $\mathbf{C}_2$) are the horizontal straight lines $y=0$ and $y=1$. Since the two distinct non-singular branches of $\partial\widetilde{T}$ (respectively, $\partial T$) that meet tangentially at $\widetilde{\zeta}_0$ (respectively, at $\zeta_0$) have at least second order contact with their corresponding osculating circles, the images of these two distinct non-singular branches of $\partial\widetilde{T}$ (respectively, $\partial T$) under the above M{\"o}bius map are curves of the form $y=w_1(x)=0+O(1/x)$ and $y=w_2(x)=1+O(1/x)$ for $x$ large enough. Thus, the image of $\widetilde{T}$ near $\widetilde{\zeta}_0$ (respectively, the image of $T$ near $\zeta_0$) under the above M{\"o}bius map is a curvilinear strip bounded by $y=0+O(1/x)$ and $y=1+O(1/x)$, where $x\geq x_0$ for some large $x_0>0$. As in the previous case, $w_i'$ is continuous and of bounded variation on $\left[x_0,+\infty\right]$, and the strips have \emph{boundary inclination} $0$ at $\infty$ (in the sense of \cite{War}) since $w_i'(x)\to 0$ as $x\to+\infty$ ($i=1,2$). 

Let us now consider a conformal map $\widetilde{\mathfrak{b}}_1$ (respectively, $\mathfrak{b}_1$) from the above curvilinear strip to the right-half of the horizontal strip $\{y\in\C:\vert y\vert<\pi/2\}$. Applying \cite{War}, we now conclude that the conformal map $\widetilde{\mathfrak{b}}_1$ (respectively, $\mathfrak{b}_1$) from the curvilinear strip to the horizontal strip is of the form $k_{8}z+o(z)$ as $\re(z)\to+\infty$, where $k_{8}\in\C^\ast$. Moreover, as $w_i'$ is continuous and of bounded variation on $\left[x_0,+\infty\right]$, the above asymptotic expression holds uniformly with respect to $\im(z)$. It now follows that the conformal map $\mathfrak{b}_1^{-1}\circ\widetilde{\mathfrak{b}}_1$ between the curvilinear strips (obtained by sending the double points on $\partial\widetilde{T}$ and $\partial T$ to $\infty$) is of the form $z+o(z)$ as $\re(z)\to+\infty$ (uniformly in $\im(z)$). Once again, due to local connectivity of the above curvilinear strips, $\mathfrak{b}_1^{-1}\circ\widetilde{\mathfrak{b}}_1$ extends continuously to the boundary of the curvilinear strip. Since the above asymptotics hold uniformly in $\im(z)$, the boundary extension also admits the same asymptotics $z+o(z)$ as $\re(z)\to+\infty$. Finally, going back to $\widetilde{T}$ and $T$ by the M{\"o}bius maps used earlier, we conclude that the asymptotics of $\mathbf{\Psi}$ near $\widetilde{\zeta}_0$ is of the form (\ref{asymp_linear_1}).
\end{proof}

\begin{figure}[ht!]
\begin{tikzpicture}
\node[anchor=south west,inner sep=0] at (0,0) {\includegraphics[width=0.9\textwidth]{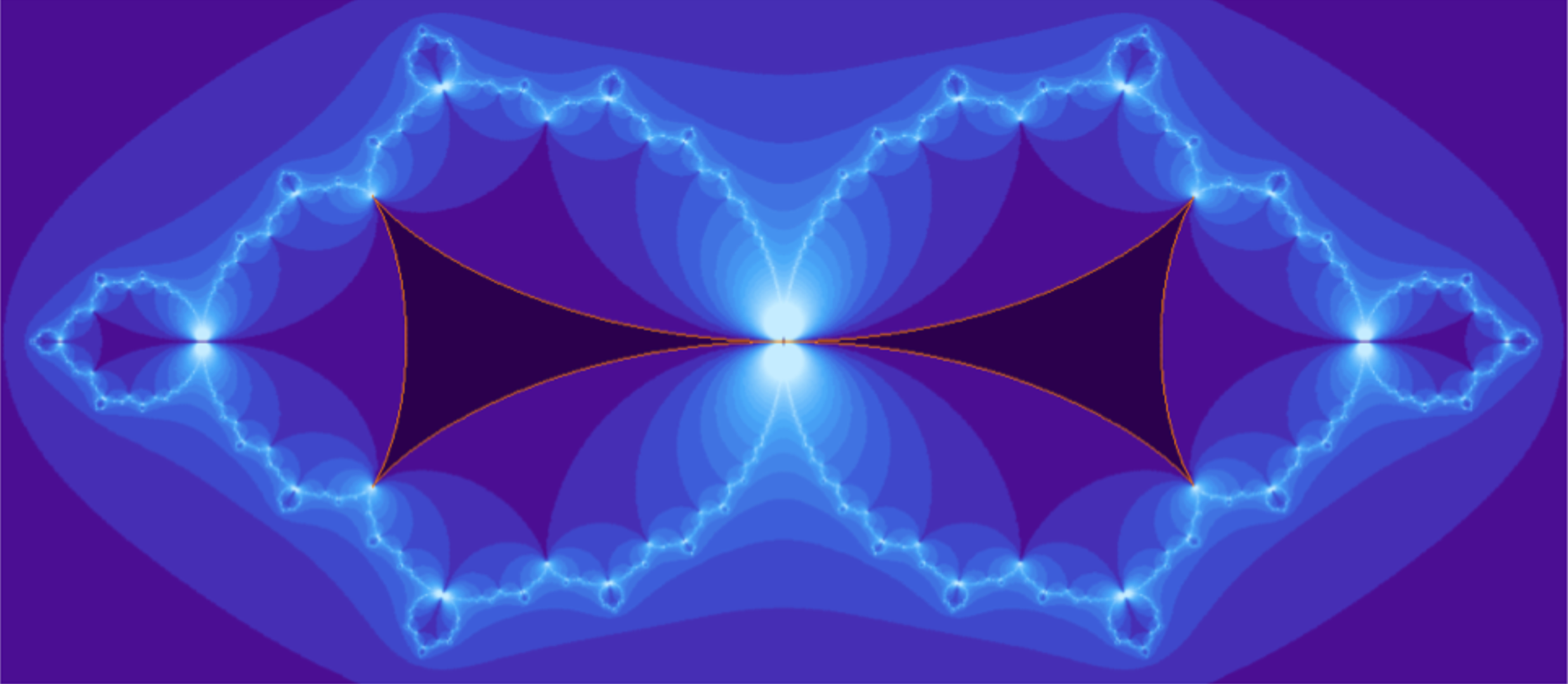}};
\node at (4,2.8) {\color{white}$T_2$};
\node at (8.4,2.8) {\color{white}$T_1$};
\end{tikzpicture}
\caption{The dynamical plane of the Schwarz reflection map $\sigma$ associated with the extremal unbounded quadrature domain realizing the unique bi-angled tree with $2$ vertices. The limit set of $\sigma$ is the boundary of the tiling set.}
\label{extremal_QD_schwarz}
\end{figure}

The above analysis of the asymptotics of $\mathbf{\Psi}$ near the singular points of $\partial\widetilde{T}$ will allow us to extend it to a global quasiconformal map.

\begin{lem}\label{global_qc}
$\mathbf{\Psi}$ can be extended to a quasiconformal map on $\widehat{\C}$.
\end{lem}
\begin{proof}
Note that by Proposition~\ref{suffridge_extremal_qd_equiv_prop}, there exist $g,\widetilde{g}\in\Sigma_d^*$ such that $g(\widehat{\C}\setminus\overline{\D})=\Omega$, and $\widetilde{g}(\widehat{\C}\setminus\overline{\D})= \widetilde{\Omega}$. Moreover, both $g,\widetilde{g}$ have $d+1$ simple critical points on $\mathbb{T}$. Let us denote the set of all non-zero critical points of $\widetilde{g}$ (respectively, of $g$) by $\textrm{crit}^\ast(\widetilde{g})$ (respectively, by $\textrm{crit}^\ast(g)$), and the set of all points on $\mathbb{T}$ that are mapped by $\widetilde{g}$ (respectively, by $g$) to the double points on $\partial\widetilde{T}$ (respectively, on $\partial T$) by $\textrm{pinch}(\widetilde{g})$ (respectively, by $\textrm{pinch}(g)$).

Using $g:\mathbb{T}\to\partial\Omega=\partial T$ and $\widetilde{g}:\mathbb{T}\to\partial \widetilde{\Omega}=\partial\widetilde{T}$, we can lift the map $\mathbf{\Psi}:\partial\widetilde{T}\to\partial T$ to a homeomorphism $\widehat{\mathbf{\Psi}}:\mathbb{T}\to\mathbb{T}$ that sends $\textrm{crit}^\ast(\widetilde{g})$ onto $\textrm{crit}^\ast(g)$ and $\textrm{pinch}(\widetilde{g})$ onto $\textrm{pinch}(g)$.

\begin{center}
$\begin{CD}
(\mathbb{T},\widetilde{\xi}_0) @>\widehat{\mathbf{\Psi}}>> (\mathbb{T},\xi_0)\\
@VV \widetilde{g} V @VV g V\\
(\partial\widetilde{T},\widetilde{\zeta}_0) @>\mathbf{\Psi}>> (\partial T,\zeta_0)
\end{CD}$
\end{center}

We will now show that $\widehat{\mathbf{\Psi}}$ is a quasisymmetric map. Since the map $\mathbf{\Psi}$ admits a conformal extension to a neighborhood of $\widetilde{T}^0$ (the desingularized droplet), it follows that $\widehat{\mathbf{\Psi}}$ can be conformally extended to a neighborhood of $\mathbb{T}\setminus\left(\textrm{crit}^\ast(\widetilde{g})\cup\textrm{pinch}(\widetilde{g})\right)$. Hence, we only need to check the quasisymmetry property of $\mathbf{\Psi}$ near $\textrm{crit}^\ast(\widetilde{g})\cup\textrm{pinch}(\widetilde{g})$.

To this end, let us first pick $\widetilde{\xi}_0\in\textrm{pinch}(\widetilde{g})$, and set $\widetilde{\zeta}_0:=\widetilde{g}(\widetilde{\xi}_0)$. Then, $\xi_0:=\widehat{\mathbf{\Psi}}(\widetilde{\xi}_0)\in\textrm{pinch}(g)$, and $\zeta_0:=\mathbf{\Psi}(\widetilde{\zeta}_0)=g(\xi_0)$ is a double point on $\partial T$. By Lemma~\ref{asymp_linear}, the asymptotics of $\mathbf{\Psi}$ near $\widetilde{\zeta}_0$ is of the form (\ref{asymp_linear_1}). Moreover, $\widetilde{g}$ (respectively, $g$) has non-zero derivative at each point of $\textrm{pinch}(\widetilde{g})$ (respectively, of $\textrm{pinch}(g)$). Hence, the lifted map $\widehat{\mathbf{\Psi}}$ is asymptotically linear near $\widetilde{\xi}_0$; i.e., 
\begin{equation}
\widehat{\mathbf{\Psi}}(\widetilde{\xi})=\xi_0+c_2(\widetilde{\xi}-\widetilde{\xi}_0)+o((\widetilde{\xi}-\widetilde{\xi}_0)),
\label{asymp_linear_2}
\end{equation}
for $\mathbb{T}\ni\widetilde{\xi}\to\widetilde{\xi}_0$, and some constant $c_2\in\C^\ast$. This implies that if $I$ and $J$ are sub-arcs of $\mathbb{T}$ with $\vert I\vert=\vert J\vert$ and $I\cap J=\{\widetilde{\xi}_0\}$, then the ratio of the lengths of the images of $I$ and $J$ under $\widehat{\mathbf{\Psi}}$ is uniformly bounded. It follows that $\widehat{\mathbf{\Psi}}$ is quasisymmetric near points of $\textrm{pinch}(\widetilde{g})$.

Let us now pick $\widetilde{\xi}_0\in\textrm{crit}^\ast(\widetilde{g})$, and set $\widetilde{\zeta}_0:=\widetilde{g}(\widetilde{\xi}_0)$. Then, $\xi_0:=\widehat{\mathbf{\Psi}}(\widetilde{\xi}_0)\in\textrm{crit}^\ast(g)$, and $\zeta_0:=\mathbf{\Psi}(\widetilde{\zeta}_0)=g(\xi_0)$ is a cusp point on $\partial T$. Since $\widetilde{g}$ has a simple critical point at $\widetilde{\xi}_0$, the Taylor series of $\widetilde{g}$ at $\widetilde{\xi}_0$ is given by
\begin{equation}
\widetilde{g}(\widetilde{\xi})=\widetilde{\zeta}_0+c_3(\widetilde{\xi}-\widetilde{\xi}_0)^2+o((\widetilde{\xi}-\widetilde{\xi}_0)^2),
\label{asymp_quad}
\end{equation} for some constant $c_3\in\C^\ast$. Similarly, since $g$ has a simple critical point at $\xi_0$, the Puiseux series of $\left(g\vert_{\widehat{\C}\setminus\overline{\D}}\right)^{-1}$ at $\zeta_0$ is given by 
\begin{equation}
\left(g\vert_{\widehat{\C}\setminus\overline{\D}}\right)^{-1}(\zeta)=\xi_0+c_4(\zeta-\zeta_0)^{\frac12}+o((\zeta-\zeta_0)^\frac12),
\label{asymp_root}
\end{equation} for some constant $c_4\in\C^\ast$. Moreover, by Lemma~\ref{asymp_linear}, the asymptotics of $\mathbf{\Psi}$ near $\widetilde{\zeta}_0$ is of the form (\ref{asymp_linear_1}). It follows that $\widehat{\mathbf{\Psi}}$ has an asymptotics of the form (\ref{asymp_linear_2}) near $\widetilde{\xi}_0$. As in the previous case, this implies that $\widehat{\mathbf{\Psi}}$ is quasisymmetric near points of $\textrm{crit}^\ast(\widetilde{g})$.

It now follows that $\widehat{\mathbf{\Psi}}:\mathbb{T}\to\mathbb{T}$ is quasisymmetric (in fact, we have showed that $\widehat{\mathbf{\Psi}}:\mathbb{T}\to\mathbb{T}$ is $C^1$). By the Ahlfors-Beurling extension theorem, we can extend $\widehat{\mathbf{\Psi}}:\mathbb{T}\to\mathbb{T}$ to a quasiconformal map $\widehat{\mathbf{\Psi}}:\widehat{\C}\setminus\overline{\D}\to\widehat{\C}\setminus\overline{\D}$. By construction, $g\circ\widehat{\mathbf{\Psi}}\circ\widetilde{g}^{-1}: \widetilde{\Omega}\to\Omega$ and $\mathbf{\Psi}:\widetilde{T}\to T$ match continuously on $\partial\widetilde{T}$ to produce a homeomorphism $\check{\mathbf{\Psi}}$ of $\widehat{\C}$ that is quasiconformal on $\widehat{\C}\setminus\partial\widetilde{T}$. Since analytic arcs and finitely many points are removable, it follows that $\check{\mathbf{\Psi}}$ is a global quasiconformal extension of $\mathbf{\Psi}:\widetilde{T}\to T$ .
\end{proof}

We denote the union of the tiles of $ \widetilde{\sigma}$ (respectively, $\sigma$) of rank up to $n\geq1$ by $\widetilde{E}^n$ (respectively, $E^n$).

\begin{proof}[Proof of Theorem~\ref{rigidity_extremal_qd_thm}]
For $\sigma, \widetilde{\sigma}$, the images of the radial line at angle $\theta\in\R/\Z$ in $\widehat{\C}\setminus\overline{\D}$ under the B{\"o}ttcher coordinates $\tau, \widetilde{\tau}$ (with chosen normalizations) are called the \emph{external rays} of $\sigma, \widetilde{\sigma}$ at angle $\theta\in\R/\Z$. They are denoted by $R_\theta(\sigma)$ and $R_\theta(\widetilde{\sigma})$, respectively.

Using arguments from polynomial dynamics, one can show that the $d+1$ cusps on $\partial\Omega$ and $\partial\widetilde{\Omega}$ are the landing points of the $d+1$ fixed external rays of $\sigma$ and $\widetilde{\sigma}$, respectively (for a detailed study of landing patterns of the external rays of Schwarz reflection maps arising from $\Sigma_d^*$, see \cite[\S 4]{LMM2}). Let $\widetilde{\zeta}_0$ be the cusp point on $\partial\widetilde{\Omega}$ that is the landing point of the fixed ray $R_0(\widetilde{\sigma})$. We set $\zeta_0:=\pmb{\Psi}(\widetilde{\zeta}_0)$, and normalize the B{\"o}ttcher coordinate $\tau$ for $\sigma$ such that $R_0(\sigma)$ lands at $\zeta_0$.
 
Let us now define a $K$-qc map $\mathbf{\Psi}_0$ of the sphere that agrees with $\mathbf{\Psi}:\widetilde{T}\to T$ on $\widetilde{T}$ (constructed in Lemma~\ref{angled_tree_conformal_map}), and with $\tau\circ\widetilde{\tau}^{-1}$ (which conjugates $\widetilde{\sigma}$ to $\sigma$) on $U\cup R_0(\widetilde{\sigma})$, where $U$ is a $\widetilde{\sigma}$-invariant neighborhood of $\infty$. The existence of such a map is guaranteed by Lemma~\ref{global_qc}.

Note that $ \widetilde{\sigma}$ (respectively, $\sigma$) fixes $\partial\widetilde{T}$ (respectively, $\partial T$) point-wise.  Using the $(d+1):1$ covering maps $ \widetilde{\sigma}: \widetilde{\sigma}^{-1}(\widetilde{T}^0)\to\widetilde{T}^0$ and $\sigma:\sigma^{-1}(T^0)\to T^0$, we can now lift $\mathbf{\Psi}_0:\widetilde{T}^0\to T^0$ to a homeomorphism $\mathbf{\Psi}_1: \widetilde{\sigma}^{-1}(\widetilde{T}^0)\to\sigma^{-1}(T^0)$ that is conformal on the interior. Moreover, we can choose the lift so that $\mathbf{\Psi}_1: \widetilde{\sigma}^{-1}(\widetilde{T}^0)\to\sigma^{-1}(T^0)$ and $\mathbf{\Psi}_0:\widetilde{T}^0\to T^0$ match on $\partial\widetilde{T}^0$ to produce a homeomorphism $\mathbf{\Psi}_1:\overline{\widetilde{E}^1}\to\overline{E^1}$ (that is conformal on the interior). In fact, it conjugates $ \widetilde{\sigma}:\partial \widetilde{E}^1\to\partial \widetilde{T}^0$ to $\sigma:\partial E^1\to\partial T^0$ (both of which are degree $d$ coverings). 

Again, $\widetilde{\sigma}:\widehat{\C}\setminus\interior{\widetilde{E}^1}\to\widehat{\C}\setminus\interior{\widetilde{T}^0}$ and $\sigma:\widehat{\C}\setminus\interior{E^1}\to\widehat{\C}\setminus\interior{T^0}$ are degree $d$ branched coverings branched only at $\infty$. Since $\mathbf{\Psi}_0$ fixes $\infty$, we can lift $\mathbf{\Psi}_0$ via $ \widetilde{\sigma}$ and $\sigma$ to obtain a $K$-qc homeomorphism from $\widehat{\C}\setminus \interior{\widetilde{E}^1}$ to $\widehat{\C}\setminus \interior{E^1}$ (note that we get the same constant $K$ as we lift by two anti-holomorphic maps $ \widetilde{\sigma}$ and $\sigma$). The lift becomes unique once we require that it maps $\widetilde{\zeta}_0$ to $\zeta_0$. With this choice, the lift matches continuously with $\mathbf{\Psi}_1\vert_{\partial\widetilde{E}^1}$ (as constructed in the previous paragraph). By quasiconformal removability of analytic arcs and finitely many points, we obtain a $K$-qc map $\mathbf{\Psi}_1$ of the sphere that agrees with $\pmb{\Psi}_0$ on $\widetilde{T}$ (since $\partial\widetilde{E}^1$ has zero area, removing it does not affect the dilatation of $\mathbf{\Psi}_1$). Also note that the normalization $\pmb{\Psi}_1(\widetilde{\zeta}_0)=\zeta_0$ implies that $\pmb{\Psi}_1$ maps the ray $R_0(\widetilde{\sigma})$ to the ray $R_0(\sigma)$. As $\pmb{\Psi}_0$ conjugates $\widetilde{\sigma}$ to $\sigma$ on $U$, and $\pmb{\Psi}_1$ is a lift of $\pmb{\Psi}_0$ via $\widetilde{\sigma}$ and $\sigma$, it now follows that $\pmb{\Psi}_1=\pmb{\Psi}_0$ on $U$.

By iterating this lifting procedure (using the degree $d$ coverings  $ \widetilde{\sigma}:\widehat{\C}\setminus\interior{\widetilde{E}^1}\to\widehat{\C}\setminus\interior{\widetilde{T}^{0}}$ and $\sigma:\widehat{\C}\setminus\interior{E^1}\to\widehat{\C}\setminus\interior{T^{0}}$), we obtain a sequence of $K$-quasiconformal maps $\{\mathbf{\Psi}_n\}_n$ of the sphere so that 
\begin{equation}
\sigma\circ\mathbf{\Psi}_n=\mathbf{\Psi}_{n-1}\circ \widetilde{\sigma},\ \textrm{on}\ \widehat{\C}\setminus\interior{\widetilde{T}},
\label{lift_relation}
\end{equation}
and
\begin{equation}
\mathbf{\Psi}_n\equiv\mathbf{\Psi}_{n-1},\ \textrm{on}\ \widetilde{E}^{n-1}\cup \widetilde{\sigma}^{-(n-1)}(U).
\label{matching_lift}
\end{equation}

Due to compactness of the family of $K$-quasiconformal homeomorphisms, there exists a $K$-quasiconformal homeomorphism $\mathbf{\Psi}_\infty$ (of $\widehat{\C}$) that is a subsequential limit of the sequence $\{\mathbf{\Psi}_n\}_n$. By Equations~\eqref{lift_relation} and~\eqref{matching_lift}, we have that $$\sigma\circ\mathbf{\Psi}_\infty=\mathbf{\Psi}_\infty\circ \widetilde{\sigma},\ \textrm{on}\ \mathcal{B}_\infty( \widetilde{\sigma})\cup\left(\widetilde{T}^\infty\setminus\interior{\widetilde{T}}\right).$$ By continuity, the above identity continues to hold true on $\widehat{\C}\setminus\interior{\widetilde{T}}$.

By construction, we have that $\mathbf{\Psi}_\infty(\widetilde{T})=T$, and hence $\mathbf{\Psi}_\infty( \widetilde{\Omega})=\Omega$. Moreover, Equations~\eqref{lift_relation} and~\eqref{matching_lift} imply that the quasiconformal map $\mathbf{\Psi}_\infty$ is conformal on $\widetilde{T}^\infty\cup\mathcal{B}_\infty( \widetilde{\sigma})$. Since
$$
\widehat{\C}= \widetilde{T}^\infty \sqcup \partial\widetilde{T}^\infty\sqcup\mathcal{B}_\infty(\widetilde{\sigma})
$$
(by Proposition~\ref{dynamical_partition_schwarz}), and $\mathcal{L}( \widetilde{\sigma})=\partial \widetilde{T}^\infty$ has measure zero (by Proposition~\ref{limit_schwarz_zero_area}), it follows from \cite[\S II.B, Corollary~2]{A3} that $\mathbf{\Psi}_\infty$ is conformal on $\widehat{\C}$. Since $\mathbf{\Psi}_\infty$ fixes $\infty$, it follows that $A:=\mathbf{\Psi}_\infty$ is our desired affine map with $A( \widetilde{\Omega})= \Omega$.
\end{proof}

\begin{thm}\label{bijection_thm}
There is a bijection between affine equivalence classes of extremal unbounded quadrature domains of order $d$ and isomorphism classes of bi-angled trees with $d-1$ vertices.
\end{thm}
\begin{proof}
This follows from Theorems~\ref{mainthm_qd_terminology} and~\ref{rigidity_extremal_qd_thm}.
\end{proof}

\section{Crofoot--Sarason Polynomials}\label{crofoot_sarason_sec}

In this section, we will study a special class of polynomials, known as \emph{Crofoot--Sarason polynomials} (see Subsection~\ref{prelim_5} for the definition of CS polynomials), and prove a classification theorem for CS polynomials in terms of bi-angled trees (see Theorem~\ref{theorem_A}). CS anti-polynomials (which are complex conjugates of CS polynomials) will play a central role in the mating description for extremal Schwarz reflection maps, which will be studied in a sequel to this work \cite{LMM2}. Since the classification result for CS polynomials (equivalently, CS anti-polynomials) is important in its own right, we write this section in a self-contained manner.

Throughout this section, we denote a CS polynomial by $q$, and the corresponding CS anti-polynomial by $p$; i.e., $p(z):=\overline{q(z)}$. The \emph{filled Julia set} of $p$ is denoted by $K(p)$; this is the set of all points in $\C$ that have bounded forward orbits under iteration of $p$. Its boundary $\partial K(p)$ is called the \emph{Julia set} of $p$. By Subsection~\ref{prelim_5.5}, $K(p)$ is a full continuum with locally connected boundary.

The following two propositions describe the structure of ``touching'' of the immediate basins of attractions of the $d-1$ finite critical fixed points of $p$ (see Figure~\ref{fig:cubic_crit_fixed}), from which the bi-angled tree structure of the Hubbard tree of $p$ follows. It will be evident from the proofs that the bi-angled tree structure of the Hubbard tree of $p$ is a special feature of the anti-holomorphic setting.

\begin{prop}\label{sarason_bi-angled} Let $q(z)$ be a CS polynomial, and denote by $U_1, \cdots, U_{d-1}$ the $d-1$ immediate basins of attraction of the $d-1$ finite critical fixed points of $p(z)=\overline{q(z)}$. Let $i,j\in\{1,\cdots, d-1\}$ with $i\neq j$. Then either \[ \overline{U_i} \cap \overline{U_j} = \emptyset\textrm{, or  } \overline{U_i} \cap \overline{U_j} = \{ \zeta \}, \]  where $\zeta$ is a non-critical fixed point of $p$. Moreover, in the latter situation, $U_i$ and $U_j$ are the only bounded Fatou components touching at $\zeta$.
\end{prop}

\begin{proof} If \[ \overline{U_i} \cap \overline{U_j} \not= \emptyset, \] it follows from fullness of $K(p)$ that there can be at most one intersection point $\zeta$. That $\zeta$ is a fixed point of $p$ follows since $p(\zeta)$ must also then lie in the intersection of $\overline{U_i}$, $\overline{U_j}$. Furthermore, as the fixed point $\zeta$ lies on $\partial U_i$, which is contained in the Julia set of $p$, it cannot be a critical point. To prove the last statement, suppose by way of contradiction there is a third component $U_k$, distinct from $U_i$, $U_j$, so that $\zeta \in \partial U_k$. Consider a small circle $C$ centered at $\zeta$, with orientation determined by ordering three points on $C$ which lie one in each of $U_i$, $U_j$, $U_k$. Since the components $U_i$, $U_j$, $U_k$ are fixed by $p$, the orientation of $C$ is preserved under mapping by $p$, but this contradicts the fact that $p$ is a local orientation-reversing diffeomorphism in a neighborhood of $\zeta$ (as $p$ is anti-holomorphic).
\end{proof}

\begin{prop}\label{connectedness_prop} Let $q(z)$ be a CS polynomial, and notation as in Proposition \ref{sarason_bi-angled}. Then \[ \bigcup_{i=1}^{d-1}\overline{U_i} \textrm{ is connected. }\]
\end{prop}

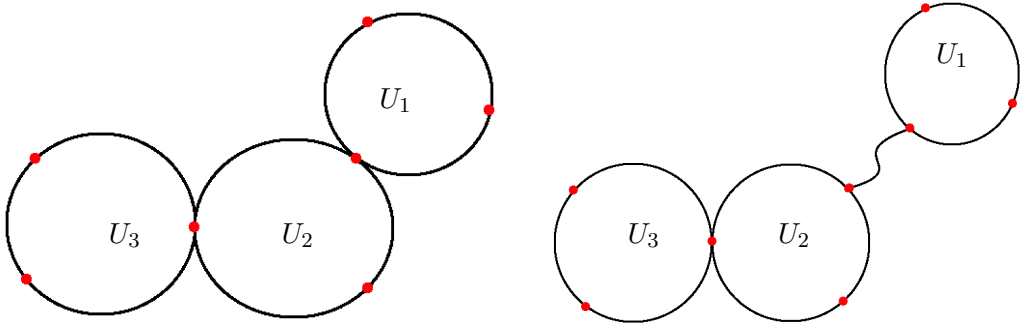
\begin{figure}[ht!]
\begin{tikzpicture}
\node[anchor=south west,inner sep=0] at (-3.2,0) {\includegraphics[width=0.48\textwidth]{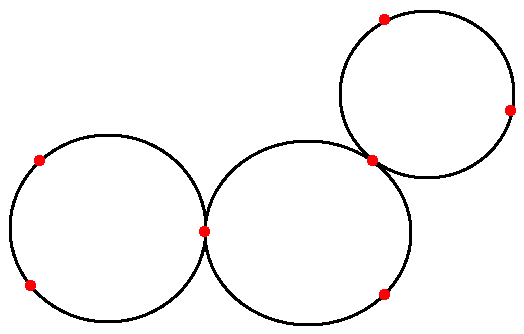}};
\node[anchor=south west,inner sep=0] at (4,0) {\includegraphics[width=0.46\textwidth]{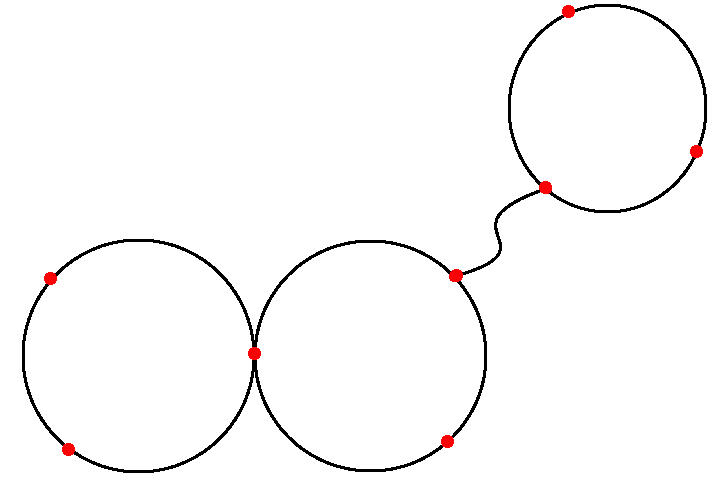}};
\node at (-1.5,1.2) {\begin{large}$U_3$\end{large}};
\node at (0.8,1.2) {\begin{large}$U_2$\end{large}};
\node at (2.1,3) {\begin{large}$U_1$\end{large}};
\node at (5.4,1.2) {\begin{large}$U_3$\end{large}};
\node at (7.4,1.2) {\begin{large}$U_2$\end{large}};
\node at (9.5,3.6) {\begin{large}$U_1$\end{large}};
\end{tikzpicture}
\caption{Left: A valid configuration of invariant (bounded) immediate basins for a CS anti-polynomial of degree $4$ (the non-critical fixed points are marked in red). Right: No CS anti-polynomial of degree $4$ can admit such a configuration of invariant (bounded) immediate basins; indeed, this configuration produces $11$ fixed points of the anti-polynomial ($8$ non-critical and $3$ critical), but the maximum possible number of fixed points is $3\cdot 4-2=10$.}
\label{fatou_connection}
\end{figure}

\begin{proof} We work by way of contradiction: suppose that $\cup_{i=1}^{d-1}\overline{U_i}$ is not connected. Then, perhaps after renumbering, there is $k$ with $1\leq k < d-1$ such that $\cup_{i=1}^k \overline{U_i}$ is connected and disjoint from $\cup_{i=k+1}^{d-1} \overline{U_i}$. We will show that there must then be more than $3d-2$ fixed points of $p$ in $\C$, contradicting Theorem~\ref{KhSw} (see Figure~\ref{fatou_connection}).  

For each $i\in\{1, \cdots, d-1\}$, there is a conformal map \[\phi_i: U_i\rightarrow \mathbb{D} \textrm{ such that } \phi\circ p\circ\phi^{-1}(z)=\overline{z}^2. \] The conjugacy $\phi_i: U_i\rightarrow \mathbb{D}$ extends homeomorphically to $\phi_i: \partial U_i\rightarrow\partial\mathbb{D}$ since $\partial U_i$ is locally connected (see \cite[Lemma~19.3]{MR2193309}) and $\overline{U_i}$ is full. Thus it follows that the map \[p: \overline{U_i}\rightarrow\overline{U_i}\] has three fixed points on $\partial U_i$ (and one interior critical fixed point). 

Let $\zeta$ be a non-critical fixed point of $p$. By Proposition \ref{sarason_bi-angled}, $\zeta$ can lie on the boundary of at most two immediate basins of attraction (of the finite critical points). Furthermore, we claim that there must be at least one non-critical fixed point of $p$ which is on the boundary of precisely one basin of attraction $U_i$ for $1\leq i \leq k$. Indeed, if this was not the case, then the set \[\widehat{\mathbb{C}}\setminus\left(\cup_{i=1}^k \overline{U_i}\right) \] would be disconnected, and so $\widehat{\mathbb{C}}\setminus K(p)$ would be disconnected, and this contradicts the fact that $K(p)$ is full (see Subsection~\ref{prelim_5.5}). Thus a count yields that $p$ has at least  \[  \underbrace{(k)}_{\textrm{ critical fixed points in }U_i\textrm{, } 1\leq i \leq k} + \underbrace{3k}_{ \textrm{ fixed points in } \partial U_i\textrm{, } 1\leq i \leq k } - \underbrace{(k-1)}_{\textrm{shared fixed points}} = 3k+1  \] fixed points in $\cup_{i=1}^{k} \overline{U_i}$. 

A similar argument in each of the other components of $\cup_{i=1}^{d-1}\overline{U_i}$ yields that $p$ has at least $3d-2-3k$ fixed points in $\cup_{i=k+1}^{d-1} \overline{U_i}$. As $\cup_{i=1}^{k} \overline{U_i}$ and $\cup_{i=k+1}^{d-1} \overline{U_i}$ are disjoint, this means $p$ has at least $3d-1>3d-2$ fixed points in $\mathbb{C}$, which is our needed contradiction.

\end{proof}

It is known that if $q(z)$ is a CS polynomial of degree $d$, then $p$ has $(3d-2)$ fixed points in the plane; i.e., it realizes the upper bound of \cite{KhSw} (see Remark \ref{lefschetz_count_rem}). Our analysis yields a dynamical description of these fixed points.

\begin{prop}\label{sharp_bound}
Let $q(z)$ be a CS polynomial of degree $d$. Then $p$ has exactly $3d-2$ fixed points in $\C$; $d-1$ of which are super-attracting, and the remaining $2d-1$ are repelling. 
\end{prop}
\begin{proof}
By definition, $p$ has $d-1$ distinct fixed critical points in $\C$. Moreover, each $U_i$ has exactly three distinct fixed points on its boundary. This gives at most $3(d-1)$ distinct fixed points of $p$ on the boundary of $\cup_{i=1}^{d-1} \overline{U_i}$. Since $K(p)$ is full and at most two distinct $U_j$ can touch at a fixed point, it follows from the proof of Proposition~\ref{connectedness_prop} that exactly $d-2$ of these boundary fixed points lie on the common boundary of two distinct $U_j$. We conclude that there are exactly $3(d-1)-(d-2)=(2d-1)$ distinct fixed points on the boundary of $\cup_{i=1}^{d-1} \overline{U_i}$. But this already accounts for $(2d-1)+(d-1)=(3d-2)$ distinct fixed points of $p$ in $\C$. By Theorem~\ref{KhSw}, these are all the fixed points of $p$.

The non-critical fixed points of $p$ lie on $\partial U_i$ (for some $i\in\{1,\cdots, d-1\}$), and hence cannot be (super-)attracting. Since each critical point of $p$ is fixed, the non-critical fixed points cannot be indifferent either \cite[Theorem~10.15, 11.17]{MR2193309}. Hence, all of the $2d-1$ non-critical fixed points of $p$ are repelling.
\end{proof}

\begin{figure}[ht!]
\begin{tikzpicture}
\node[anchor=south west,inner sep=0] at (0,0) {\includegraphics[width=0.7\textwidth]{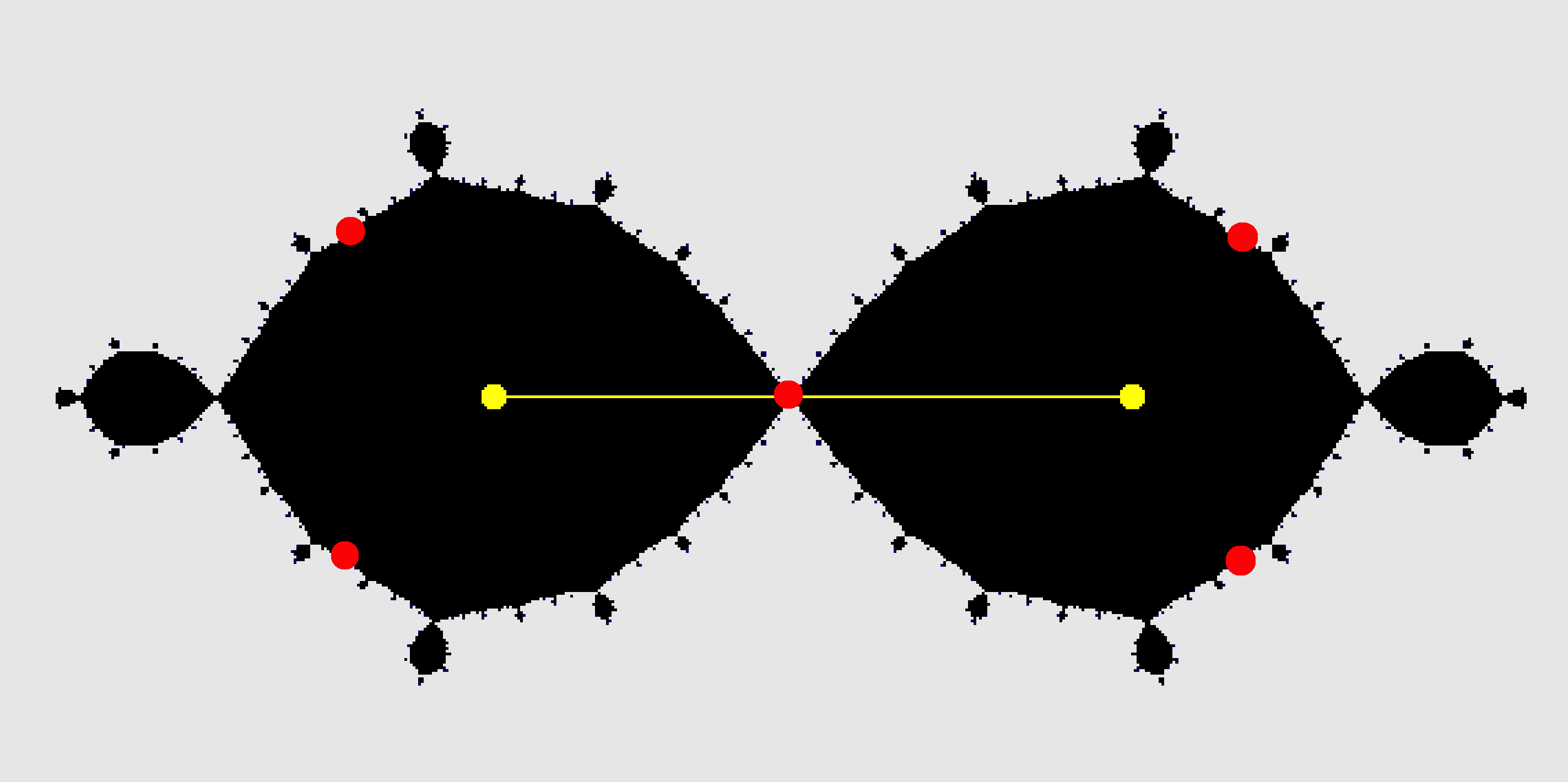}};
\end{tikzpicture}
\caption{The dynamical plane of a critically fixed degree $3$ anti-polynomial whose angled Hubbard tree (shown in yellow) is isomorphic to the unique bi-angled tree with $2$ vertices. The five repelling fixed points of the anti-polynomial are also marked (in red)}
\label{fig:cubic_crit_fixed}
\end{figure}

We will now deduce from the previous results that the angled Hubbard tree of a CS anti-polynomial (defined in Subsection~\ref{prelim_5.5}) has the structure of a bi-angled tree.

\begin{prop}\label{cs_angled_hubbard}
The angled Hubbard tree $\mathcal{T}(p)$ of a CS anti-polynomial $p$ of degree $d$ is a bi-angled tree with $d-1$ vertices.
\end{prop}
\begin{proof}
We will continue to use the notations of Proposition~\ref{connectedness_prop}. 

Let us first observe that the three radial lines $\{re^{2\pi i\theta}:\theta\in\{0,\frac13,\frac23\}, r\in[0,1)\}\subset\overline{\D}$ are set-wise fixed by $\overline{z}^2$. For each immediate basin of attraction $U_i$ (of a finite super-attracting fixed point of $p$), the pre-images of these radial lines under the B{\"o}ttcher coordinate $\phi_i$ are fixed (as curves) by $p$. These are called the fixed \emph{internal rays} of $\overline{U_i}$. Note that since each $U_i$ is a Jordan domain with exactly three fixed points on its boundary, the three fixed internal rays of $U_i$ land at distinct repelling fixed point on $\partial U_i$. Propositions~\ref{connectedness_prop} and~\ref{sarason_bi-angled} now imply that the Hubbard tree $\mathcal{T}(p)$ of $p$ is contained in the union of the fixed internal rays of $U_1,\cdots, U_{d-1}$ along with their landing points. Moreover, $\mathcal{T}(p)$ intersects the Julia set of $p$ precisely at the $d-2$ fixed points that are cut-points of $\cup_{i=1}^{d-1}\overline{U_i}$. It also follows from Propositions~\ref{connectedness_prop} and~\ref{sarason_bi-angled} that every branch point of $\mathcal{T}(p)$ is a (finite) fixed critical point of $p$, and there are exactly three edges meeting at such a branch point. Hence, $\mathcal{T}(p)$ has precisely $d-1$ vertices that are fixed critical points of $p$ in $\C$, and each vertex has degree at most $3$. 

As the fixed internal rays forming the Hubbard tree have angles $0, \frac{2\pi}{3}$, or $\frac{4\pi}{3}$ (in B{\"o}ttcher coordinates), it follows that the counter-clockwise oriented angle between two fixed internal rays meeting at a fixed critical point is $\frac{2\pi}{3}$ or $\frac{4\pi}{3}$. It now follows from the definition of bi-angled trees that $\mathcal{T}(p)$ has a natural bi-angled tree structure.
\end{proof}

To realize each bi-angled tree as the angled Hubbard tree of some CS anti-polynomial, we will endow a bi-angled tree with an orientation-reversing angled tree map (see Subsection~\ref{prelim_5.5}), and invoke the realization theorem \cite[Theorem~5.1]{Poi3}.

\begin{prop}\label{bi_angled_cs_hubbard}
Each bi-angled tree $\mathcal{T}$ with $d-1$ vertices is isomorphic to the angled Hubbard tree of a CS anti-polynomial of degree $d$.
\end{prop}
\begin{proof}
Let us fix a bi-angled tree $\mathcal{T}$ with $d-1$ vertices. We define a \emph{local degree function} $\textswab{d}:V(\mathcal{T})\to\N$ by setting $\textswab{d}(v)=2$, for all $v\in V(\mathcal{T})$. The relations $$-2(\frac{2\pi}{3})=\frac{2\pi}{3},\ \textrm{and}\ -2(\frac{4\pi}{3})=\frac{4\pi}{3},\ (\textrm{mod}\ 2\pi)$$ imply that $\tau\equiv\mathrm{id}$ is an orientation-reversing angled tree map of degree $d$ on $\mathcal{T}$. Moreover, since all vertices of $\mathcal{T}$ are critical and fixed under $\tau$, it follows that all vertices are of \emph{Fatou type}. Hence, the condition of being \emph{expanding} (in the sense of \cite[\S 1]{Poi3}) is vacuously satisfied by the orientation-reversing angled tree map $\tau:\mathcal{T}\to\mathcal{T}$. The realization theorem \cite[Theorem~5.1]{Poi3} now provides us with a postcritically finite anti-polynomial $p$ of degree $d$ such that the angled Hubbard tree of $p$ is isomorphic to $\mathcal{T}$, and $p$ has $d-1$ distinct, fixed critical points in $\C$ (as the angled tree map $\tau$ has the same). Therefore, $p$ is our desired CS anti-polynomial of degree $d$ whose angled Hubbard tree is isomorphic to $\mathcal{T}$.
\end{proof}

We now prove the following Theorem:

\begin{thm}\label{crofoot_sarason_bi-angled_bijection_thm}
There exists a bijection between equivalence classes of CS polynomials of degree $d$ (respectively, affine conjugacy classes of CS anti-polynomials of degree $d$) and isomorphism classes of bi-angled trees with $d-1$ vertices. 
\end{thm}

\begin{proof}
By Proposition~\ref{cs_angled_hubbard}, for any CS polynomial $q$ of degree $d$, the angled Hubbard tree of the corresponding CS anti-polynomial $p=\overline{q}$ is a bi-angled tree with $d-1$ vertices. Thanks to Proposition~\ref{equiv_conj_sara_cro_prop} and the fact that the angled Hubbard trees of affinely conjugate anti-polynomials are isomorphic, the map $$q\mapsto \mathcal{T}(p)$$ induces a map from the set of equivalence classes of CS polynomials of degree $d$ to the set of isomorphism classes of bi-angled trees with $d-1$ vertices (see Definition~\ref{isom_def}).  

Surjectivity of the above map is the content of Proposition~\ref{bi_angled_cs_hubbard}.

Finally, suppose that there exist two CS polynomials $q_1$ and $q_2$ (of degree $d$) such that the angled Hubbard trees of $p_1:=\overline{q}_1$ and $p_2:=\overline{q}_2$ are isomorphic. By Propositions~\ref{sarason_bi-angled} and~\ref{connectedness_prop}, for $i\in\{1,2\}$, each vertex of the Hubbard tree of $p_i$ is a fixed critical point, and each edge of the tree is the closure of the union of two invariant internal rays. Hence, $p_i$ fixes each vertex of its Hubbard tree, and maps each edge homeomorphically to itself. It follows that the dynamics induced by $p_1$ and $p_2$ on their respective Hubbard trees are isotopic to the identity (relative to the vertices of the respective trees). The uniqueness part of \cite[Theorem~5.1]{Poi3} now implies that $p_1$ and $p_2$ are affinely conjugate. Hence, by Proposition~\ref{equiv_conj_sara_cro_prop}, the CS polynomials $q_1$ and $q_2$ are equivalent in the sense of Definition~\ref{sara_cro_equiv_def}. This completes the proof of injectivity of the map between equivalence classes of CS polynomials (of degree $d$) and isomorphism classes of bi-angled trees (with $d-1$ vertices).

Similarly, the bijection between affine conjugacy classes of CS anti-polynomials (of degree $d$) and isomorphism classes of bi-angled trees (with $d-1$ vertices) is obtained by sending a CS anti-polynomial $p$ to its angled Hubbard tree $\mathcal{T}(p)$.
\end{proof}

\begin{rem}\label{hele_shaw_surgery_rem}
A quite different proof of the existence  Proposition~\ref{bi_angled_cs_hubbard} can be given in terms of Hele-Shaw flow and quasiconformal surgery. Indeed, given a bi-angled tree $\mathcal{T}$, we have proved the existence of an extremal unbounded quadrature domain $\Omega$ of order $d$ with its associated bi-angled tree isomorphic to $\mathcal{T}$. Following \cite[\S 5.3]{MR3454377}, we now run Hele-Shaw flow for a sufficiently small time on the corresponding droplet which produces a new unbounded quadrature domain $\check{\Omega}$ (also of order $d$) with non-singular boundary and of connectivity $d-1$. In particular, the new droplet has $d-1$ connected components each of which is a topological disk with non-singular boundary. Denoting the Schwarz reflection map of $\check{\Omega}$ by $\check{\sigma}$, one now observes that $\check{\sigma}:\check{\sigma}^{-1}(\check{\Omega})\to\check{\Omega}$ is an \emph{anti-rational-like map} of degree $d$ (in the sense of \cite[Definition~4]{Buff}) with a unique pole at $\infty$. Arguing as in the proof of \cite[Lemma~4.3]{MR3454377} (see \cite[Theorem~4]{Buff} for a similar surgery procedure in the holomorphic case), one now concludes that $\check{\sigma}:\check{\sigma}^{-1}(\check{\Omega})\to\check{\Omega}$ extends to an anti-quasiregular map (of $\widehat{\C}$) of degree $d$ with a unique pole at $\infty$ and a simple fixed critical point in each component of the droplet. Moreover, this anti-quasiregular map is quasiconformally conjugate to an anti-rational map of degree $d$. Clearly, this straightened map is an anti-polynomial of degree $d$ with $d-1$ distinct fixed critical points in the plane. Using Propositions~\ref{connectedness_prop} and~\ref{sarason_bi-angled}, one may argue that the angled Hubbard tree of this critically fixed anti-polynomial is isomorphic to $\mathcal{T}$.

Yet another way of establishing a direct link between (\ref{Sigma_d^*}) and (\ref{crofoot-sarason}) is as follows. For $f\in\Sigma_d^*$ with $d-2$ double points on $f(\mathbb{T})$, the restriction of the associated Schwarz reflection map $\sigma$ to the closure of each of the $d-1$ invariant components of $T^\infty(\sigma)$ is conformally conjugate to the reflection map $\rho$ (see \cite[\S 3.1]{LLMM1}). Using the conjugacy $\mathcal{E}$ between $\rho\vert_\mathbb{T}$ and $\overline{z}^2\vert_\mathbb{T}$ (defined in \cite[\S 3.2]{LLMM1}), one can now perform a topological surgery on each of these $d-1$ invariant components of $T^\infty(\sigma)$ to replace the action of $\rho$ by the action of $\overline{z}^2$. This defines a critically fixed orientation-reversing branched covering of degree $d$ on the topological $2$-sphere. It is not hard to check that this branch cover is Thurston equivalent to a degree $d$ CS anti-polynomial whose Hubbard tree, viewed as a bi-angled tree, is isomorphic to $\mathcal{T}(f)$.
\end{rem}

\section{The Class $S_d^*$}\label{S_d^*_section}

As already mentioned in the Introduction, the class $\Sigma_d^*$ is related to the more classical space $S_d^*$ (also defined in the Introduction). Arguing as in the proof of Proposition~\ref{crit_points_on_circle}, one sees that for $f\in S_d^*$, $f(\mathbb{T})$ has $d-1$ cusps. Analogous to the $\Sigma_d^*$ setting, we will call $f\in S_d^*$ a \emph{Suffridge polynomial} if $f(\mathbb{T})$ has the maximal number of double points: $d-2$ according to \cite[Lemma~2.2]{2014arXiv1411.3415L}. Given a Suffridge polynomial $f\in S_d^*$, there is a natural rooted binary tree $\mathcal{T}(f)$ associated with $f$ (see Definitions \ref{rooted_binary_tree}, \ref{rooted_bi-angled_tree} and Table \ref{table_2} below). In this Section we will prove analogous existence and rigidity results (corresponding to Theorems \ref{mainthm_qd_terminology}, \ref{rigidity_extremal_qd_thm} already proven for $\Sigma_d^*$) for the class $S_d^*$.

\begin{definition}\label{rooted_binary_tree} A \emph{rooted binary tree} is a planar, rooted tree for which every vertex has a left child, a right child, neither, or both. \end{definition}

\begin{definition}\label{rooted_bi-angled_tree} For an extremal bounded quadrature domain $\Omega$ of order $d>2$, we associate a rooted binary tree $\mathcal{T}=\mathcal{T}(\Omega)=(V,E)$ as follows. Since $d>2$, there is a unique bounded fundamental tile $T_1$ which shares a common boundary point $c_1$ with the unique unbounded fundamental tile $T_0$. Assign a \emph{rooted} vertex at $1$ associated to $T_1$. $T_1$ has two remaining cusps $c_2$, $c_3$ (labelled so that $(c_1, c_2, c_3)$ is counter-clockwise). For $j=2,3$, we include in $V$ the point $1+\exp(i (\pi/3 + (2-j)2\pi/3) )$ and in $E$ the linear segment connecting $1$ to $1+\exp(i (\pi/3 + (2-j)2\pi/3) )$ if and only if the cusp point $c_j$ is a boundary point of two fundamental tiles. This defines at most two new vertices and edges in $\mathcal{T}$. We perform a similar procedure for each new vertex, allowing for edges to decrease in length in successive steps of this procedure so as to avoid self-intersection. This recursively defines the rooted binary tree $\mathcal{T}$. 
\end{definition}

\begin{table}
\begin{adjustwidth}{-.65in}{-.65in} 
\begin{center}
\begin{tabular}{ |c|c|c|c|c| } 
\hline
$d$ & Rooted Binary tree & Droplet & Suffridge Polynomial  \\
\hline

    2 & $\emptyset$ & \begin{centering} \parbox[c]{8em}{\scalebox{.12}{
\begingroup%
  \makeatletter%
  \providecommand\color[2][]{%
    \errmessage{(Inkscape) Color is used for the text in Inkscape, but the package 'color.sty' is not loaded}%
    \renewcommand\color[2][]{}%
  }%
  \providecommand\transparent[1]{%
    \errmessage{(Inkscape) Transparency is used (non-zero) for the text in Inkscape, but the package 'transparent.sty' is not loaded}%
    \renewcommand\transparent[1]{}%
  }%
  \providecommand\rotatebox[2]{#2}%
  \ifx\svgwidth\undefined%
    \setlength{\unitlength}{566bp}%
    \ifx\svgscale\undefined%
      \relax%
    \else%
      \setlength{\unitlength}{\unitlength * \real{\svgscale}}%
    \fi%
  \else%
    \setlength{\unitlength}{\svgwidth}%
  \fi%
  \global\let\svgwidth\undefined%
  \global\let\svgscale\undefined%
  \makeatother%
  \begin{picture}(1,1.09363958)%
    \put(0,0){\includegraphics[width=\unitlength,page=1]{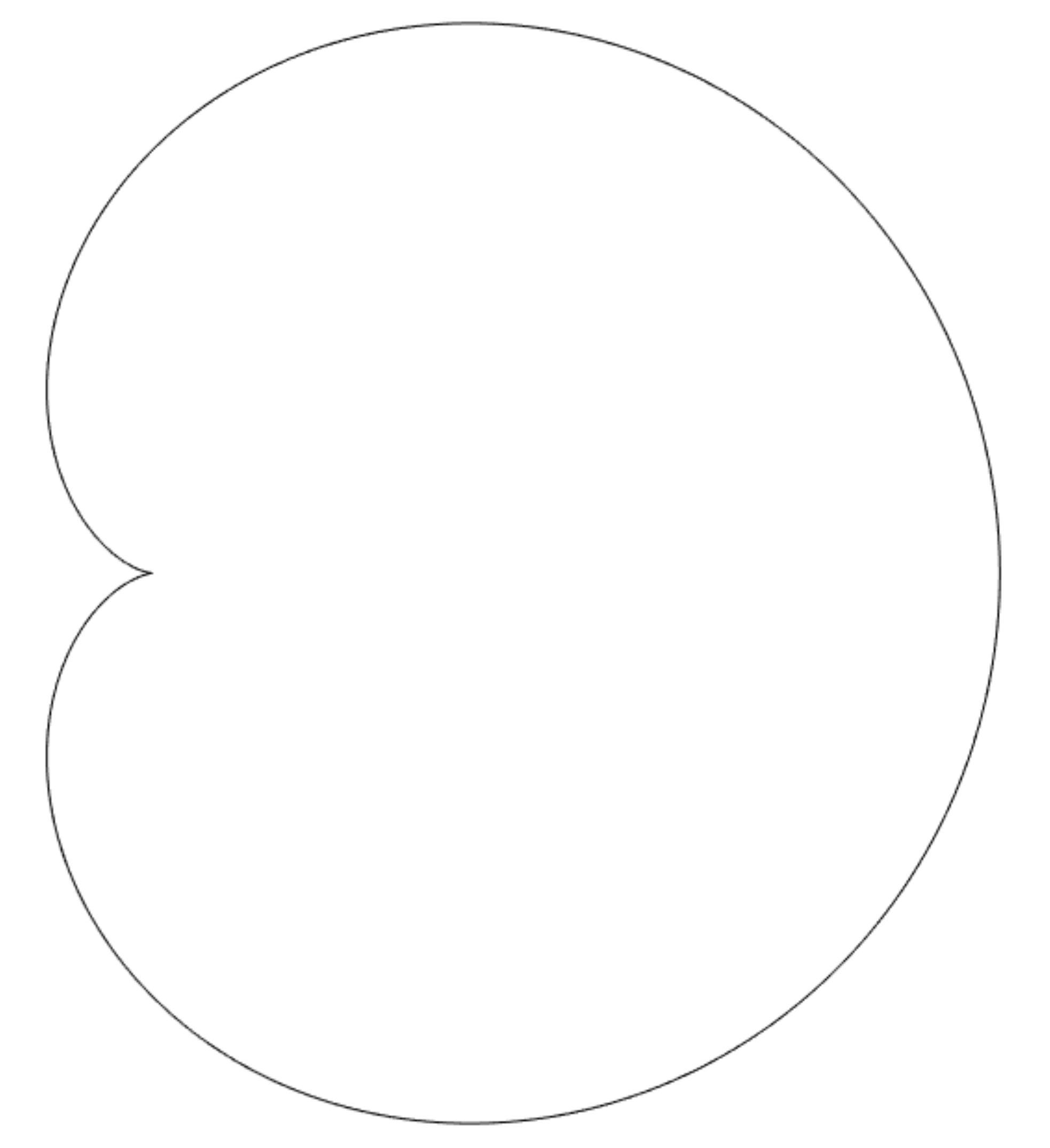}}%
  \end{picture}%
\endgroup%
}} \end{centering} &  $ z+\frac{z^2}{2} $ \\

    3 & \parbox[c]{4em}{\scalebox{.15}{
\begingroup%
  \makeatletter%
  \providecommand\color[2][]{%
    \errmessage{(Inkscape) Color is used for the text in Inkscape, but the package 'color.sty' is not loaded}%
    \renewcommand\color[2][]{}%
  }%
  \providecommand\transparent[1]{%
    \errmessage{(Inkscape) Transparency is used (non-zero) for the text in Inkscape, but the package 'transparent.sty' is not loaded}%
    \renewcommand\transparent[1]{}%
  }%
  \providecommand\rotatebox[2]{#2}%
  \ifx\svgwidth\undefined%
    \setlength{\unitlength}{286.88332625bp}%
    \ifx\svgscale\undefined%
      \relax%
    \else%
      \setlength{\unitlength}{\unitlength * \real{\svgscale}}%
    \fi%
  \else%
    \setlength{\unitlength}{\svgwidth}%
  \fi%
  \global\let\svgwidth\undefined%
  \global\let\svgscale\undefined%
  \makeatother%
  \begin{picture}(1,1.02253522)%
    \put(0,0){\includegraphics[width=\unitlength,page=1]{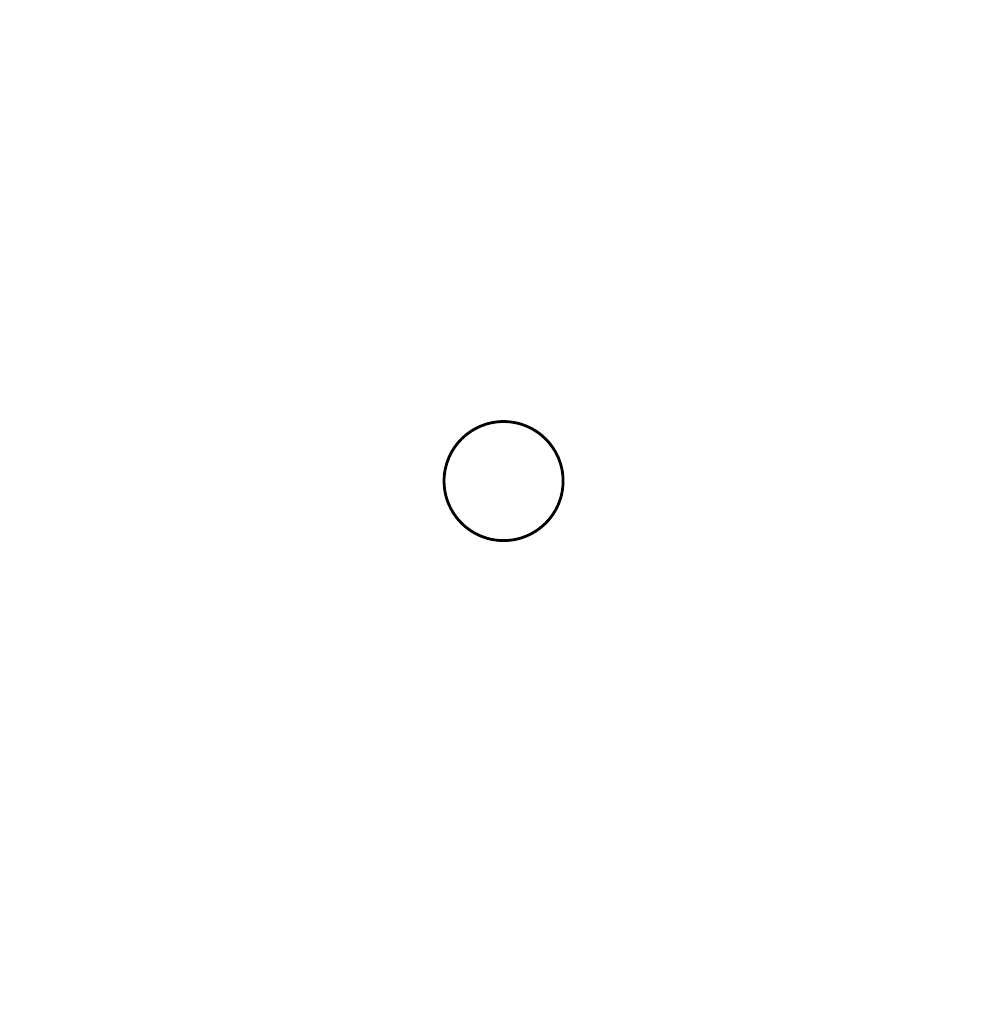}}%
  \end{picture}%
\endgroup%
}} & \parbox[c]{8em}{\scalebox{.15}{
\begingroup%
  \makeatletter%
  \providecommand\color[2][]{%
    \errmessage{(Inkscape) Color is used for the text in Inkscape, but the package 'color.sty' is not loaded}%
    \renewcommand\color[2][]{}%
  }%
  \providecommand\transparent[1]{%
    \errmessage{(Inkscape) Transparency is used (non-zero) for the text in Inkscape, but the package 'transparent.sty' is not loaded}%
    \renewcommand\transparent[1]{}%
  }%
  \providecommand\rotatebox[2]{#2}%
  \ifx\svgwidth\undefined%
    \setlength{\unitlength}{474bp}%
    \ifx\svgscale\undefined%
      \relax%
    \else%
      \setlength{\unitlength}{\unitlength * \real{\svgscale}}%
    \fi%
  \else%
    \setlength{\unitlength}{\svgwidth}%
  \fi%
  \global\let\svgwidth\undefined%
  \global\let\svgscale\undefined%
  \makeatother%
  \begin{picture}(1,1.13924051)%
    \put(0,0){\includegraphics[width=\unitlength,page=1]{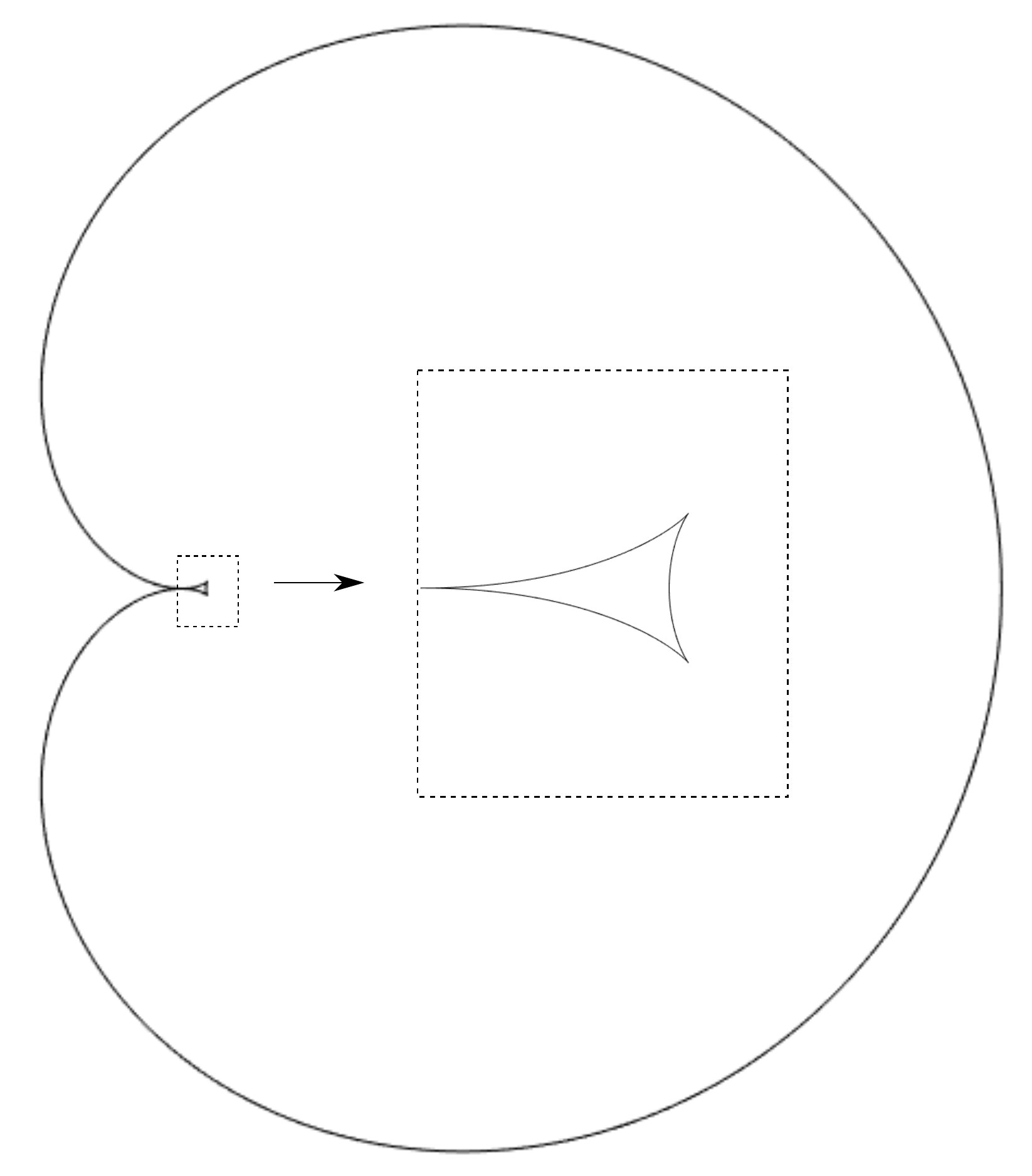}}%
  \end{picture}%
\endgroup%
}} & $z+\frac{2\sqrt{2}}{3}z^2+\frac{z^3}{3}$ \\ 

    4 & \parbox[c]{4em}{\scalebox{.15}{
\begingroup%
  \makeatletter%
  \providecommand\color[2][]{%
    \errmessage{(Inkscape) Color is used for the text in Inkscape, but the package 'color.sty' is not loaded}%
    \renewcommand\color[2][]{}%
  }%
  \providecommand\transparent[1]{%
    \errmessage{(Inkscape) Transparency is used (non-zero) for the text in Inkscape, but the package 'transparent.sty' is not loaded}%
    \renewcommand\transparent[1]{}%
  }%
  \providecommand\rotatebox[2]{#2}%
  \ifx\svgwidth\undefined%
    \setlength{\unitlength}{198.78990834bp}%
    \ifx\svgscale\undefined%
      \relax%
    \else%
      \setlength{\unitlength}{\unitlength * \real{\svgscale}}%
    \fi%
  \else%
    \setlength{\unitlength}{\svgwidth}%
  \fi%
  \global\let\svgwidth\undefined%
  \global\let\svgscale\undefined%
  \makeatother%
  \begin{picture}(1,1.10976061)%
    \put(0,0){\includegraphics[width=\unitlength,page=1]{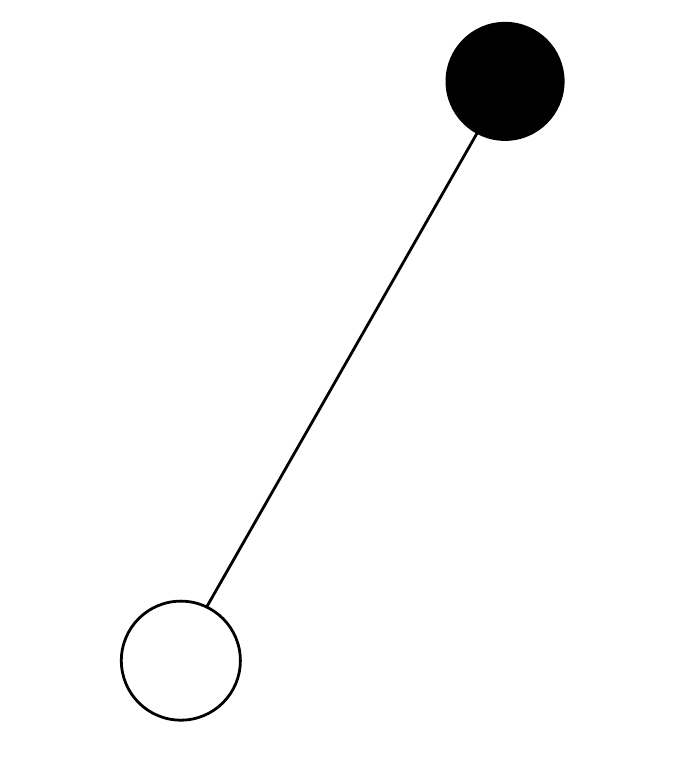}}%
  \end{picture}%
\endgroup%
}} & \parbox[c]{8em}{\scalebox{.15}{
\begingroup%
  \makeatletter%
  \providecommand\color[2][]{%
    \errmessage{(Inkscape) Color is used for the text in Inkscape, but the package 'color.sty' is not loaded}%
    \renewcommand\color[2][]{}%
  }%
  \providecommand\transparent[1]{%
    \errmessage{(Inkscape) Transparency is used (non-zero) for the text in Inkscape, but the package 'transparent.sty' is not loaded}%
    \renewcommand\transparent[1]{}%
  }%
  \providecommand\rotatebox[2]{#2}%
  \ifx\svgwidth\undefined%
    \setlength{\unitlength}{479bp}%
    \ifx\svgscale\undefined%
      \relax%
    \else%
      \setlength{\unitlength}{\unitlength * \real{\svgscale}}%
    \fi%
  \else%
    \setlength{\unitlength}{\svgwidth}%
  \fi%
  \global\let\svgwidth\undefined%
  \global\let\svgscale\undefined%
  \makeatother%
  \begin{picture}(1,1.1440501)%
    \put(0,0){\includegraphics[width=\unitlength,page=1]{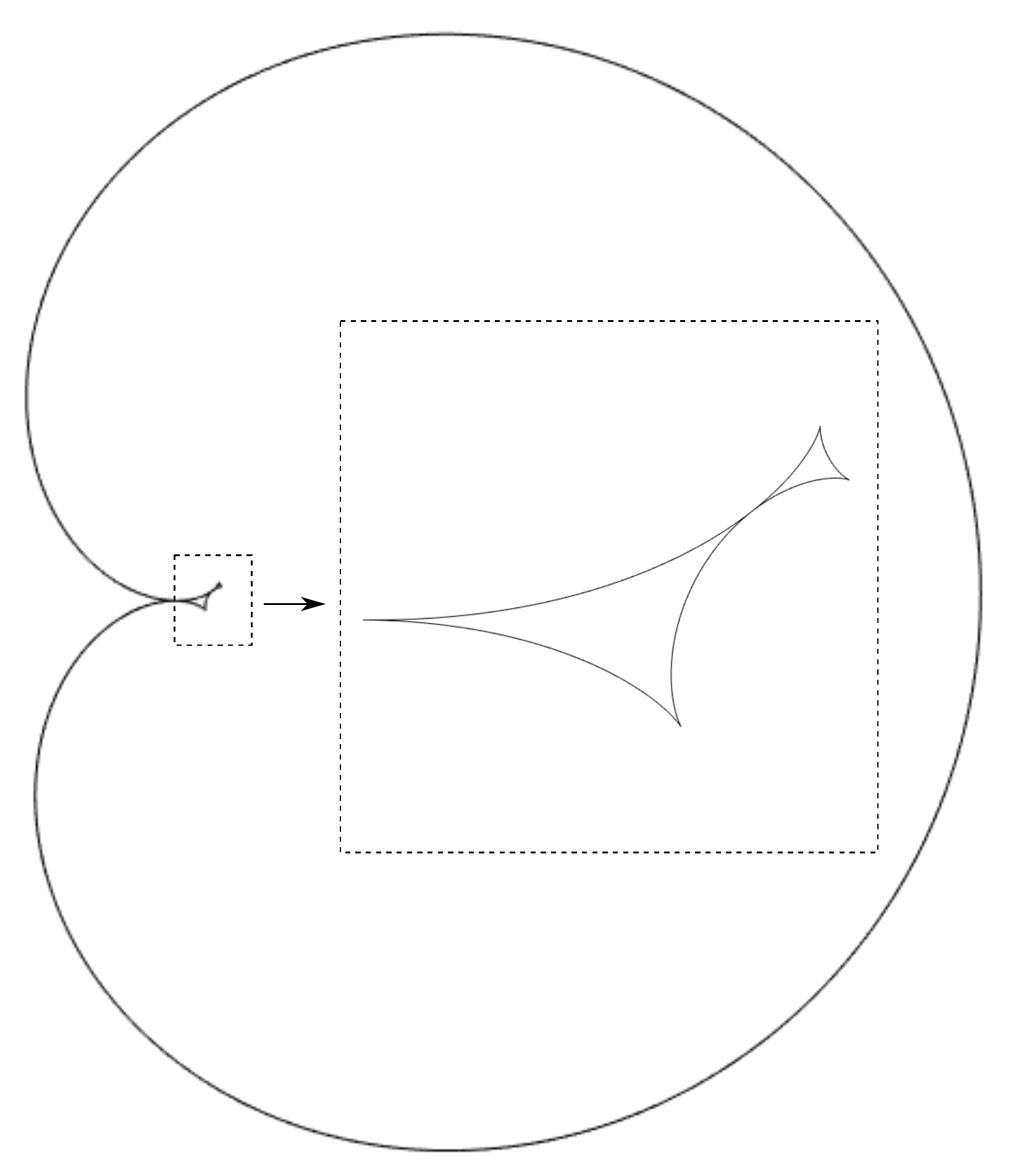}}%
  \end{picture}%
\endgroup%
}} & $ z+\frac{3}{2}Ae^{it}z^2 + Ae^{-it}z^3+\frac{z^4}{4} $ \\

    4 & \parbox[c]{4em}{\scalebox{.15}{
\begingroup%
  \makeatletter%
  \providecommand\color[2][]{%
    \errmessage{(Inkscape) Color is used for the text in Inkscape, but the package 'color.sty' is not loaded}%
    \renewcommand\color[2][]{}%
  }%
  \providecommand\transparent[1]{%
    \errmessage{(Inkscape) Transparency is used (non-zero) for the text in Inkscape, but the package 'transparent.sty' is not loaded}%
    \renewcommand\transparent[1]{}%
  }%
  \providecommand\rotatebox[2]{#2}%
  \ifx\svgwidth\undefined%
    \setlength{\unitlength}{198.78990834bp}%
    \ifx\svgscale\undefined%
      \relax%
    \else%
      \setlength{\unitlength}{\unitlength * \real{\svgscale}}%
    \fi%
  \else%
    \setlength{\unitlength}{\svgwidth}%
  \fi%
  \global\let\svgwidth\undefined%
  \global\let\svgscale\undefined%
  \makeatother%
  \begin{picture}(1,1.10976061)%
    \put(0,0){\includegraphics[width=\unitlength,page=1]{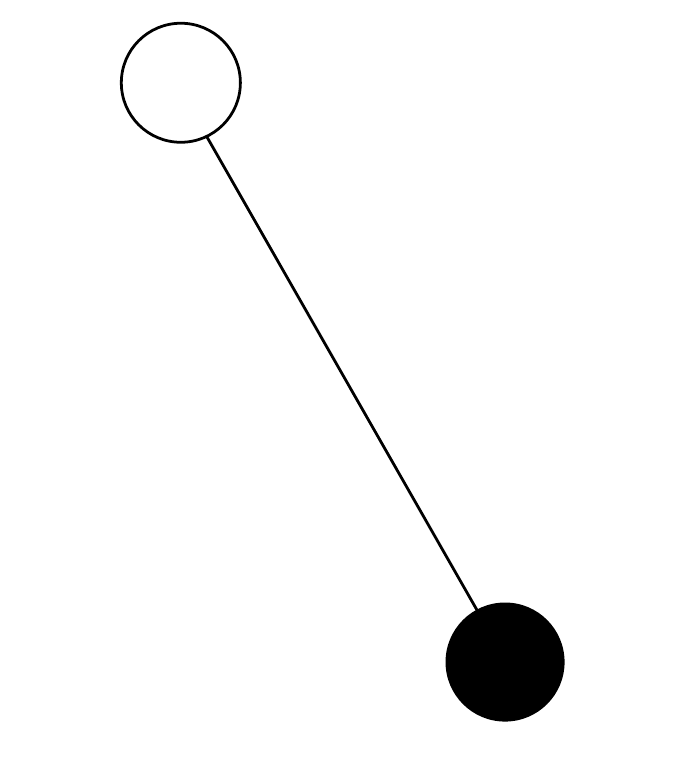}}%
  \end{picture}%
\endgroup%
}} & \parbox[c]{8em}{\scalebox{.15}{
\begingroup%
  \makeatletter%
  \providecommand\color[2][]{%
    \errmessage{(Inkscape) Color is used for the text in Inkscape, but the package 'color.sty' is not loaded}%
    \renewcommand\color[2][]{}%
  }%
  \providecommand\transparent[1]{%
    \errmessage{(Inkscape) Transparency is used (non-zero) for the text in Inkscape, but the package 'transparent.sty' is not loaded}%
    \renewcommand\transparent[1]{}%
  }%
  \providecommand\rotatebox[2]{#2}%
  \ifx\svgwidth\undefined%
    \setlength{\unitlength}{494bp}%
    \ifx\svgscale\undefined%
      \relax%
    \else%
      \setlength{\unitlength}{\unitlength * \real{\svgscale}}%
    \fi%
  \else%
    \setlength{\unitlength}{\svgwidth}%
  \fi%
  \global\let\svgwidth\undefined%
  \global\let\svgscale\undefined%
  \makeatother%
  \begin{picture}(1,1.09919028)%
    \put(0,0){\includegraphics[width=\unitlength,page=1]{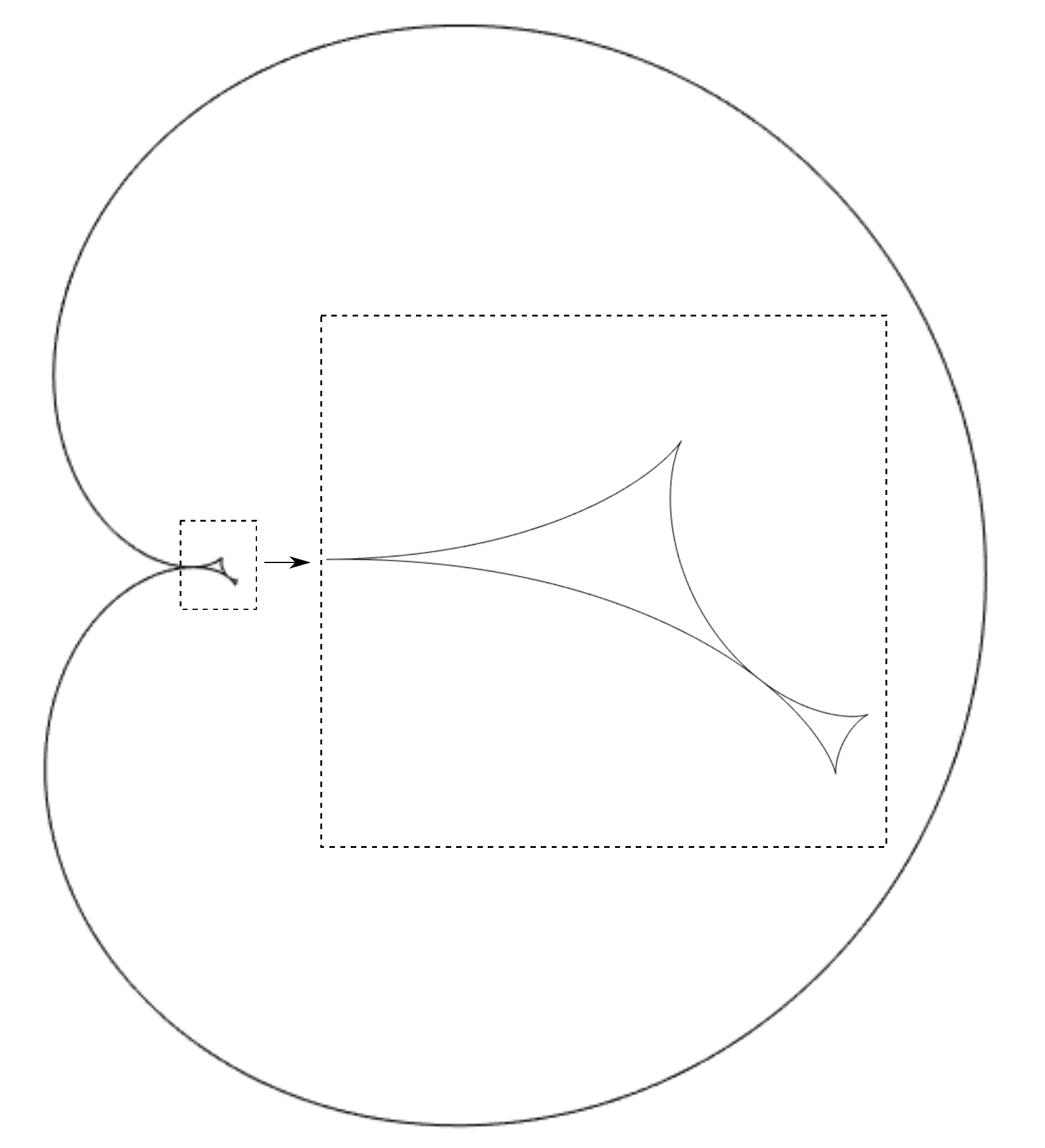}}%
  \end{picture}%
\endgroup%
}} & $ z+\frac{3}{2}Ae^{-it}z^2 + Ae^{it}z^3+\frac{z^4}{4} $ \\

    5 & \parbox[c]{4em}{\scalebox{.15}{
\begingroup%
  \makeatletter%
  \providecommand\color[2][]{%
    \errmessage{(Inkscape) Color is used for the text in Inkscape, but the package 'color.sty' is not loaded}%
    \renewcommand\color[2][]{}%
  }%
  \providecommand\transparent[1]{%
    \errmessage{(Inkscape) Transparency is used (non-zero) for the text in Inkscape, but the package 'transparent.sty' is not loaded}%
    \renewcommand\transparent[1]{}%
  }%
  \providecommand\rotatebox[2]{#2}%
  \ifx\svgwidth\undefined%
    \setlength{\unitlength}{260.20718071bp}%
    \ifx\svgscale\undefined%
      \relax%
    \else%
      \setlength{\unitlength}{\unitlength * \real{\svgscale}}%
    \fi%
  \else%
    \setlength{\unitlength}{\svgwidth}%
  \fi%
  \global\let\svgwidth\undefined%
  \global\let\svgscale\undefined%
  \makeatother%
  \begin{picture}(1,1.43169045)%
    \put(0,0){\includegraphics[width=\unitlength,page=1]{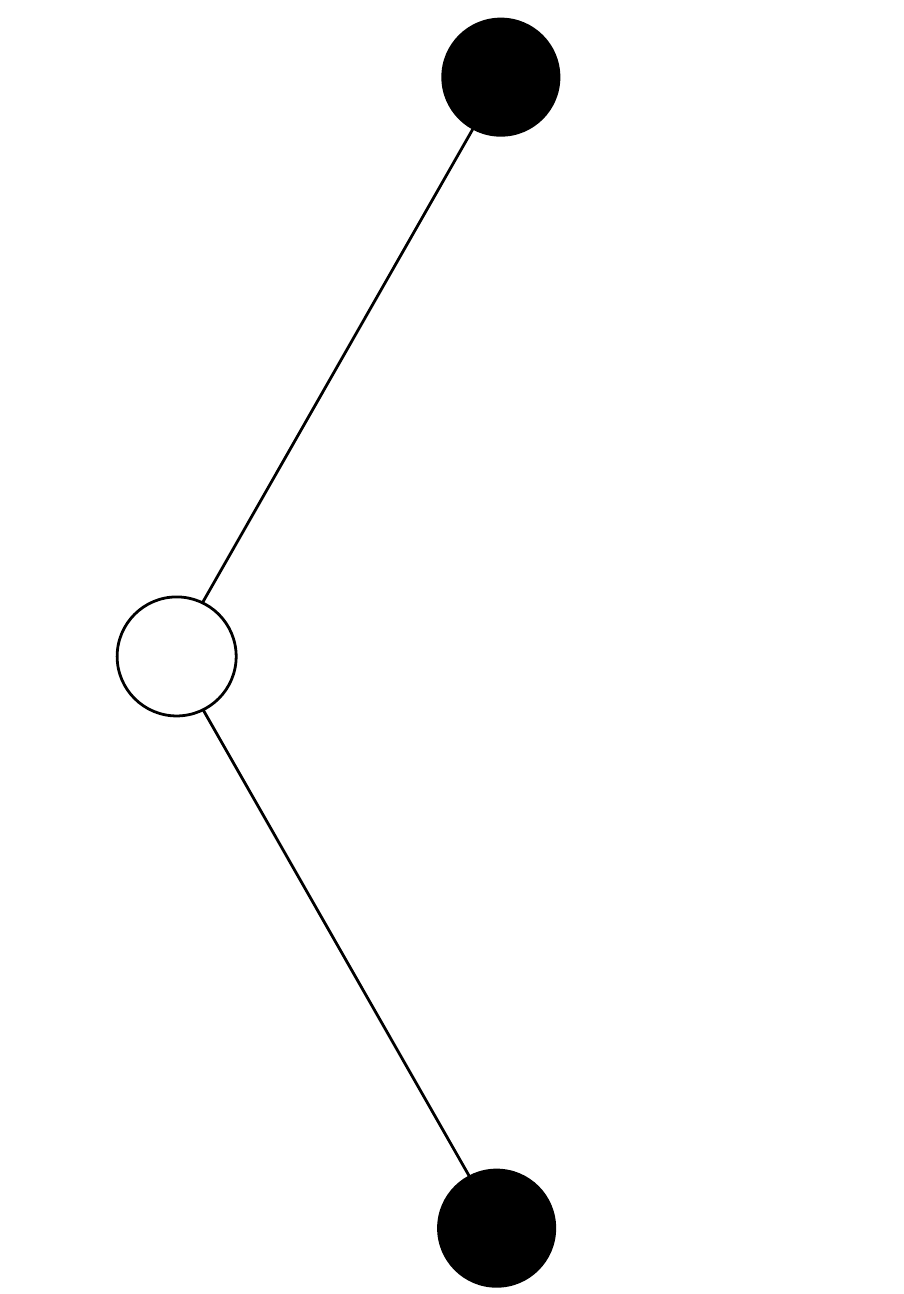}}%
  \end{picture}%
\endgroup%
}} & \parbox[c]{8em}{\scalebox{.15}{
\begingroup%
  \makeatletter%
  \providecommand\color[2][]{%
    \errmessage{(Inkscape) Color is used for the text in Inkscape, but the package 'color.sty' is not loaded}%
    \renewcommand\color[2][]{}%
  }%
  \providecommand\transparent[1]{%
    \errmessage{(Inkscape) Transparency is used (non-zero) for the text in Inkscape, but the package 'transparent.sty' is not loaded}%
    \renewcommand\transparent[1]{}%
  }%
  \providecommand\rotatebox[2]{#2}%
  \ifx\svgwidth\undefined%
    \setlength{\unitlength}{555bp}%
    \ifx\svgscale\undefined%
      \relax%
    \else%
      \setlength{\unitlength}{\unitlength * \real{\svgscale}}%
    \fi%
  \else%
    \setlength{\unitlength}{\svgwidth}%
  \fi%
  \global\let\svgwidth\undefined%
  \global\let\svgscale\undefined%
  \makeatother%
  \begin{picture}(1,1.11351351)%
    \put(0,0){\includegraphics[width=\unitlength,page=1]{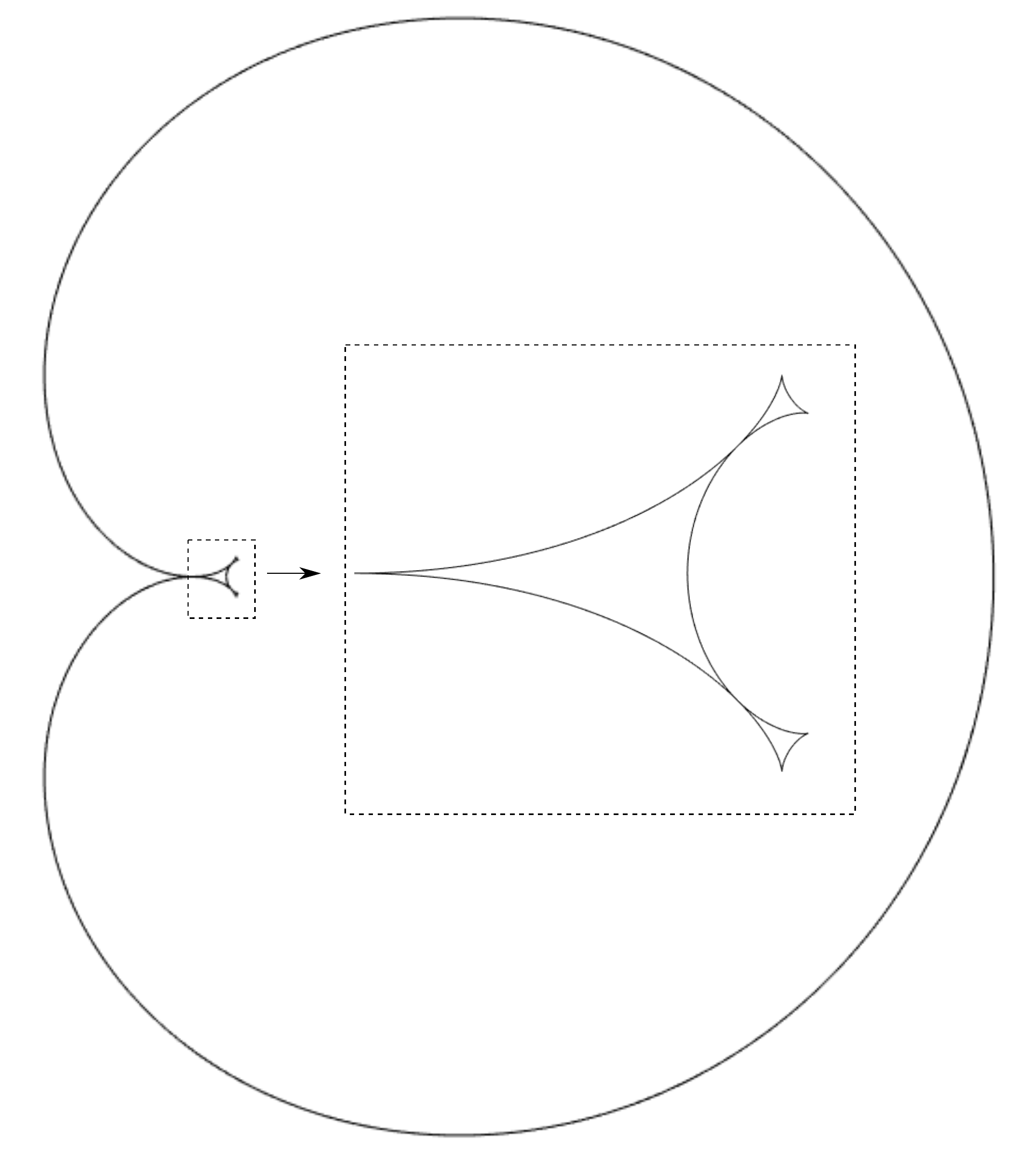}}%
  \end{picture}%
\endgroup%
}} & $ z+\frac{8}{5}\sqrt{\frac{2}{3}}z^2+\frac{6}{5}z^3+\frac{4}{5}\sqrt{\frac{2}{3}}z^4+\frac{z^5}{5} $ \\

& & &  \\

\hline
\end{tabular}

\label{table_2}
\end{center}
\end{adjustwidth}
\captionsetup{singlelinecheck=off}
    \caption[.]{ Pictured are all the Suffridge polynomials (with associated rooted binary trees and droplets) in $S_d^*$ for $d=2$, $3$, $4$, and one Suffridge polynomial in $S_5^*$. The rooted vertex is marked white, whereas non-rooted vertices are marked black. Here
    \begin{displaymath} \hspace{-10mm}  A=(1/2)\sqrt{3\left(\sqrt{15}-3\right)} \textrm{, } t=(1/3)\cos^{-1}\left((3/16)\sqrt{9+5\sqrt{15}}\right).
    \end{displaymath}  For the case $d=3$, see \cite{MR0047198}, \cite{MR227392} and \cite{MR0220919}. For the case $d=4$, see \cite{MR0280704}. The polynomial corresponding to $d=5$ appears in \cite{2014arXiv1411.3415L}.}
\label{table_2}
\end{table}

\begin{definition}\label{isom_def_rooted} Two rooted binary trees $\mathcal{T}$, $\mathcal{T}'$ are said to be \emph{isomorphic} if they are isomorphic as trees via a planar homeomorphism which preserves the root, and left/right children.
\end{definition}

\begin{rem}\label{rooted_binary_isom_rem} One usually makes no mention of isomorphism for rooted binary trees and simply identifies isomorphic rooted binary trees. With this convention, for instance, there is precisely one rooted binary tree with one vertex, and precisely two rooted binary trees with two vertices (see Table \ref{table_2}). This is the convention used in the statement of Theorem \ref{theorem_B}. It will nevertheless be useful for us to have Definition \ref{isom_def_rooted} for the statements and proofs of Theorems \ref{mainthm_qd_terminology_bounded}, \ref{rigidity_extremal_bounded_qd_thm}.  \end{rem}

\subsection{Dynamics of Schwarz Reflections Arising from $S_d^*$}\label{dyn_S_d_schwarz_subsec} A straightforward adaptation of Proposition~\ref{s.c.q.d.} implies that for $f\in S_d^*$, the domain $\Omega=f(\D)$ is a bounded quadrature domain with associated Schwarz reflection map \[\sigma=f\circ\eta\circ\left(f\vert_{\D}\right)^{-1}.\] The map $\sigma$ has a $d$-fold pole at the origin, and no other critical point in $\Omega$. Moreover, $\sigma:\sigma^{-1}(\Omega)\to\Omega$ is a covering map of degree $d-1$, and $\sigma:\sigma^{-1}(\interior{\Omega^c})\to\interior{\Omega^c}$ is a degree $d$ proper branched covering map (branched only at $0$).

As in Subsection~\ref{Dynamical_Partition}, we define $T=\widehat{\mathbb{C}}\setminus \Omega$, $T^0=T\setminus\{$The singular points on $\partial T\}$, and $$T^\infty(\sigma)=\bigcup_{n\geq0} \sigma^{-n}(T^0).$$ We will call $T^\infty(\sigma)$ the \emph{tiling set} of $\sigma$.

As in the case for $\Sigma_d^*$, the cusps (respectively, the double points) on $\partial\Omega$ are of the type $(3,2)$ (respectively, are intersection points of two distinct non-singular branches of $\partial\Omega$ with a contact of order $1$).

\noindent The proof of the next proposition is completely analogous to that of Proposition~\ref{basin_topology}.

\begin{prop}\label{tiling_topology_S_d}
The tiling set $T^\infty(\sigma)$ is open. Its closure $\overline{T^\infty(\sigma)}$ is a compact, connected set.
\end{prop}

\noindent The following result states that the family of Schwarz reflections arising from $S_d^*$ is quasiconformally closed.

\begin{prop}\label{qc_def_S_d_prop}
Let $g\in S_d^*$, $\Omega:=g(\D)$, and $\sigma$ the Schwarz reflection map of $\Omega$. Further, let $\mu$ be a $\sigma$-invariant Beltrami coefficient on $\widehat{\C}$, and $\mathbf{\Phi}:(\widehat{\mathbb{C}},0)\rightarrow(\widehat{\mathbb{C}},0)$ be a quasiconformal map satisfying $\mathbf{\Phi}_{\overline{z}}/\mathbf{\Phi}_{z}=\mu$ a.e.. Then $\mathbf{\Phi}(\Omega)$ is a simply connected bounded quadrature domain. If the quasiconformal map $\mathbf{\Phi}$ is normalized appropriately, there exists $h\in S_d^*$ with $\mathbf{\Phi}(\Omega)=h(\D)$, and $\mathbf{\Phi}\circ\sigma\circ\mathbf{\Phi}^{-1}$ is the Schwarz reflection map of $\mathbf{\Phi}(\Omega)$.
\end{prop}
\begin{proof}
Arguing as in the proof of Proposition~\ref{qc_def_prop}, one first sees that $\widehat{\Omega}:=\mathbf{\Phi}(\Omega)$ is a bounded quadrature domain with Schwarz reflection map $\widehat{\sigma}:=\mathbf{\Phi} \circ \sigma \circ \mathbf{\Phi}^{-1} : \overline{\widehat{\Omega}} \rightarrow \widehat{\mathbb{C}}$.

Since $g\in S_d^*$, we know that $\sigma$ has a pole of order $d$ at the origin. Normalizing $\mathbf{\Phi}$ so that it fixes $\infty$, we conclude that $\widehat{\sigma}$ also has a pole of order $d$ at the origin. Moreover, $\widehat{\sigma}:\widehat{\sigma}^{-1}(\widehat{\Omega})\to\widehat{\Omega}$ is a covering map of degree $d-1$. Hence, an analogue of Proposition~\ref{s.c.q.d.} for bounded quadrature domains provides us with a rational map $h$ of degree $d$ such that $h(\D)=\widehat{\Omega}$ and $h$ is conformal on $\D$. We may normalize $h$ so that $h(0)=0$. Since the Schwarz reflection map $\widehat{\sigma}$ of $\widehat{\Omega}$ has a pole of order $d$ at the origin, it follows that $h$ has a pole of order $d$ at $\infty$. Hence, $h$ must be a polynomial of the form \[h(z)=a_1z+a_{2}z^2+\cdots+a_{d}z^d. \] Note also that since $\partial\Omega$ has $d-1$ cusps, $h$ must have $d-1$ critical points on $\mathbb{T}$. The rest of the proof is similar to that of Proposition~\ref{qc_def_prop}. More precisely, by post-composing $\mathbf{\Phi}$ with multiplication by a non-zero complex number, and conjugating $h$ by a rotation, one concludes that $a_1=1$, and $a_d=1/d$; i.e., $h\in S_d^*$.
\end{proof}

\noindent Lastly, we state a counterpart of Proposition~\ref{dynamical_partition_schwarz} for the family $S_d^*$.

\begin{prop}\label{dynamical_partition_schwarz_S_d} 
$\widehat{\C}=\overline{T^\infty(\sigma)}.$
\end{prop}
\begin{proof}
As in the proof of Proposition~\ref{dynamical_partition_schwarz}, one can use Proposition~\ref{qc_def_S_d_prop} and finite-dimensionality of the family of Schwarz reflections arising from $S_d^*$ to show that every connected component of $\widehat{\C}\setminus\overline{T^\infty(\sigma)}$ is eventually periodic. The fact that the unique critical point of $\sigma$ lies in its tiling set now  implies that $\widehat{\C}\setminus\overline{T^\infty(\sigma)}$ may not have any periodic component as the closure of such a component must intersect the closure of the post-critical set of $\sigma$. Hence, $\widehat{\C}\setminus\overline{T^\infty(\sigma)}=\emptyset$, and the result follows.
\end{proof}

\subsection{Existence in $S_d^*$}\label{existence_S_d_subsec}

In this subsection we prove an existence theorem for $S_d^*$ which is analogous to the Theorem \ref{mainthm_qd_terminology} for the class $\Sigma_d^*$: 

\begin{thm}\label{mainthm_qd_terminology_bounded} 
Let $\mathcal{T}$ be a rooted binary tree. Then there exists an extremal bounded quadrature domain $\Omega$ such that $\mathcal{T}(\Omega)$ is isomorphic to $\mathcal{T}$.
\end{thm}

\begin{proof} Following \cite{MR0235107}, we define, for $d>2$: \[ A_{k,j} := \frac{d-k+1}{d}\frac{\sin kj\pi/(d+1)}{\sin j\pi/(d+1)} \textrm{  } (j, k\in\{1,2,\cdots, d\}), \] and \[ P(z; d,j):=\sum_{k=1}^dA_{k,j}z^k. \] Let $f:=P(z;d,1)$. The following geometric description of $f(\mathbb{T})$ is proven in \cite{MR0235107} (see also Figure \ref{fig:n=5,6}): $f(\mathbb{T})$ has precisely $(d-1)/2$ double points (each real-valued) if $d$ is odd, and $(d-2)/2$ double points (each real-valued) if $d$ is even. Enumerate the double points of $f(\mathbb{T})$ as $x_1,\cdots, x_l$ where $x_1<x_2<\cdots<x_l$, and denote the two preimages of $x_j$ as $\zeta_j^{\pm}$ for $1\leq j\leq l$. The curve $f(\mathbb{T})$ has two $\mathbb{R}$-symmetric cusps $c_j$, $\overline{c_j}$ with $x_j<\textrm{Re}(c_j)<x_{j+1}$ for each $1\leq j\leq l-1$. If $d$ is odd, there are two further cusps $c_{l}$, $\overline{c_l}$ with $\textrm{Re}(c_l)>x_l$. If $d$ is even, there is one more cusp $c_l \in \mathbb{R}$ with $\textrm{Re}(c_l)>x_l$.

\begin{figure}
\centering
\begin{subfigure}{.5\textwidth}
  \centering
  \scalebox{.26}{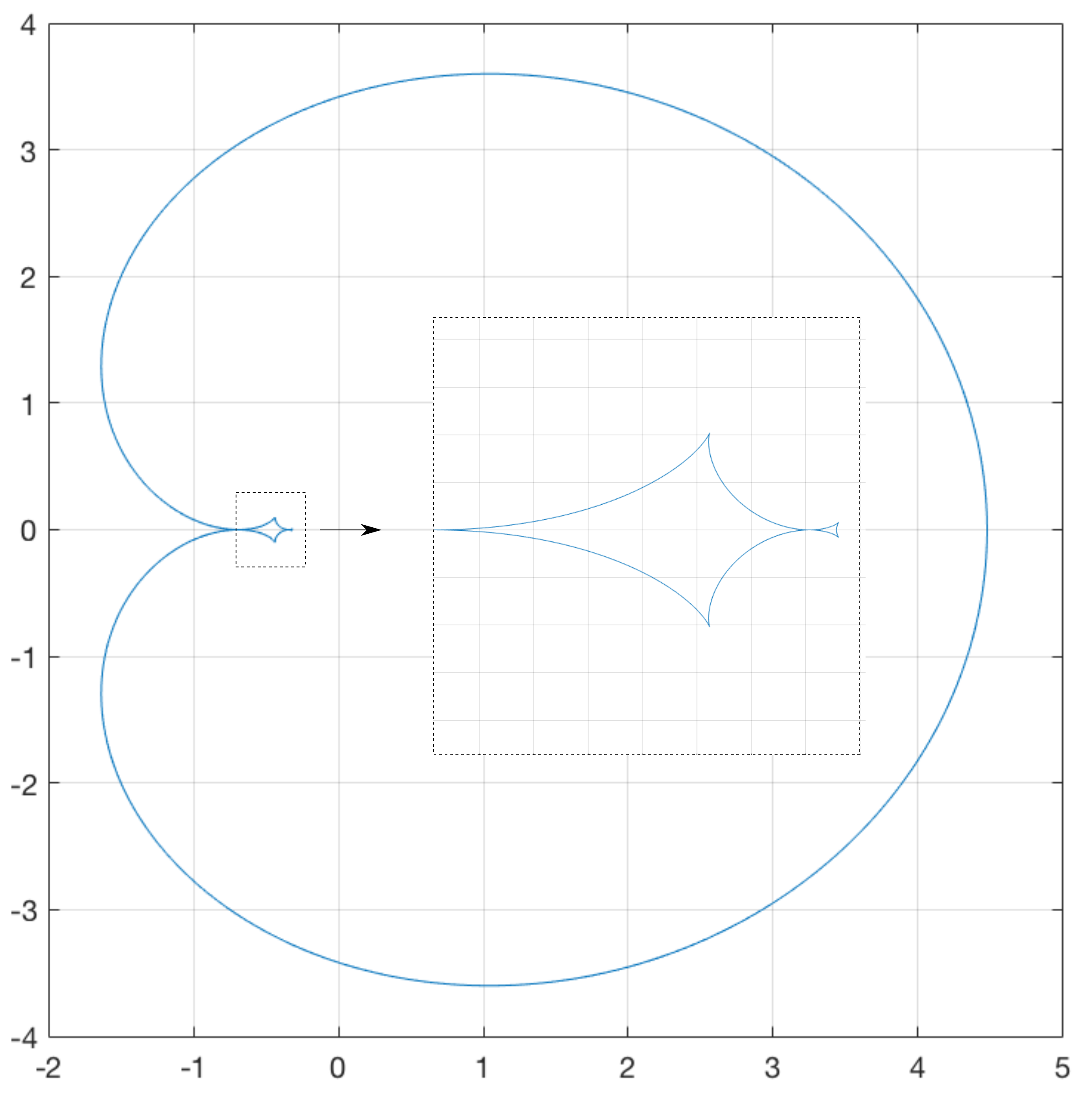}
  \caption{ $P(z; 5,1)$ }
  \label{fig:sub1}
\end{subfigure}
\begin{subfigure}{.5\textwidth}
  \centering
  \scalebox{.26}{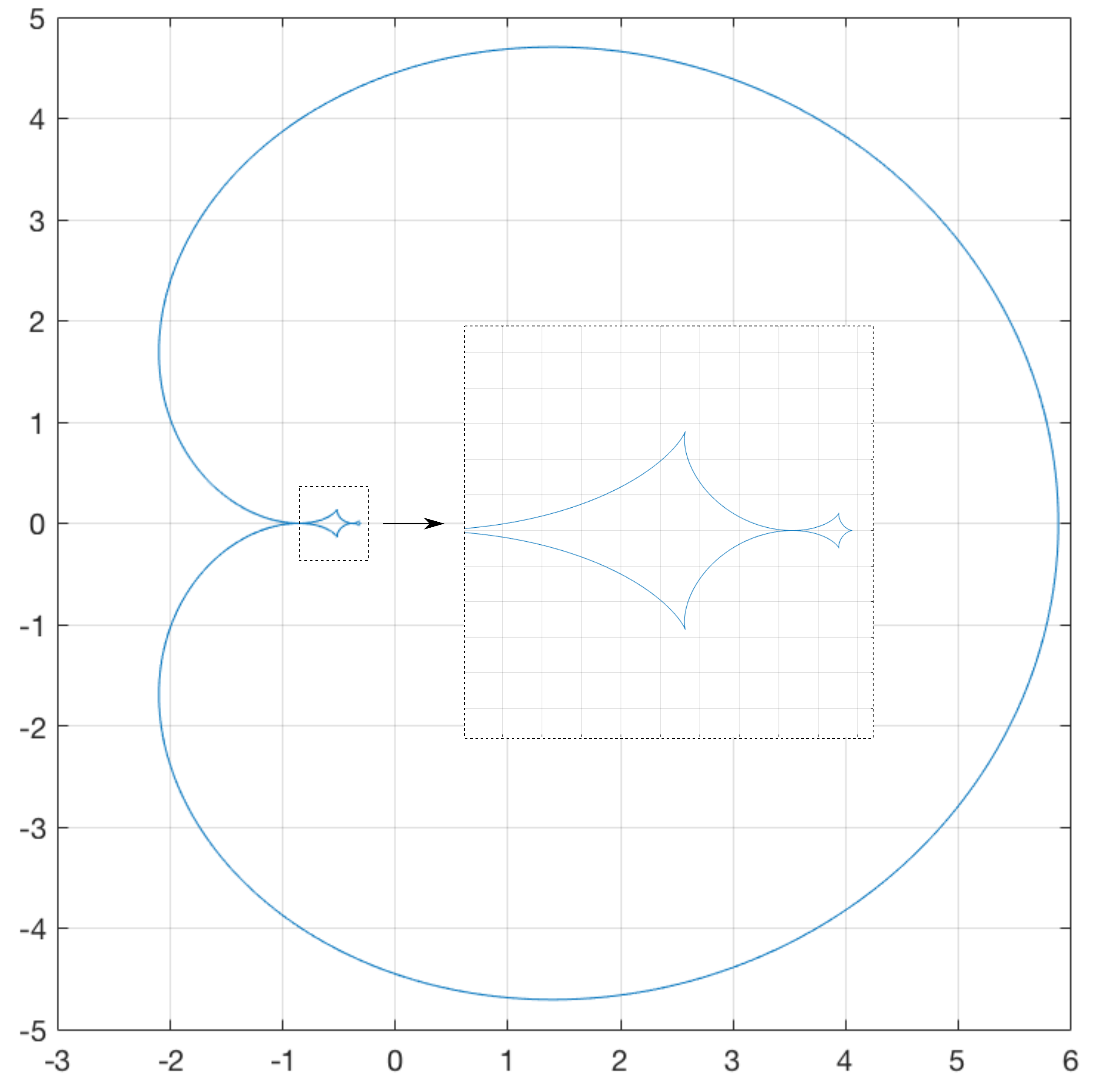}
  \caption{ $P(z; 6,1)$ }
  \label{fig:sub2}
\end{subfigure}
\caption{Pictured is the image of $\mathbb{T}$ under certain polynomials of degrees $5$ and $6$ studied in \cite{MR0235107}. The proof of Theorem \ref{mainthm_qd_terminology_bounded} uses an approach in \cite{2014arXiv1411.3415L} to ``separate'' all but the left-most double point, whence the techniques in the proof of Theorem \ref{mainthm_qd_terminology} apply. }
\label{fig:n=5,6}
\end{figure}

For a positive integer $k$, let $\Pi_k$ be the space of complex polynomials of degree at most $k$. Following \cite{2014arXiv1411.3415L}, we claim that there exists a polynomial \[ r(z) =  \sum_{j=2}^{d-1}a_jz^j \] such that:

\begin{enumerate}

\item $r'(z)$ is self-dual in $\Pi_{d-1}$, i.e., $r'(z)=z^{d-1}\overline{r(1/\overline{z})}$.

\item \[\vspace{3mm}\textrm{Re}\left( \frac{r(\zeta_j^+)-r(\zeta_j^-)}{(\zeta_j^+)^{(d+1)/2}} \right) = 0 \textrm{ for } j\in\{2,\cdots, l\}.  \]

\item $r(\zeta_1^+) = r(\zeta_1^-)$.  

\end{enumerate}

Such a polynomial exists because $\Pi_{d-1}$ has $d-2$ real dimension, whereas the last two conditions give us $l<d-2$  homogeneous linear equations. Then, for sufficiently small $\delta$, $g_\delta(z):=f(z)+\delta r(z)\in S_d^*$ and has a unique double point $g_\delta(\zeta_1^{\pm})$ (see \cite[\S 3.2]{2014arXiv1411.3415L}). All cusps of $g_\delta(\mathbb{T})$ are contained on the boundary of the unique bounded component of $\mathbb{C}\setminus g_\delta(\overline{\mathbb{D}})$. The techniques in the proof of Theorem~\ref{mainthm_qd_terminology} now apply to yield any pattern of allowable maximal pinchings in this bounded droplet.
\end{proof}

\subsection{Rigidity in $S_d^*$}

The goal of this subsection is to prove a rigidity theorem for $S_d^*$ to the effect that extremal bounded quadrature domains are uniquely determined by their rooted binary trees.

\begin{thm}\label{rigidity_extremal_bounded_qd_thm}
Let $\widetilde{\Omega}$ and $\Omega$ be two bounded extremal quadrature domains such that $\mathcal{T}(\widetilde{\Omega})$ and $\mathcal{T}(\Omega)$ are isomorphic as rooted binary trees. Then, there exists an affine map $A$ with $A(\widetilde{\Omega})=\Omega$.
\end{thm}
\begin{proof}
The proof is an adaptation of Theorem~\ref{rigidity_extremal_qd_thm}, so we only outline the necessary modifications.

The fact that $\mathcal{T}(\widetilde{\Omega})$ and $\mathcal{T}(\Omega)$ are isomorphic as rooted binary trees implies that the arguments in the proof of Lemma~\ref{angled_tree_conformal_map} can be used to define a conformal map $\mathbf{\Psi}:\widetilde{T}\to T$ such that

\begin{itemize}
\item $\mathbf{\Psi}$ maps the unbounded connected component of $\interior{\widetilde{T}}$ onto that of $\interior{T}$,

\item $\mathbf{\Psi}(\infty)=\infty$, 

\item $\mathbf{\Psi}$ sends the cusps (respectively, double points) on $\partial\widetilde{T}$ to the cusps (respectively, double points) on $\partial T$.
\end{itemize}

Since the cusps (respectively, the double points) on the boundaries of $\widetilde{T}$ and $T$ are of the type $(3,2)$ (respectively, are intersection points of two distinct non-singular branches with a contact of order $1$), we can argue as in Lemmas~\ref{asymp_linear} and~\ref{global_qc} to conclude that $\mathbf{\Psi}$ admits a quasiconformal extension to $\widehat{\C}$. We now lift the quasiconformal map $\mathbf{\Psi}$ by iterates of the Schwarz reflection maps $\widetilde{\sigma}$ and $\sigma$ to obtain a sequence of quasiconformal maps with uniformly bounded dilatation (the lifting is possible because $\mathbf{\Psi}$ maps the unique critical value $\infty$ of $\widetilde{\sigma}$ to the unique critical value $\infty$ of $\sigma$). Finally, we pass to a subsequential limit of these quasiconformal maps which conformally conjugates $\widetilde{\sigma}$ to $\sigma$ on the tiling set $T^\infty(\widetilde{\sigma})$. Since $\widehat{\C}=\overline{T^\infty(\widetilde{\sigma})}$ (see Proposition~\ref{dynamical_partition_schwarz_S_d}), this limiting map is a global quasiconformal conjugacy between $\widetilde{\sigma}$ and $\sigma$ that is conformal away from the boundary of the tiling set $T^\infty(\widetilde{\sigma})$. The arguments of Proposition~\ref{limit_schwarz_zero_area} apply verbatim to the current setting to show that $\partial T^\infty(\widetilde{\sigma})$ has zero area. Therefore, the above limiting quasiconformal homeomorphism of the sphere is conformal outside a set of zero area, and hence must be a M{\"o}bius map. The theorem now follows.
\end{proof}

\subsection{Counting Suffridge Polynomials}

\noindent Lastly, we consider the problem of counting Suffridge polynomials of degree $d$.

\begin{thm}\label{counting_extremal_bounded_qds} Let $d\geq2$. Then \begin{equation}\label{counting_formula} \# \bigg( \big\{ f \in S_d^* : f \emph{ has } d- 2 \emph{ double points} \big\}\big/ \hspace{1mm} \mathbb{Z}_{d-1} \bigg) = \frac{1}{d-1}{2(d-2) \choose d-2}.   \end{equation}\end{thm}

\begin{proof} 

By an analogue of Proposition \ref{suffridge_extremal_qd_equiv_prop} for $S_d^*$, the left-hand side of (\ref{counting_formula}) is the number of bounded extremal quadrature domains (of order $d$) up to affine equivalence. It then follows from Theorems \ref{mainthm_qd_terminology_bounded} and \ref{rigidity_extremal_bounded_qd_thm} that the left-hand side of (\ref{counting_formula}) is the number of rooted binary trees with $d-2$ vertices, and this is given by the sequence of \emph{Catalan numbers} $C_{d-2}$ (given by the right-hand side of (\ref{counting_formula})).
\end{proof}

\bibliographystyle{alpha}
\bibliography{bibfile_qd}

\end{document}